\documentclass[a4paper, 11pt]{article}

\usepackage[english]{babel}
\usepackage[utf8x]{inputenc}
\usepackage{ulem}
\usepackage[T1]{fontenc}

\usepackage{amsmath,amsfonts,amssymb,amsthm,bbm,eufrak,mathrsfs,booktabs,color,epsfig,graphicx,hyperref,url,subcaption,algorithm,enumitem,xcolor,natbib,subfiles}
\usepackage[export]{adjustbox}

\usepackage[top=3cm, bottom=3cm, left=3cm, right=2.5cm]{geometry}

\newtheorem{theorem}{Theorem}
\newtheorem{proposition}{Proposition}
\newtheorem{lemma}{Lemma}

\newtheorem{assumption}{Assumption}
\newtheorem{definition}{Definition}
\newtheorem{remark}{Remark}
\newtheorem{example}{Example}

\DeclareMathOperator{\argmax}{argmax}

\DeclareMathOperator{\ind}{\mathbbm{1}}

\DeclareMathOperator{\esp}{\mathbb{E}}

\DeclareMathOperator{\FMR}{FMR}
\DeclareMathOperator{\mFMR}{mFMR}
\DeclareMathOperator{\proc}{\mathcal{C}}
\DeclareMathOperator{\Pstar}{\mathbb{P}_{\theta^*}}
\DeclareMathOperator{\VCdim}{\mathscr{V}}
\DeclareMathOperator{\Radm}{\mathfrak{R}}

\renewcommand{\P}{\mathbb{P}}
\newcommand{\subsetEsti}{\mathcal{D}}

\renewcommand*{\thefootnote}{\fnsymbol{footnote}}

\begin{document}

\begin{center}
{\LARGE
	{  False membership rate control in mixture models}
}

\bigskip

Ariane Marandon$^{1}\footnotemark{}$, Tabea Rebafka$^{1}$, Etienne Roquain$^{1}$, Nataliya Sokolovska$^{2}$
\bigskip

\end{center}

{\small
{
$^{1}$ LPSM, Sorbonne Universit\'e, Universit\'e de Paris \& CNRS, 4, Place Jussieu, 75005, Paris, France

$^{2}$ LCQB, Sorbonne Universit\'e, CNRS, 4, Place Jussieu, 75005, Paris, France

}
}

\bigskip

\bigskip

\footnotetext{Corresponding author: ariane.marandon-carlhian@sorbonne-universite.fr.}

\begin{abstract}

The clustering task consists in partitioning elements of a sample into homogeneous groups.
Most datasets contain individuals that are ambiguous and  intrinsically difficult to attribute to one or another cluster.   
However, in practical applications, misclassifying individuals is potentially disastrous and should be avoided. 
To keep the misclassification rate small, one can decide to   classify only {\it a part} of the sample. 
 In the supervised setting, this approach is well known and  referred to as classification with an abstention option. 
In this paper the approach is revisited in an unsupervised mixture-model framework 
and the purpose  is to develop a method that comes with the guarantee that the false membership rate (FMR) does not exceed a predefined nominal level $\alpha$.
A plug-in procedure is proposed, for which a theoretical analysis is provided, by quantifying the FMR deviation with respect to the target level $\alpha$ with explicit remainder terms. 
Bootstrap versions of the procedure are shown to improve the performance in numerical experiments.

\end{abstract}

{\bf Keywords.} Classification with abstention, clustering, False discovery rate, mixture models.

\renewcommand*{\thefootnote}{\arabic{footnote}}
\setcounter{footnote}{0}

\section{Introduction}

\subsection{Background}

Clustering is a standard statistical task that aims at grouping together individuals with similar features.
However, it is common that data sets include ambiguous individuals that are inherently difficult to classify, which makes the clustering result   potentially unreliable. 
To illustrate this point, consider a Gaussian mixture model
 with overlapping mixture components. Then it is difficult, or even impossible, to assign the correct cluster label to data points that fall in the overlap of those clusters, see Figure~\ref{FMR:fig:intro}. Hence, when
 the overlap is large (Figure~\ref{FMR:fig:intro} panel (b)), the  misclassification rate of a standard clustering method  is inevitably elevated.
 
This issue is critical in applications where misclassifications come with a high cost for the user and should be  avoided. This is for example the case for medical diagnosis, where an error can have severe consequences on the individual's health. When there is too much uncertainty, a solution is to avoid classification for such individuals, and to adopt a wiser ``abstention decision'', that leaves the door open for further medical exams. 

In a supervised setting, classification with a reject (or abstention) option is a long-standing statistical paradigm, that can be traced back to \cite{chow70}, with more recent works including \cite{herbei06,bartlett08,wegkamp2011support}, among others.  
In this line of research, rejection is accounted for by adding a term to the risk that penalizes any rejection (i.e., non classification).


	Recently,  still in the supervised setting, \cite{geifman17} and \cite{angelopoulos2021learn} have considered the problem of having a prescribed control of the classification error among the classified items (those that are not rejected). In these works the proposed method consists of thresholding the estimated class probabilities estimated by a pre-trained classifier, in a data-driven manner. Both of these works provide the guarantee that the resulting selective classifier has its true risk bounded by a prescribed level with high probability. 

\subsection{Aim and approach}

The goal of the present work is to propose a labelling guarantee on the classified items 
 in the  more challenging unsupervised setting, where no training set is  available and data are assumed to be generated from a finite mixture model. 
 This is achieved by the possibility to refuse to cluster ambiguous individuals and by using the false membership rate (FMR), which is defined as the average proportion of misclassifications among the classified objects. Our procedures are devised to keep the FMR below some nominal level $\alpha$, while  classifying a maximum number of items.

It is important to understand the role of the  nominal level $\alpha$ in our approach. 
It is chosen by the user and depends on their acceptance or tolerance for misclassified objects. 
Since the FMR is the misclassification risk that is allowed on the classified objects, the final interpretation of an FMR control at level $\alpha$ is clear: if, for instance, $\alpha$ is set to  $5\%$ and $100$ items are finally chosen to be classified by the method, then the number of misclassified items is expected to be at most $5$. This high interpretability is similar to the one of the false discovery rate (FDR) in multiple testing, which has known a great success in applications since its introduction by \cite{BH1995}. 
This is a clear advantage of our approach  for practical use compared to the methods with a rejection option that are based on a penalized risk.

In our framework, a procedure is composed of two intertwined decisions: 
\begin{itemize}
\item a clustering method inferring the labels; 
\item a selection rule deciding which items to label.
\end{itemize}
Importantly, the selection rule is only applied \textit{after} a clustering method is fitted on the (entire) sample. In other words, the procedure consists of two subsequent steps: a clustering step, after which cluster labels are kept fixed, and a selection step, that chooses which items to classify -- in which case, the label from the previous clustering step is assigned. For the items that are not selected, we discard the cluster label, that is, we effectively abstain to make a classification decision for those items. 
In particular, we emphasize that the clustering method is not fitted again after selection (which would lead to bias in general).

The quality of the selection heavily relies on the appropriate quantification of the uncertainty of the cluster labels. For this, our approach is model-based, and can be viewed as a method that  thresholds the posterior probabilities of the cluster labels with a data-driven choice of the threshold.
The performance of the method will depend on the quality of the estimates of these posterior probabilities in the mixture model.

The adaptive character of our method is illustrated in Figure~\ref{FMR:fig:intro}: when the clusters are well separated (panel (a)), the new procedure only discards few items and provides  a clustering close to the correct one. However, when clusters are overlapping (panel (b)), to avoid a high misclassification error,  the procedure discards most of the items and only provides few labels, for which the uncertainty is low. 
In both cases, the proportion of misclassified items among the selected ones is small and in particular close to the target level $\alpha$ (here $10\%$).  Hence, by adapting the amount of labeled or discarded items, our method always delivers a reliable clustering result, inspite of the varying  intrinsic difficulty of the clustering task.

\begin{figure}
\begin{minipage}{0.495\linewidth}
	\centering
	\includegraphics[width=0.5\linewidth]{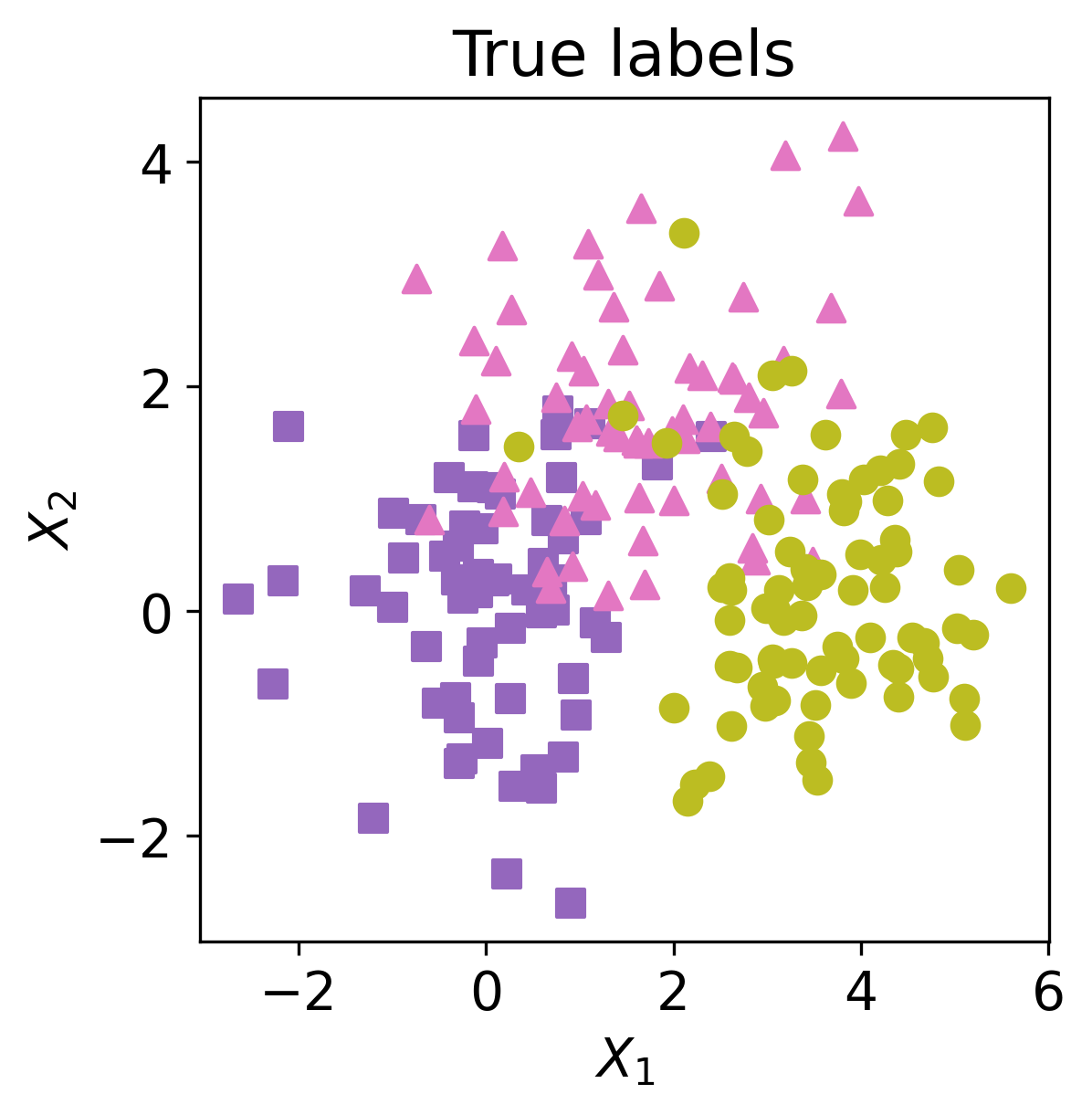}%
	\includegraphics[width=0.5\linewidth]{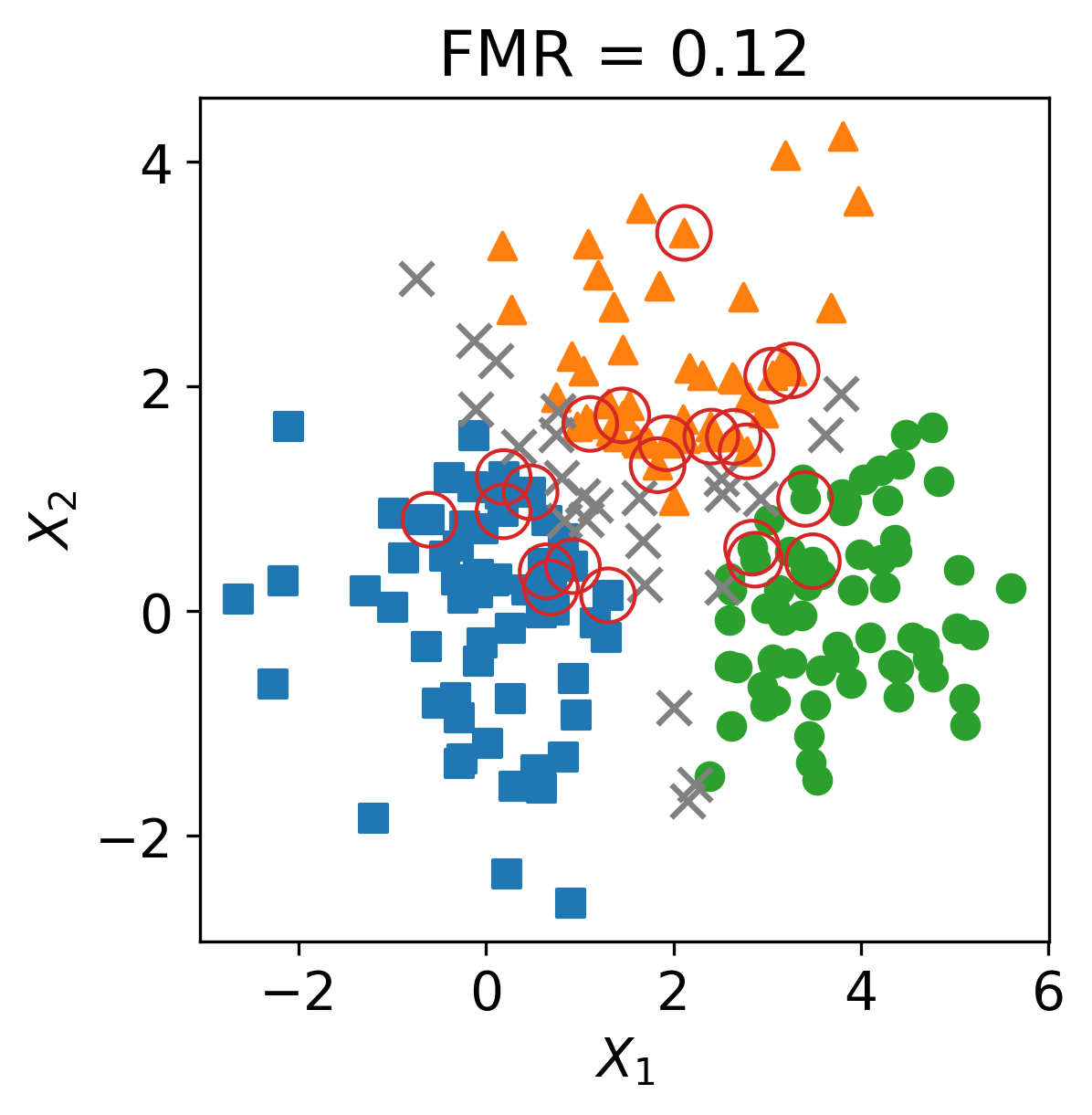}
	\subcaption{Separated clusters}
\end{minipage}
\begin{minipage}{0.495\linewidth}
	\centering
	\includegraphics[width=0.5\linewidth]{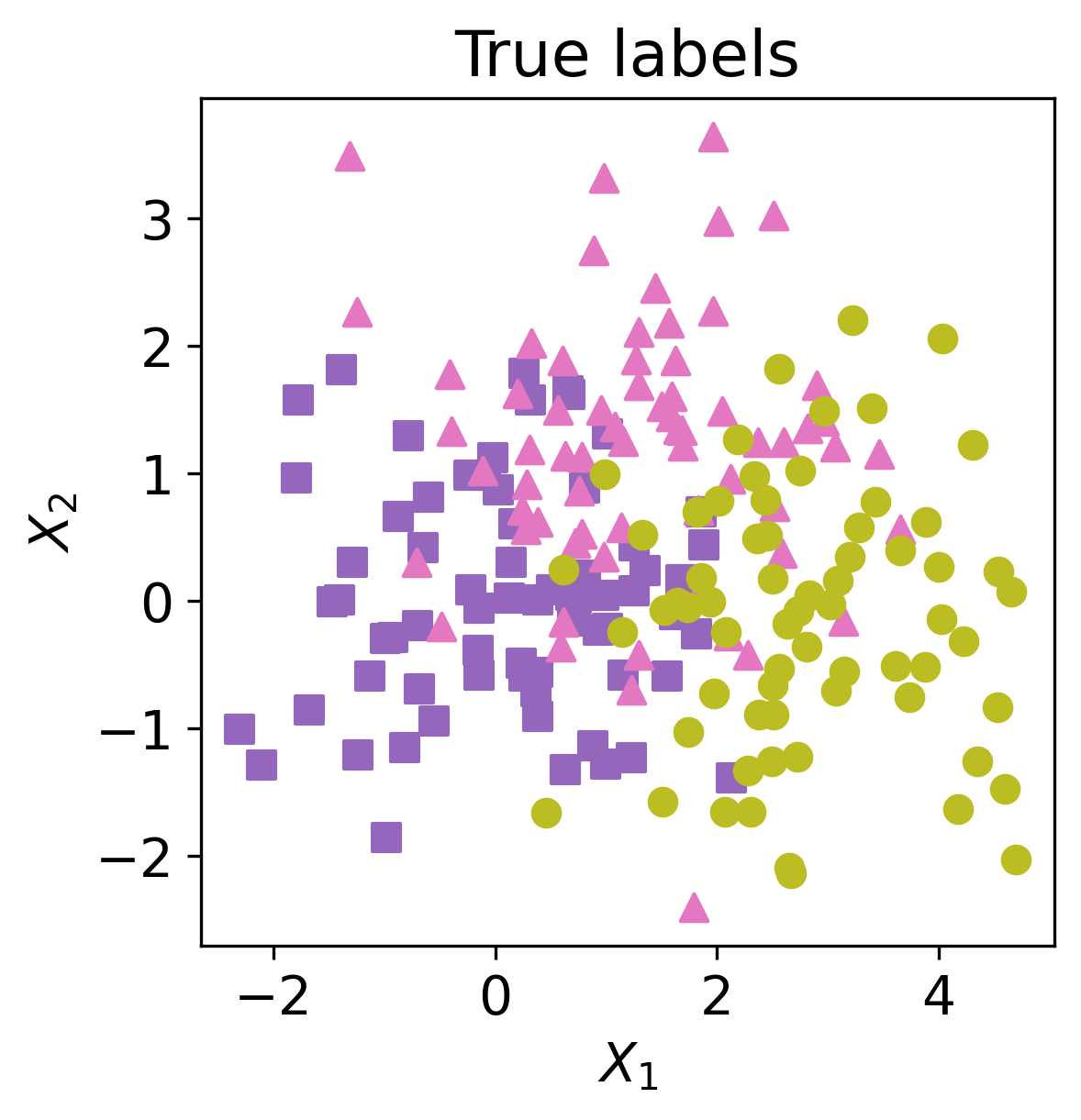}%
	\includegraphics[width=0.5\linewidth]{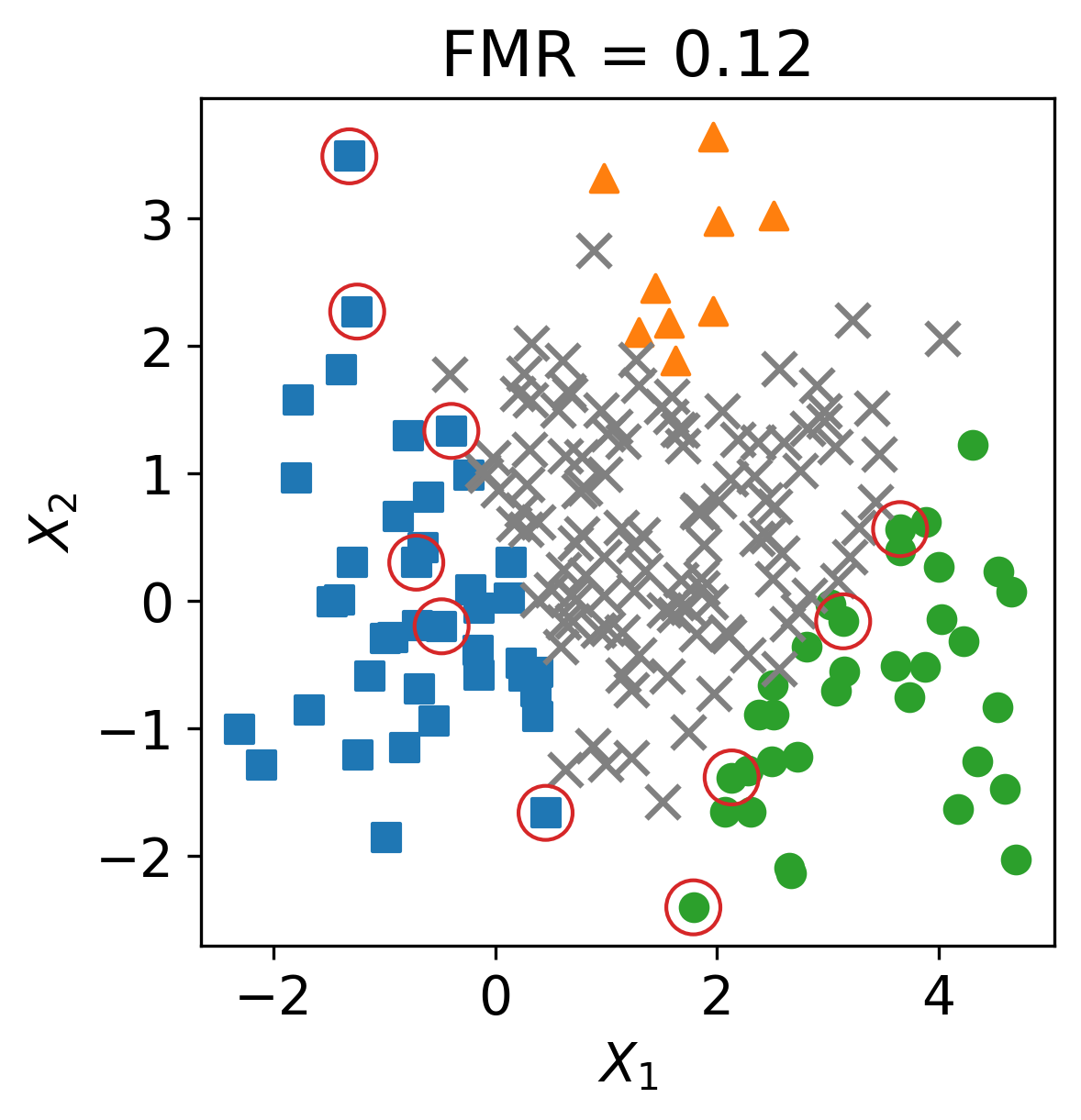}
	\subcaption{Ambiguous clusters}
\end{minipage}
\caption{Data from Gaussian mixtures with three components ($n=200$), in a fairly separated case (panel (a)) and an ambiguous case (panel (b)). In each panel, the left part displays the true clustering, while the right part illustrates the new procedure (plug-in procedure at level $\alpha=10\%$), that does not cluster all items. The points not classified are depicted by grey crosses. Red circles indicate erroneous labels. } \label{FMR:fig:intro} 
\end{figure}

\subsection{Presentation of the results}\label{FMR:sec:contrib}

Let us now describe in more details the main contributions of the paper.

\begin{itemize}

\item We introduce three new data-driven procedures that perform simultaneously selection and clustering: the plug-in procedure (illustrated in Figure~\ref{FMR:fig:intro}) and two bootstrap procedures (parametric and non-parametric), see Section~\ref{FMR:sec:empprocedure}. 

\item We provide a theoretical analysis of the plug-in procedure, quantifying the FMR deviation with respect to the target level $\alpha$ with explicit remainder terms, which become small when the sample size grows. In addition,  this procedure is shown to satisfy the following optimality property: any other procedure that provides an FMR control necessarily classifies as many or less items than the plug-in procedure, up to a small remainder term (Theorem~\ref{FMR:th:convergence_rates}).

\item Numerical experiments\footnotemark{} establish that the bootstrap procedures improve the plug-in procedure, and thus are more reliable for practical use, where the sample size may be moderate, see Section~\ref{FMR:sec:numexp}. In particular, the FMR control is shown to be valid in various scenarios, including those where the overall misclassification risk (with no abstention option) is too large.

\footnotetext{
We publicly release the code of these experiments at \url{https://github.com/arianemarandon/fmrcontrol}. We have also included a Jupyter notebook that demonstrates the use of our procedures. 
}

\item Our analysis also shows that a fixed threshold procedure that only labels items with a maximum posterior probability larger than $1-\alpha$ is generally suboptimal for an FMR control at level $\alpha$, see Section~\ref{FMR:sec:numexp}. To this extent, our procedures can be seen as refined algorithms that classify more individuals while   maintaining the FMR control.

\item The practical impact of our approach is demonstrated on a real data set, see Section~\ref{FMR:sec:realdata}. 

\end{itemize}

\subsection{Relation to previous work}

\paragraph{Other clustering guarantees in unsupervised learning}

While we provide a specific FMR control guarantee on the clustering, other criteria, not particularly linked to a rejection option, have been previously proposed in an unsupervised setting. Previous works provided  essentially two types of guarantees:
while early works focused on the probability of exact recovery \citep{arora05, vempala04, abbe18}, recent contributions rather considered minimizing the misclassification risk \citep{LR2015, lu2016statistical, giraud18, chretien19}. 
Other criteria include the probability to make a  different decision  than the Bayes rule \citep{azizyan13}, or the fact that all clusters are mostly homogeneous with high probability \citep{najafi20}.  
All these works provide a guarantee only if the setting is favorable enough. By contrast, providing a rejection option is the key to obtain a guarantee in any setting (in the worst situation, the procedure will not classify any item).


\paragraph{Comparison to \cite{denis20} and \cite{maryhuard21}} 
We describe here two recent studies that are related to ours, because they also use a FMR-like criterion. The first one is the work of \cite{denis20}, which also relies on a thresholding of the (estimated) posterior probabilities. However, the control is different, because it does not provide an FMR control, but rather a type-II error control concerning the probability of classifying an item. Also, the proposed procedure therein requires an additional labeled  sample (semi-supervised setting), which is not needed in our context.

The work of \cite{maryhuard21} also proposes a control of the FMR.
However, the analysis therein is solely based on the case where the model parameters are known (thus corresponding to the oracle case developed in Section~\ref{FMR:sec:oracle} here). Compared to \cite{maryhuard21}, the present work provides number of new contributions, which are all given in Section~\ref{FMR:sec:contrib}.
Let us also emphasize that we handle the label switching problem in the FMR, which seems to be  overlooked in \cite{maryhuard21}.

\paragraph{Relation to the false discovery rate}

The FMR is closely related to the false discovery rate (FDR) in multiple testing, defined as the average proportion of errors among the discoveries. In fact, we can roughly view the problem of designing an abstention rule as testing, for each item $i$, whether the clustering rule correctly classifies item $i$ or not. With this analogy, our selection rule is based on quantities similar to the local FDR values \citep{ETST2001}, a key quantity to build optimal FDR controlling procedures in multiple testing mixture models, see, e.g., \cite{storey2003positive,SC2007,CSWW2019,rebafka22}.
In particular, our final selection procedure shares similarities with the procedure introduced in \cite{SC2007}, also named cumulative $\ell$-value procedure \citep{abraham2022empirical}. In addition, our theoretical analysis is related to the work of \cite{rebafka22}, although the nature of the algorithm developed therein is different from here: they use the $q$-value procedure of \cite{storey2003positive}, while our method rather relies on the cumulative $\ell$-value procedure.

\subsection{Organization of the paper}

The paper is organized as follows: Section~\ref{FMR:sec:setting} introduces the 
model and relevant notation, namely 
the FMR criterion, with a particular care for the label switching problem. Section~\ref{FMR:sec:methods} presents the new methods: the oracle, plug-in and the bootstrap approaches. Our main theoretical results are provided in Section~\ref{FMR:sec:theory}, after introducing appropriate assumptions.
Section~\ref{FMR:sec:xp} presents numerical experiments and an application to a real data set, while a conclusion is given in Section~\ref{FMR:sec:discussion}. Proofs of the results and technical details are deferred to appendices.

\section{Setting}\label{FMR:sec:setting}

This section presents the notation, model, procedures and criteria that will be used throughout the manuscript.

\subsection{Model}\label{FMR:sec:model}

Let $\mathbf X  = (X_1,\dots,X_n)$ be an observed random sample of size $n$. Each $X_i$ is an i.i.d. copy of a $d$-dimensional real random vector, which is assumed to follow the standard mixture model: 
\begin{align*}
Z &\sim \mathcal{M}(\pi_1, \dots, \pi_Q), \\
X | Z = q &\sim F_{\phi_q}, \quad 1 \leq q \leq Q,
\end{align*}
where $\mathcal{M}(\pi_1, \dots, \pi_Q)$ denotes the multinomial distribution of parameter $\pi$  (equivalently, $\pi_q= \P(Z=q)$ for each $q$). The model parameters are given by
\begin{itemize}
\item
 the probability distribution $\pi$ on $\{1,\dots,Q\}$ that is assumed to satisfy $\pi_q>0$ for all $q$. Hence, $\pi_q$ corresponds to the probability of being in class $q$;
\item
the parameter ${\mathbf \phi} = (\phi_1, \dots, \phi_Q)\in  \mathcal{U}^Q$, where $\{ F_{u}, u \in \mathcal{U} \}$ is a collection of distributions on $\mathbb{R}^d$. Every distribution $F_{u}$ is assumed to have a density with respect to the Lebesgue measure on $\mathbb{R}^d$, denoted by $f_u$. Moreover, we assume that the $\phi_q$'s are all distinct. 
\end{itemize}
The number of classes $Q$ is assumed to be known and fixed throughout the manuscript (see Section~\ref{FMR:sec:discussion} for a discussion). Thus, the overall parameter is $\theta = (\pi, \phi)$, the parameter  set is denoted by $\Theta$, and the distribution of $(Z,X)$ is denoted by $P_\theta$. The distribution family  $\{ P_\theta, \theta \in \Theta \}$ is the considered statistical model.
We also assume that $\Theta$ is an open subset of $\mathbb{R}^K$ for some $K\geq 1$ with the corresponding topology. 

In this mixture model, the latent  vector $\mathbf Z = (Z_1, \dots, Z_n)$  encodes a partition  of the $n$ observations into  $Q$ classes  given by $\{ 1\leq i\leq n\::\:Z_i=q\}$, $1\leq q\leq Q$. We  refer to this model-based, random partition as the true latent clustering in the sequel. 

In what follows, the ``true'' parameter that generates $(Z,X)$ is assumed to be fixed and is denoted by $\theta^*\in \Theta$. 


\subsection{Procedure and criteria}

Our approach starts with a given clustering rule, that aims at recovering the true latent clustering for all observed items.
In general, a clustering rule is defined as a (measurable) function of the observation $\mathbf{X}$ returning a vector of labels $\widehat{\mathbf{Z}}=(\hat{Z}_i)_{1\leq i\leq n}\in \{1,\dots,Q\}^n$ for which the label $q$ is assigned to individual $i$ if and only if $\hat{Z}_i=q$. Note that in the unsupervised setting only the partition of the observations is of interest, not the labels themselves. Switching the labels of $\widehat{\mathbf{Z}}$ does not change the corresponding partition. 



The classification error of $\widehat{\mathbf{Z}}$, with respect to specific labels, is given by $\varepsilon(\widehat{\mathbf{Z}} ,\mathbf{Z}) =  \sum_{i=1}^n \ind \{Z_i \neq \hat{Z}_i \}$.  
A label-switching invariant error is the clustering risk of $\widehat{\mathbf{Z}}$ defined by
\begin{align}\label{clusteringrisk} 
R(\widehat{\mathbf{Z}}) =  \esp_{\theta^*} \left( \underset{\sigma \in [Q]}{\min} \esp_{\theta^*} \left(n^{-1} \varepsilon( \sigma(\widehat{\mathbf{Z}}) ,\mathbf{Z}) \:|\: \mathbf X \right)\right), 
\end{align}
where 
$[Q]$ denotes the set of all permutations on $\{1,\dots,Q\}$. The minimum over all permutations $\sigma$ is the way to handle the aforementioned label-switching problem. 

\begin{remark}
The position of the minimum w.r.t. $\sigma$ in the risk \eqref{clusteringrisk} matters: the permutation $\sigma$ is allowed to depend on $X$ but not on $Z$. Hence, this risk has to be understood as being computed up to a data-dependent label switching.  This definition coincides with the usual definition of the misclassification risk in the situation where the true clustering is deterministic, see \cite{LR2015,lu2016statistical}. Hence, it can be seen as a natural extension of the latter to a mixture model where the true clustering is random.
\end{remark}

Classically, we aim to find a clustering rule $\widehat{\mathbf{Z}}$ such that the clustering risk is ``small''. However, as mentioned above, whether this is possible or not depends on the intrinsic difficulty of the clustering problem and thus of the true parameter $\theta^*$ (see Figure~\ref{FMR:fig:intro}). Therefore, the idea is to provide a selection rule, that is, a (measurable) function of the observation $\mathbf{X}$ returning a subset of indices $S \subset \{1,\dots,n\}$, such that the clustering risk {\it with restriction to $S$} is small. 
 Throughout the paper, a {\it procedure} refers to a couple $\proc=(\widehat{\mathbf{Z}}, S)$, where $\widehat{\mathbf{Z}}$ is a clustering rule and $S$ is a selection rule.

\begin{definition}[False membership rate] The false membership rate (FMR) of a procedure $\proc=(\widehat{\mathbf{Z}}, S)$ is given by
\begin{align}\label{defFMR}
\FMR_{\theta^*}(\proc) = \esp_{\theta^*} \left( \underset{\sigma \in [Q]}{\min} \esp_{\theta^*} \left( \frac{\varepsilon_S(\sigma(\widehat{\mathbf{Z}}) ,\mathbf{Z})}{\max(|S|,1)} \: \bigg|\: \mathbf X\right)\right) ,\end{align}
where $\varepsilon_S(\widehat{\mathbf{Z}} ,\mathbf{Z}) =  \sum_{i\in S} \ind \{Z_i \neq \hat{Z}_i \}$  
denotes the misclassification error  restricted to subset $S$. 
\end{definition}

In this work, the aim is  to find a procedure $\proc$ such that the false membership rate is controlled at a nominal level $\alpha$, that is, 
 $\FMR_{\theta^*}(\proc)\leq \alpha$.
 Obviously, choosing $S$ empty  implies $\varepsilon_S( \sigma(\widehat{\mathbf{Z}}),\mathbf{Z})=0$ a.s. for any permutation $\sigma$ and thus satisfies this control.
Hence, 
 while maintaining the control $\FMR_{\theta^*}(\proc)\leq \alpha$, we aim to classify as much individuals as possible, that is, to make $\mathbb{E}_{\theta^*} |S|$ as large as possible. 

The definition of the FMR \eqref{defFMR} involves an expectation of a ratio, which is more difficult to handle than a ratio of expectations. Hence,  the  following simpler alternative criterion will also be useful in our analysis. 

\begin{definition}[Marginal false membership rate] The \textit{marginal} false membership rate (mFMR) of a procedure $\proc=(\widehat{\mathbf{Z}}, S)$ is given by
\begin{align}\label{defmFMR}
\mFMR_{\theta^*} ( \proc ) = \frac{\esp_{\theta^*} \left( \underset{\sigma \in [Q]}{\min} \esp_{\theta^*} \left( \varepsilon_S( \sigma(\widehat{\mathbf{Z}}),\mathbf{Z})  \: \bigg|\: \mathbf X \right) \right)}{\esp_{\theta^*}(|S|)} , 
\end{align}
with the convention $0/0 = 0$.
\end{definition} 
Note that the mFMR is similar to the criterion introduced in \cite{denis20} in the supervised setting.

\subsection{Notation}

We will extensively use the following notation: for all $q\in \{1,\dots,Q\}$ and $\theta=(\pi,\phi)\in \Theta$, we let
\begin{align}
\ell_q(X, \theta) &= \mathbb{P}_{\theta}( Z = q | X) = \frac{\pi_q f_{\phi_q}(X)} {\sum_{\ell=1}^Q \pi_{\ell} f_{\phi_{\ell}}(X) }; \label{equlvalue} \\
T(X, \theta) &= 1 - \max_{q \in \{1, \dots, Q \}} \ell_q(X, \theta) \in [0,1-1/Q] . \label{equTXtheta} 
\end{align}
We can see that $\ell_q(X, \theta)$ is the posterior probability of belonging to class $q$ given the measurement $X$ under the distribution $P_\theta$.
The quantity $T(X, \theta)$ is a measure of the risk when classifying $X$: it is close to $0$ when there exists a class $q$ such that $\ell_q(X, \theta)$ is close to $1$, that is, when $X$ can be classified with large confidence.

\section{Methods} \label{FMR:sec:methods}

In this section, we introduce new methods for controlling the FMR. We start by identifying an {\it oracle} method, that uses the true value of the parameter $\theta^*$. 
Substituting the unknown parameter $\theta^*$ by an estimator in that oracle provides our first method, called the {\it  plug-in} procedure. We then define a refined version of the plug-in procedure, that accounts for the variability of the estimator and is based on a {\it bootstrap} approach. 

\subsection{Oracle procedures} \label{FMR:sec:oracle}
 
 \paragraph{MAP clustering}
 
Here, we proceed as if an oracle had given us the true value of $\theta^*$ and we introduce an oracle procedure $\proc^*=(\widehat{\mathbf{Z}}^*,S^*)$ based on this value.  
As the following lemma shows, the best clustering rule is well-known and given by the Bayes clustering $\widehat{\mathbf{Z}}^* =  (\widehat{Z}^*_1,\dots,\widehat{Z}^*_n)$, which can be written as
\begin{align}\label{equ:Bayesclustering}
\widehat{Z}^*_i   \in \underset{q \in \{1, ..., Q \}}{\argmax} \;\ell_q(X_i, \theta^*) 
,\:\: i\in \{1,\dots,n\},
\end{align}
where $\ell_q(\cdot) $ is the posterior probability given by \eqref{equlvalue}.

\begin{lemma}\label{lem:clusteringBayes}
We have 
$\min_{\widehat{\mathbf{Z}}} R(\widehat{\mathbf{Z}}) =  R(\widehat{\mathbf{Z}}^*)= n^{-1}\: \sum_{i=1}^n  \esp_{\theta^*} (T_i^* )$, for the Bayes clustering $\widehat{\mathbf{Z}}^*$ defined by \eqref{equ:Bayesclustering} and for
\begin{align}\label{eqTistar}
T_i^* = T(X_i, \theta^*) = \Pstar( Z_i \neq \hat{Z}^*_i | X_i) ,\:\: i\in \{1,\dots,n\},
\end{align}
where $T(\cdot)$ is given by \eqref{equTXtheta}. 
\end{lemma}

In words, Lemma~\ref{lem:clusteringBayes} states that the oracle statistics  $T_i^* $ correspond to the posterior misclassification probabilities of the Bayes clustering. 
To decrease the overall misclassification risk, it is natural to avoid classification of points with a high value of the test statistic $T_i^*$.

\paragraph{Thresholding selection rules}

In this section, we introduce the selection rule, that decides which items are to be classified.
From the above paragraph, it is natural to consider a thresholding-based selection rule of the form $S=\{i\in \{1,\dots,n\}\::\: T^*_i \leq t\}$, for some threshold $t$ to be chosen suitably.
The following result gives insights for the choice of such a threshold $t$. 


\begin{lemma}\label{lem:FMRTistar}
For a procedure $\proc=(\widehat{\mathbf{Z}}^*, S)$ with Bayes clustering and an arbitrary selection $S$,
\begin{align}\label{FMRoracle}
\FMR_{\theta^*}(\proc)= \esp_{\theta^*} \left( \frac{\sum_{i\in S} T^*_i }{\max(|S|,1)}  \right) .
\end{align}
\end{lemma}

As a consequence, 
 a first way to build an (oracle) selection is to set
$$S=\{i\in \{1,\dots,n\}\::\: T^*_i \leq \alpha\}.$$ 
Since an average of numbers smaller than $\alpha$ is also smaller than $\alpha$, the corresponding procedure controls the FMR at level $\alpha$. 
This procedure is referred to as the procedure {\it  with fixed threshold} in the sequel.
It corresponds to the following naive approach: to get a clustering with a risk of $\alpha$, we only keep the items that are in their corresponding class with a posterior probability of at least $1-\alpha$. 
By contrast, the selection rule considered here is rather $$S=\{i\in \{1,\dots,n\}\::\: T^*_i \leq t(\alpha)\},$$ for a threshold 
$t(\alpha)\geq \alpha$ maximizing $|S|$ under the constraint $\sum_{i\in S} T^*_i \leq \alpha |S|$. It uniformly improves the procedure with fixed threshold and will in general lead to a (much) broader selection. This gives rise to the {\it oracle procedure}, that can be easily implemented by ordering the $T^*_i$'s, see Algorithm~\ref{algo:oracle}.

\begin{algorithm}[t]
   Input: {Parameter $\theta^*$, sample $(X_1, \dots, X_n)$, level $\alpha$.}\\
  1. Compute the posterior probabilities $\Pstar(Z_i = q | X_i)$, $1\leq i\leq n$, $1\leq q\leq Q$;\\
  2.  Compute the Bayes clustering $\widehat{Z}^*_i$, $1\leq i\leq n$, according to \eqref{equ:Bayesclustering};\\
  3. Compute the probabilities $T^*_i$, $1\leq i\leq n$, according to \eqref{eqTistar};\\
  4. Order these probabilities in increasing order $T^*_{(1)}\leq \dots \leq T^*_{(n)}$;\\
  5. Choose $k^*$ the maximum of $k\in \{0,\dots,n\}$ such that $\max(k,1)^{-1} \sum_{j=1}^k T^*_{(j)} (\mathbf{X})\leq \alpha$;\\
  6. Select $S^*_\alpha$, the index corresponding to the $k^*$ smallest elements among the $T^*_i$'s.\\
   Output: {Oracle procedure $\proc_\alpha=(\widehat{\mathbf{Z}}^*, S^*_\alpha)$.}
     \caption{Oracle procedure}\label{algo:oracle}
\end{algorithm}

\begin{algorithm}[t]
   Input: {Sample $(X_1, \dots, X_n)$, level $\alpha$.}\\
  1. Compute an estimator $\hat{\theta}$ of $\theta$;\\
  2. Run the oracle procedure given in Algorithm~\ref{algo:oracle} with $\hat{\theta}$ in place of $\theta^*$.\\
   Output: {Plug-in procedure $\widehat{\proc}^{\mbox{\tiny PI}}_\alpha=(\widehat{\mathbf{Z}}^{\mbox{\tiny PI}}, \widehat{S}^{\mbox{\tiny PI}}_\alpha)$.}
     \caption{Plug-in procedure}\label{algo:pi}
\end{algorithm}

\begin{algorithm}[t]
   Input: {Sample $(X_1, \dots, X_n)$, level $\alpha$, number $B$ of bootstrap runs.}\\
  1. Choose a grid of increasing levels $(\alpha(k))_{1\leq k\leq K}$;\\
  2. Compute $\widehat{\FMR}^{B}_{\alpha(k)}$, $1\leq k\leq K$, according to \eqref{MCboot};\\
  3. Choose $\tilde{k}$ according to \eqref{ktildeboot}.\\
   Output: {Bootstrap procedure $\widehat{\proc}^{\mbox{\tiny boot}}_\alpha=\widehat{\proc}^{\mbox{\tiny PI}}_{\alpha(\tilde{k})}$.}
     \caption{Bootstrap procedure}\label{algo:boot}
\end{algorithm}

\subsection{Empirical procedures} \label{FMR:sec:empprocedure}

\paragraph{Plug-in procedure}

The oracle procedure cannot be used in practice since $\theta^*$ is  generally unknown. 
A natural idea then is to approach $\theta^*$ by an estimator $\hat{\theta}$ and to plug this estimate into the oracle procedure. 
The resulting procedure, denoted $\widehat{\proc}^{\mbox{\tiny PI}}=(\widehat{\mathbf{Z}}^{\mbox{\tiny PI}}, \widehat{S}^{\mbox{\tiny PI}}_\alpha)$, is called the {\it plug-in procedure} and is implemented in Algorithm~\ref{algo:pi}.


In Section~\ref{FMR:sec:theory}, we establish that the plug-in procedure has suitable properties: when $n$ tends to infinity, provided that the chosen estimator $\hat{\theta}$ behaves well and under mild regularity assumptions on the model, the FMR of the plug-in procedure is close to the level $\alpha$, while it is nearly optimal in terms of average selection number. 

\paragraph{Bootstrap procedure}

Despite the favorable theoretical properties shown in  Section~\ref{FMR:sec:theory}, 
 the plug-in procedure achieves an FMR that can exceed $\alpha$ in some situations, as we will see in our numerical experiments (Section~\ref{FMR:sec:xp}). This is in particular the case when the estimator $\hat{\theta}$ is too rough. Indeed, the uncertainty of $\hat{\theta}$ near $\theta^*$ is ignored by the plug-in procedure.

To take into account this effect, we propose to use a bootstrap approach. It is based on the following result.

\begin{lemma}\label{lem:FMRPI}
For a given level $\alpha\in (0,1)$, the FMR of the plug-in procedure $\widehat{\proc}^{\mbox{\tiny PI}}_\alpha$ is given by 
\begin{align}
\FMR (\widehat{\proc}^{\mbox{\tiny PI}}_\alpha) 
&= \esp_{\mathbf{X}\sim P_{\theta^*}} \left( \min_{\sigma \in [Q]}  \frac{\sum_{i=1}^n  \{ 1-
\ell_{\sigma(\hat{Z}^{\mbox{\tiny PI}}_i(\mathbf{X}))}(X_i, \theta^*) 
\} \ind \{ i\in \hat{S}^{\mbox{\tiny PI}}_\alpha(\mathbf{X}) \} } {\max(|\hat{S}^{\mbox{\tiny PI}}_\alpha(\mathbf{X})|,1) } \right).\label{eqFMRplugin}
\end{align}
\end{lemma}

The general idea is as follows: since $\FMR (\widehat{\proc}^{\mbox{\tiny PI}}_{\alpha})$ can exceed $\alpha$, we choose $\alpha'$ as large as possible such that $\widehat{\FMR}_{\alpha'} \leq\alpha$, for which $\widehat{\FMR}_{\alpha'}$ is a bootstrap approximation of $\FMR (\widehat{\proc}^{\mbox{\tiny PI}}_{\alpha'})$ based on  \eqref{eqFMRplugin}.

The bootstrap approximation reads as follows: in the RHS of \eqref{eqFMRplugin}, we replace  the true parameter $\theta^*$ by $\hat{\theta}$ and $\mathbf{X}\sim P_{\theta^*}$ by  $\mathbf{X}'\sim \hat{P}$, where $\hat{P}$ is an empirical substitute of $P_{\theta^*}$. This empirical distribution  $\hat{P}$ is $P_{\hat{\theta}}$ for the parametric bootstrap and the uniform distribution over the $X_i$'s for the non-parametric bootstrap. This yields the bootstrap approximation of $\FMR (\widehat{\proc}^{\mbox{\tiny PI}}_\alpha)$ given by
\begin{align*}
\widehat{\FMR}_\alpha &=  \esp_{\mathbf{X}'\sim \hat{P}} \left( \min_{\sigma \in [Q]} \frac{\sum_{i=1}^n \{
 1-\ell_{\sigma(\hat{Z}^{\mbox{\tiny PI}}_i(\mathbf{X}'))}(X'_i, \hat{\theta}(\mathbf X)) 
 \} \ind \{ i\in \hat{S}^{\mbox{\tiny PI}}_\alpha(\mathbf{X}') \} } {\max(|\hat{S}^{\mbox{\tiny PI}}_\alpha(\mathbf{X}')|,1) } \bigg| \: \mathbf{X} \right) .
\end{align*}
Classically, the latter is itself approximated by  a Monte-Carlo scheme:
\begin{align}\label{MCboot}
\widehat{\FMR}^{B}_\alpha &= \frac{1}{B} \sum_{b=1}^B \min_{\sigma \in [Q]}  \frac{\sum_{i=1}^n \{
 1-\ell_{\sigma(\hat{Z}^{\mbox{\tiny PI}}_i(\mathbf{X}^b))}(X^b_i, \hat{\theta}(\mathbf X)) 
\}
 \ind \{ i\in \hat{S}^{\mbox{\tiny PI}}_\alpha(\mathbf{X}^b) \} } {\max(|\hat{S}^{\mbox{\tiny PI}}_\alpha(\mathbf{X}^b)|,1) },
\end{align}
with $\mathbf{X}^1,\dots, \mathbf{X}^B$ i.i.d. $\sim \hat{P}$ corresponding to the bootstrap samples of $\mathbf{X}$. 

Let $(\alpha(k))_{1 \leq k \leq K}\in (0,1)^K$ be a grid of increasing nominal levels (possibly with restriction to values slightly below the target level $\alpha$). Then, the bootstrap procedure at level $\alpha$ is defined as $\widehat{\proc}^{\mbox{\tiny boot}}_\alpha=\widehat{\proc}^{\mbox{\tiny PI}}_{\alpha(\tilde{k})}$, where 
\begin{align}\label{ktildeboot}
\tilde{k}=\max\left\{  k \in \{1,\dots, K\}\::\: \widehat{\FMR}^{B}_{\alpha(k)}\leq \alpha\right\}.
\end{align}
This procedure is implemented in Algorithm~\ref{algo:boot}.

\begin{remark}[Parametric versus non parametric bootstrap]
The usual difference between parametric and non parametric bootstrap also holds in our context: the parametric bootstrap is fully based on $P_{\hat{\theta}}$,  while the non parametric bootstrap builds an artificial sample (with replacement) from the original sample, which does not come from a $P_\theta$-type distribution. This gives rise to different behaviors in practice: 
when $\hat{\theta}$ is too optimistic (which will be typically the case here when the estimation error is large), the correction brought by the parametric bootstrap (based on $P_{\hat{\theta}}$) is often weaker than that of the non parametric one. By contrast, when  $\hat{\theta}$ is close to the true parameter, the parametric bootstrap approximation is more faithful because it uses the model, see Section~\ref{FMR:sec:xp}.\end{remark}

\section{Theoretical guarantees for the plug-in procedure}\label{FMR:sec:theory}

In this section, we derive theoretical properties for the plug-in procedure: we show that its FMR  and mFMR are close to $\alpha$, while its expected selection number is close to be optimal under some conditions. 

\subsection{Additional notation and assumptions} \label{FMR:sec:notation}

We make use of an optimality theory for mFMR control, that will be developed in detail in Section~\ref{FMR:sec:optimalproc}. 
This approach extensively relies on the following quantities (recall the definition of $T(X, \theta) $ in \eqref{equTXtheta}):
\begin{align}
\mFMR^*_t& = \esp_{\theta^*} \left( T(X, \theta^*) \:|\: T(X, \theta^*) < t \right);\label{mFMRstart}\\
t^*(\alpha) &=  \sup \left\{ t \in [0,1]\::\: \mFMR^*_t \leq \alpha \right\} \label{tstaralpha}\\
\alpha_c &= \inf\{\mFMR^*_t \::\: t\in (0,1], \mFMR^*_t>0\}; \label{alphac}\\
 \bar{\alpha}&=\mFMR^*_1\label{alphabar}.
  \end{align}
In words, $\mFMR^*_t$ is the mFMR of an oracle procedure that selects the $T_i^*$ smaller than some threshold $t$ (Lemma~\ref{mFMR_properties}).  Then, $t^*(\alpha)$ is the optimal threshold such that this procedure has an mFMR controlled at level $\alpha$. Next, $\alpha_c$ and $\bar{\alpha}$ are the lower and upper bounds for the nominal level $\alpha$, respectively, for which the optimality theory can be applied. 

Now, we introduce our main assumption, which will be ubiquitous in our analysis. 

\begin{assumption} \label{Tcontinuity} 
For all $\theta\in \Theta$ and $q \in \{1, \dots, Q\}$, under $P_{\theta^*}$, the r.v. $\ell_q(X, \theta)$ given by \eqref{equlvalue} is continuous. 
In addition, the function $t\mapsto \Pstar \left( T(X, \theta) < t \right)$ is increasing on $(\alpha_c,\bar\alpha)$, where $T(X, \theta)$ is given by \eqref{equTXtheta}.
\end{assumption}

Note that  Assumption~\ref{Tcontinuity} implies the continuity of the r.v. $T(X, \theta)$. 
Indeed, $\P(T(X, \theta)=t)\leq \sum_{q=1}^Q \P(\ell_q(X, \theta) =1-t)$. 
Hence, this assumption implies that $t\mapsto \Pstar \left( T(X, \theta) < t \right)$ is both continuous on $[0,1]$ and increasing on $(\alpha_c,\bar\alpha)$.
This 
is useful in several regards: first, it prohibits ties in the $T(X_i,\theta)$'s, $1\leq i\leq m$, so that the selection rule (see Algorithm~\ref{algo:oracle}) can be truly formulated as a thresholding rule (see Lemma~\ref{lempluginthres}). Second, it entails interesting properties for function $t\mapsto \mFMR^*_t$, see Lemma~\ref{mFMR_properties} (this in particular ensures that the supremum in \eqref{tstaralpha} is a maximum).
Also note that the inequality $0\leq \alpha_c<  \bar{\alpha}< 1-1/Q$ holds under Assumption~\ref{Tcontinuity}.

The next assumption ensures that the density family $\{f_u, u \in \mathcal{U} \}$ is smooth, and will be useful to establish consistency results.

\begin{assumption} \label{densitycontinuity} For $P_{\theta^*}$-almost all $x\in \mathbb{R}^d$, $u \in \mathcal{U} \mapsto f_u(x)$ is continuous. 
\end{assumption} 
Moreover, we can derive convergence rates under the following additional regularity conditions. 
\begin{assumption} \label{lip}
There exist positive constants $r = r(\theta^*), C_1 = C_1(\theta^*), C_2 = C_2(\theta^*, \alpha), C_3 = C_3(\theta^*, \alpha)$ such that
\begin{enumerate}[label={\textbf{~(\roman*)}},
  ref={\theassumption.\roman*}]
\item for $\Pstar$-almost all $x$, $u\in \mathcal{U} \mapsto f_u(x)$ is continuously differentiable, and 
\begin{align*} \sum_{1\leq q\leq Q} \esp_{\theta^*} \underset{\substack{\theta\in \Theta \\  \|\theta-\theta^*\| \leq r}}{\sup} \| \nabla_\theta \ell_q(X, \theta) \| \leq  C_1; \end{align*}
\item for all $t,t'\in [0,1],$ 
$|\Pstar(T(X, \theta^*) < t ) - \Pstar(T(X, \theta^*) < t' )|\leq C_2 |t - t'|;$
\item for all $\beta \in [(\alpha_c+\alpha)/2,(\alpha+ \bar{\alpha})/2] ,$ $| t^*(\beta) - t^*(\alpha) |\leq C_3|\beta - \alpha|$.
 \end{enumerate}
\end{assumption}

\begin{example}
In Appendix~\ref{FMR:sec:gauss}, it is proved that Assumptions~\ref{Tcontinuity},~\ref{densitycontinuity} and~\ref{lip} hold true in the homoscedastic two-component multivariate Gaussian mixture model, see Lemma~\ref{lem:exass}.
\end{example}

Next, we consider the following complexity assumption to ensure concentration of the underlying empirical processes. It is given in terms of the VC dimension of specific function classes involving $\ell_q$. In the sequel, the VC dimension of a function set $\mathscr{F}$ is defined as the VC dimension of the set family $\{ \{ x \in \mathbb{R}^d:f(x) \geq u  \}, f \in \mathscr{F}, u \in \mathbb{R} \}$, see, e.g., \cite{baraud2016bounding}.
We denote
\begin{align}
\VCdim &= \mbox{ VC dimension of } \{\ell_q(., \theta), \theta \in \Theta, 1\leq q \leq Q\}\label{VCdim};\\
\VCdim_{-} &= \mbox{ VC dimension of } \{ \ind\{\ell_q(., \theta) - \ell_{q'}(., \theta) \geq 0  \}, \theta \in \Theta, 1\leq q,q'\leq Q \} . \label{VCdimminus}
\end{align}

\begin{assumption} \label{ass_vcdim}
	The VC dimensions 
	$\VCdim$ and $\VCdim_{-}$ are finite. 
\end{assumption}

\begin{example}
	In the two-component case $Q=2$ where $P_\theta$ belongs to an exponential family, we have that $\VCdim, \VCdim_{-}  \lesssim k^2 \log(k)$ (see Lemma \ref{lem:VCdim_exp}) with $k$ the dimension of the sufficient statistic vector. For instance, $k = d + d^2$ for the Gaussian family, hence  $\VCdim, \VCdim_{-} \lesssim d^4\log(d)$ in that case. (For the specific case of the homoscedastic Gaussian family, we have that $\VCdim, \VCdim_{-}  \lesssim d$, see Lemma \ref{lem:VCdim_gaussian}). 
\end{example}

Let us now discuss conditions on the estimator $\hat{\theta}$ on which the plug-in procedure is based. We start by introducing the following assumption (used in the concentration part of the proof, see Lemma~\ref{lem:conc}).

\begin{assumption} \label{ass_estim}
The estimator $\hat{\theta}$ is assumed to take its values in a countable subset $\subsetEsti$ of $\Theta$.
\end{assumption}
This assumption is a minor restriction, because we can always choose $\subsetEsti\subset \mathbb{Q}^K$ (recall $\Theta\subset \mathbb{R}^K$). 
Next, we additionally define a quantity measuring the quality of the estimator:  for all $\epsilon>0$,
\begin{equation}\label{equeta}
\eta(\epsilon,\theta^*)=\Pstar \left( \underset{\sigma \in [Q]}{\min} \|\hat{\theta}^\sigma - \theta^*\|_2 \geq \epsilon \right).
\end{equation}
\begin{example} 
		The literature provides several results regarding the estimation of Gaussian mixtures, see e.g. \cite{regev17} for a review. Proposition~\ref{prop:etagauss} revisits some of these results, for the estimator derived from  EM algorithm \citep{dempster77,balakrishnan17} and the constrained MLE \citep{ho16}.		
\end{example}

\subsection{Results}

We now state our main results, starting with the consistency of the plug-in procedure.

\begin{theorem}[Asymptotic optimality of the plug-in procedure] \label{consistency}
	Let Assumptions~\ref{Tcontinuity},~\ref{densitycontinuity}, and~\ref{ass_vcdim} be true. 
	Consider an estimator $\hat{\theta}$ satisfying Assumption~\ref{ass_estim} and which is consistent in the sense that for all $\epsilon>0$, the probability $\eta(\epsilon,\theta^*)$ given by \eqref{equeta} tends to $0$ as $n$ tends to infinity.
	Then the corresponding plug-in procedure $\widehat{\proc}^{\mbox{\tiny PI}}_{\alpha}$ (Algorithm~\ref{algo:pi}) satisfies the following: for any $\alpha \in (\alpha_c, \bar{\alpha})$, we have 
	\begin{align*}
		\limsup_n \FMR(\widehat{\proc}^{\mbox{\tiny PI}}_{\alpha}) \leq \alpha, \quad
		\limsup_n \mFMR(\widehat{\proc}^{\mbox{\tiny PI}}_{\alpha}) \leq \alpha,
	\end{align*}
	and for any procedure $\proc = (\widehat{\mathbf{Z}}, S) $ that controls the mFMR at level $\alpha$, we have
	\begin{align*}
		\liminf_n \{ n^{-1}\esp_{\theta^*}(|\widehat{S}^{\mbox{\tiny PI}}_{\alpha}|) -n^{-1} \esp_{\theta^*}(|S|) \} \geq 0. 
	\end{align*}
\end{theorem}

Next, we derive convergence rates under the additional regularity conditions given by Assumption~\ref{lip}.

\begin{theorem}[Optimality of the plug-in procedure with rates] \label{FMR:th:convergence_rates}
	Consider the setting of Theorem~\ref{consistency}, where in addition Assumption~\ref{lip} holds. 
	Recall $\eta(\epsilon,\theta^*)$ defined by \eqref{equeta} and $\VCdim, \VCdim_{-}$ defined by \eqref{VCdim}, \eqref{VCdimminus} respectively.
	Let $s^*$ denote the selection rate of the oracle procedure mentioned in Section~\ref{FMR:sec:notation}, with threshold $t^*(\alpha)$ and applied at level $(\alpha+\alpha_c)/2$. With constants $A>0$ and $B>0$ only depending on $Q, C_1, C_2, C_3, \VCdim, \VCdim_{-}$ and $s^*$, we have for any sequence $\epsilon_n > 0$ tending to zero, for $n$ larger than a constant only depending on $\alpha$ and $\theta^*$,
	\begin{align}
		\FMR(\widehat{\proc}^{\mbox{\tiny PI}}_{\alpha}) &\leq \alpha+ A\sqrt{\epsilon_n} + B\sqrt{\log n/n} + 5/n^2 + \eta(\epsilon_n, \theta^*) \label{FMR_bound} \\
		n^{-1} \esp_{\theta^*}(|\widehat{S}^{\mbox{\tiny PI}}_{\alpha}|) - n^{-1} \esp_{\theta^*}(|S|)&\geq  - A\sqrt{\epsilon_n} - B\sqrt{\log n/n} - 5/n^2 - \eta(\epsilon_n, \theta^*)   \label{sr_bound},
	\end{align}
	for any procedure $\proc = (\widehat{\mathbf{Z}}, S) $ that controls the mFMR at level $\alpha$.
\end{theorem}
The proof is based on a more general non-asymptotical result, for which the remainder terms are more explicit, see Theorem~\ref{result} and Appendix~\ref{FMR:sec:proofmain}. It employs techniques that share similarities with the work of \cite{rebafka22} developed in a different context. 
Here, a difficulty is to handle the new statistic $T(X_i,\hat{\theta})$ which is defined as an extremum, see \eqref{equTXtheta}.

Theorem~\ref{FMR:th:convergence_rates} establishes that, given a model which is regular enough and a consistent estimator, the plug-in procedure controls the FMR and is asymptotically optimal up to remainder terms which are of the order of $\sqrt{\epsilon_n} +\sqrt{\log n/n} + \eta(\epsilon_n, \theta^*)$. Here, $\epsilon_n$ dominates the convergence rate of the parameter estimate, and is taken large enough to ensure that $\eta(\epsilon_n, \theta^*)$ vanishes. 

For instance, in the multivariate Gaussian mixture model (with further assumptions) and by considering either the EM estimator or the constrained MLE, we have $\eta(\epsilon_n, \theta^*) \leq 1/n$ for $\epsilon_n = C\sqrt{\log(n) /n }$, see Proposition~\ref{prop:etagauss}.
This implies that the remainder terms in \eqref{FMR_bound} and  \eqref{sr_bound} are at most of order $((\log n)/n)^{1/4}$.

\section{Experiments} \label{FMR:sec:xp}

In this section, we evaluate the behavior of the new procedures: plug-in (Algorithm~\ref{algo:pi}), parametric bootstrap and non parametric bootstrap (Algorithm~\ref{algo:boot}). For this, we use both synthetic and real data.

\subsection{Synthetic data set}\label{FMR:sec:numexp}

The performance of our procedures is studied via simulations in different settings with various difficulties. All of them are Gaussian mixture models, with possible restrictions on the parameter space. For parameter estimation, the classical EM algorithm is applied with $100$ iterations and $10$ starting points chosen with Kmeans++ \citep{arthur2006k}. In the bootstrap procedures $B=1000$ bootstrap samples are generated. The performance of all procedures is assessed via the {\it sample FMR} and the proportion of classified data points, which is referred to as the {\it selection frequency}. For every setting and every set of parameters, depicted results display the mean over $100$ simulated datasets. 
As a baseline, we consider the fixed threshold procedure in which one selects data points that have a maximum posterior group membership probability that exceeds $1-\alpha$. 
The oracle procedure (Algorithm~\ref{algo:oracle}) is also considered in our experiments for comparison.


\paragraph{Known proportions and covariances} 

In the first setting, the true mixture proportions and covariance matrices are known and used in the EM algorithm. We consider the case $Q=2$, $\pi_1 = \pi_2 = 1/2$ and $\Sigma_1=\Sigma_2 = I_d$ with $I_d$ the $(d \times d)$-identity matrix. For the mean vectors, we set $\mu_1 = 0$ and $\mu_2 = (\epsilon / \sqrt{d},\dots, \epsilon / \sqrt{d})$. The quantity $\epsilon$ corresponds to the mean separation, that is, $\| \mu_1 - \mu_2 \|_2=\epsilon $ and accounts for the difficulty of the clustering problem. 

Figure~\ref{FMR:fig:semiOracle} displays the FMR for nominal level  $\alpha = 0.1$, sample size $n=100$, dimension $d =2$ and varying mean separation $\epsilon\in \{1,\sqrt{2},2,4\}$. 
Globally, our procedures all have an FMR close to the target level $\alpha$ (excepted for the very well separated case $\epsilon=4$ for which the FMR is much smaller because a large part of the items can be trivially classified). In addition, the selection rate is always close to the one of the oracle procedure.
On the other hand, the baseline procedure is too conservative: its FMR can be well below the nominal level and it selects up to 50\% less than the other procedures. This is well expected, because unlike our procedures, the baseline has a fixed threshold and thus does not adapt to the difficulty of the problem. 

We also note that the FMR of the plug-in approach is slightly inflated for a weak separation ($\epsilon=1$). This comes from the parameter estimation, which is difficult in that case. This also illustrates the interest of the bootstrap methods, that allow to recover the correct level in that case, by appropriately correcting the plug-in approach.

\begin{figure}
	\begin{minipage}{0.8\linewidth}
		\centering
		\includegraphics[width=0.6\linewidth]{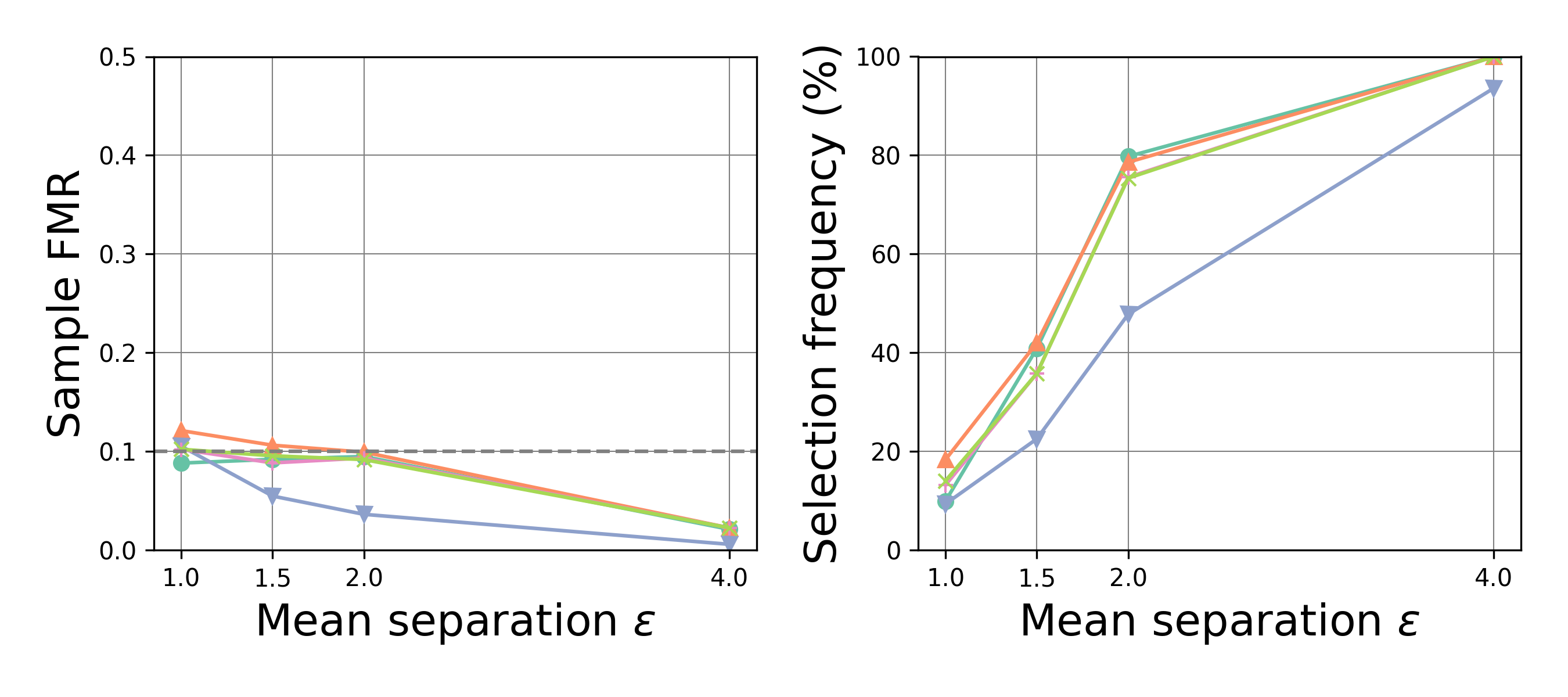}  
	\end{minipage}%
	\begin{minipage}{0.2\linewidth}
		\includegraphics[width=\linewidth]{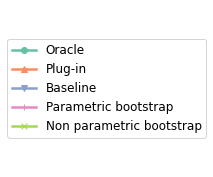}  
	\end{minipage}
\caption{FMR (left panel) and selection frequency (right panel) as a function of the mean separation $\epsilon$. Known mixture proportions and covariances setting with $Q=2$, $n=100$, $d=2$, $\alpha=0.1$.} \label{FMR:fig:semiOracle} 
\end{figure}

\paragraph{Diagonal covariances} 

In this setting, the true parameters are the same as in the previous paragraph, but the true mixture proportions and covariance matrices are unknown. However, to help the estimation, we suppose a diagonal structure for $\Sigma_1$ and $\Sigma_2$, which is used in the EM algorithm.

Figure~\ref{FMR:fig:diagepsilon} displays the FMR and the selection frequency as a function of the separation $\epsilon$. The conclusion is qualitatively the same as in the previous case, but with larger FMR values for a weak separation. 
Overall, it shows that the plug-in procedure is anti-conservative and that the bootstrap corrections are able to recover an FMR and a selection frequency close to the one of the oracle. However, for a weak separation, namely $\epsilon=1$, the parametric bootstrap correction is not enough and the latter procedure still overshoots the nominal level $\alpha$. 
Indeed, in our simulations, it appears that $P_{\hat{\theta}}$ is often a distribution that is more favorable than  $P_{\theta^*}$ from a statistical point of view (for instance, with more separated clusters). 
These conclusions also hold for varying sample size $n$, see Figure~\ref{FMR:fig:diagsamplesize}. 

Figure \ref{FMR:fig:constrained} displays the FMR and the selection frequency for varying nominal level $\alpha$, with $\epsilon = \sqrt{2}$ and $n=200$. The plug-in is still anti-conservative, while  the bootstrap procedures have an FMR that is close to $\alpha$ uniformly on the considered $\alpha$ range. 
Moreover, we note that for all our procedures (including the plug-in), the gap between the FMR and the nominal level is roughly constant with $\alpha$: this illustrates the adaptive aspect of our procedures.
This is in contrast with the baseline procedure, for which this gap highly depends on $\alpha$, and which may be either anti-conservative or sub-optimal depending on the $\alpha$ value.

\begin{figure}
	\adjustbox{minipage=1.3em,valign=t}{\subcaption{}\label{FMR:fig:diagepsilon}}%
	\begin{minipage}{0.745\linewidth}
		\centering
		\includegraphics[width=0.7\linewidth]{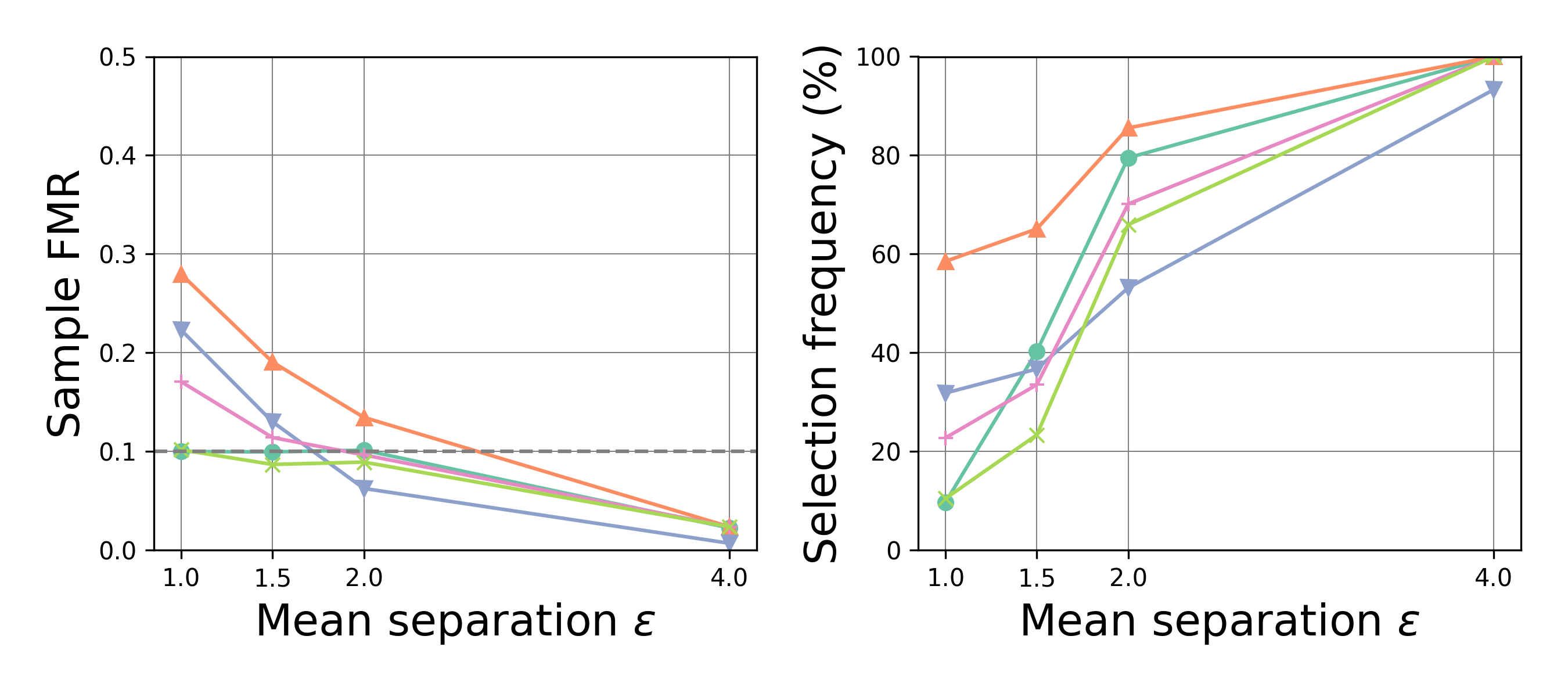}  
	\end{minipage}%
	\begin{minipage}{0.25\linewidth}
	\hspace{\fill}
	\end{minipage}
	
	\adjustbox{minipage=1.3em,valign=t}{\subcaption{}\label{FMR:fig:diagsamplesize}}%
	\begin{minipage}{0.745\linewidth}
		\centering
		\includegraphics[width=0.7\linewidth]{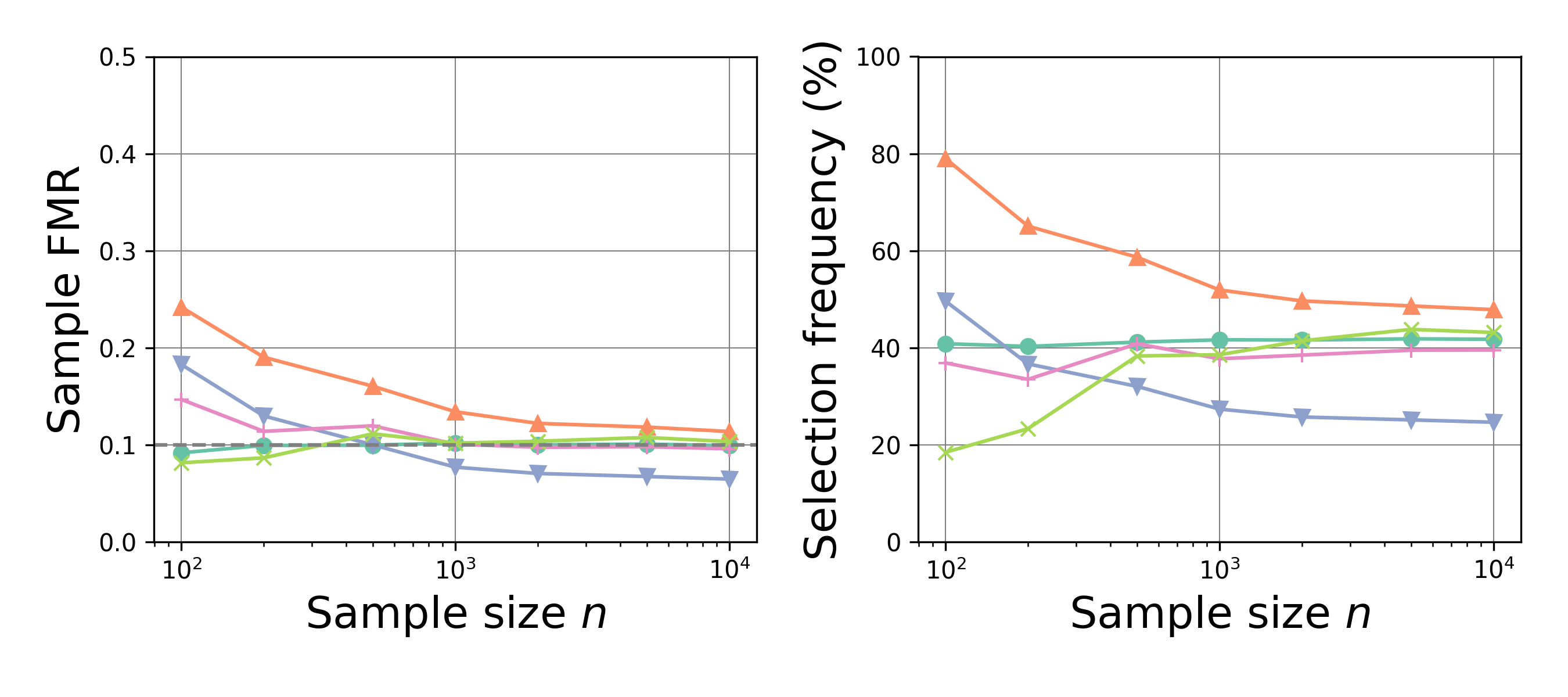}  
	\end{minipage}%
	\begin{minipage}{0.25\linewidth}
		\includegraphics[width=\linewidth]{legend.png}  
	\end{minipage}
	
	\adjustbox{minipage=1.3em,valign=t}{\subcaption{}\label{FMR:fig:constrained}}%
	\begin{minipage}{0.745\linewidth}
		\centering
		\includegraphics[width=0.7\linewidth]{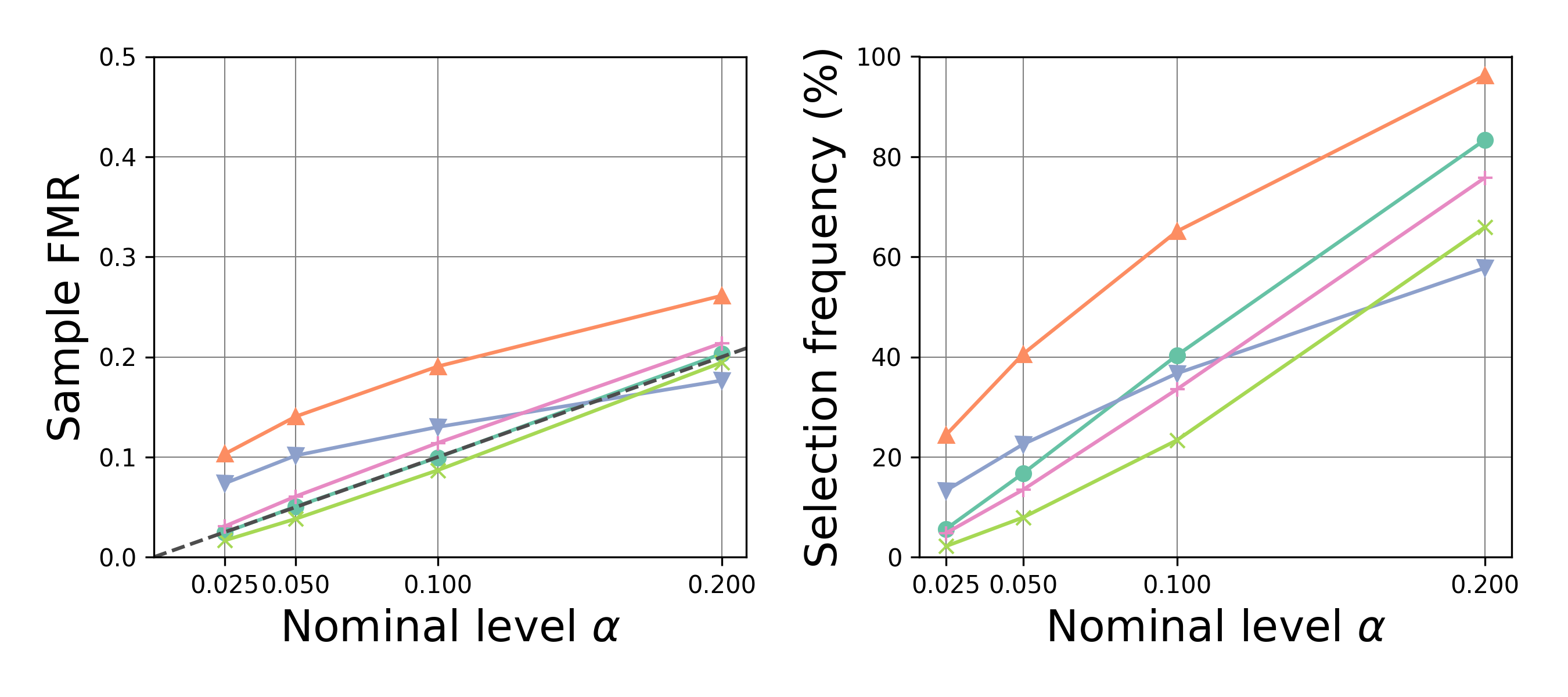}  
	\end{minipage}%
	\begin{minipage}{0.25\linewidth}
	\hspace{\fill}
	\end{minipage}
	
	\caption{FMR (left panel) and selection frequency (right panel) as a function of: (a) the mean separation; (b) the sample size $n$; (c) the nominal level $\alpha$. Diagonal covariances setting with  $Q=2$, $d=2$. Default settings are: $n=200$, $\alpha=0.1$, $\epsilon= \sqrt{2}$.} 
\end{figure}

\paragraph{Three-component mixture} 

We next increase the number of classes to $Q=3$. Figure \ref{FMR:fig:Q3} displays the FMR and the selection frequency for varying $\alpha$, with a mean separation $\| \mu_1 - \mu_2 \|_2=\| \mu_1 - \mu_3 \|_2=\| \mu_2 - \mu_3 \|_2=\epsilon =3$. The mean separation is chosen so that the selection frequency of the oracle rule is approximately the same as in the previous paragraph. The increase in $Q$ leads to a deterioration of the performances. Specifically, the FMR of the plug-in overshoots the nominal level by a large amount, and when $n$ is too small, the parametric bootstrap procedure can be anti-conservative while the non parametric bootstrap is over-conservative.
This deterioration is expected since from the theory established in Section \ref{FMR:sec:theory} (see Theorem \ref{FMR:th:convergence_rates}), the residual terms increase with $Q$, and since the difficulty of the estimation is also increased. 
However, for  a fairly large sample size ($n=1000$), both bootstrap procedures are correctly mimicking the oracle.

\begin{figure}
	\begin{minipage}{\linewidth}
		\centering
		\includegraphics[width=0.9\linewidth]{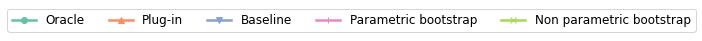}  
	\end{minipage}
	\begin{minipage}{0.5\linewidth}
		\centering
		\includegraphics[width=\linewidth]{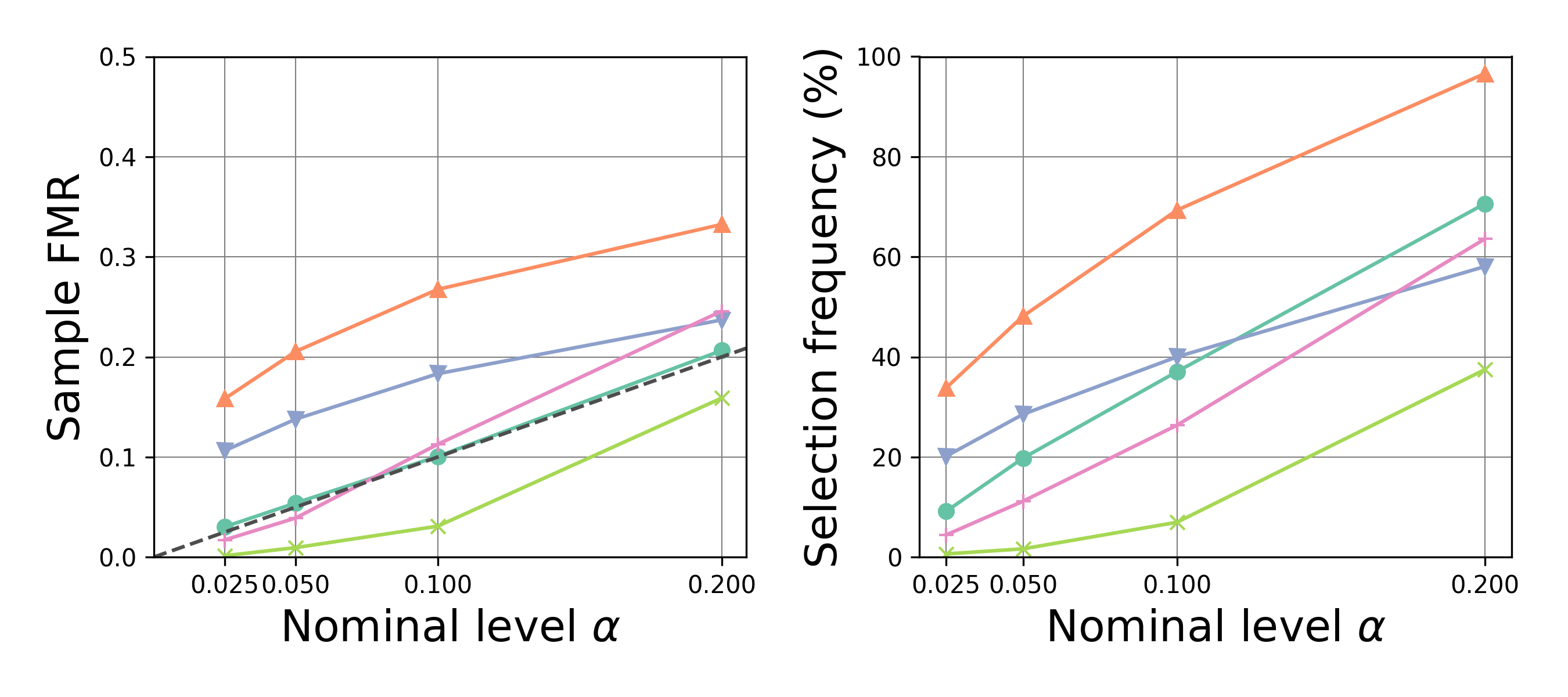}  
		\subcaption{$n=200$}
	\end{minipage}%
	\begin{minipage}{0.5\linewidth}
		\centering
		\includegraphics[width=\linewidth]{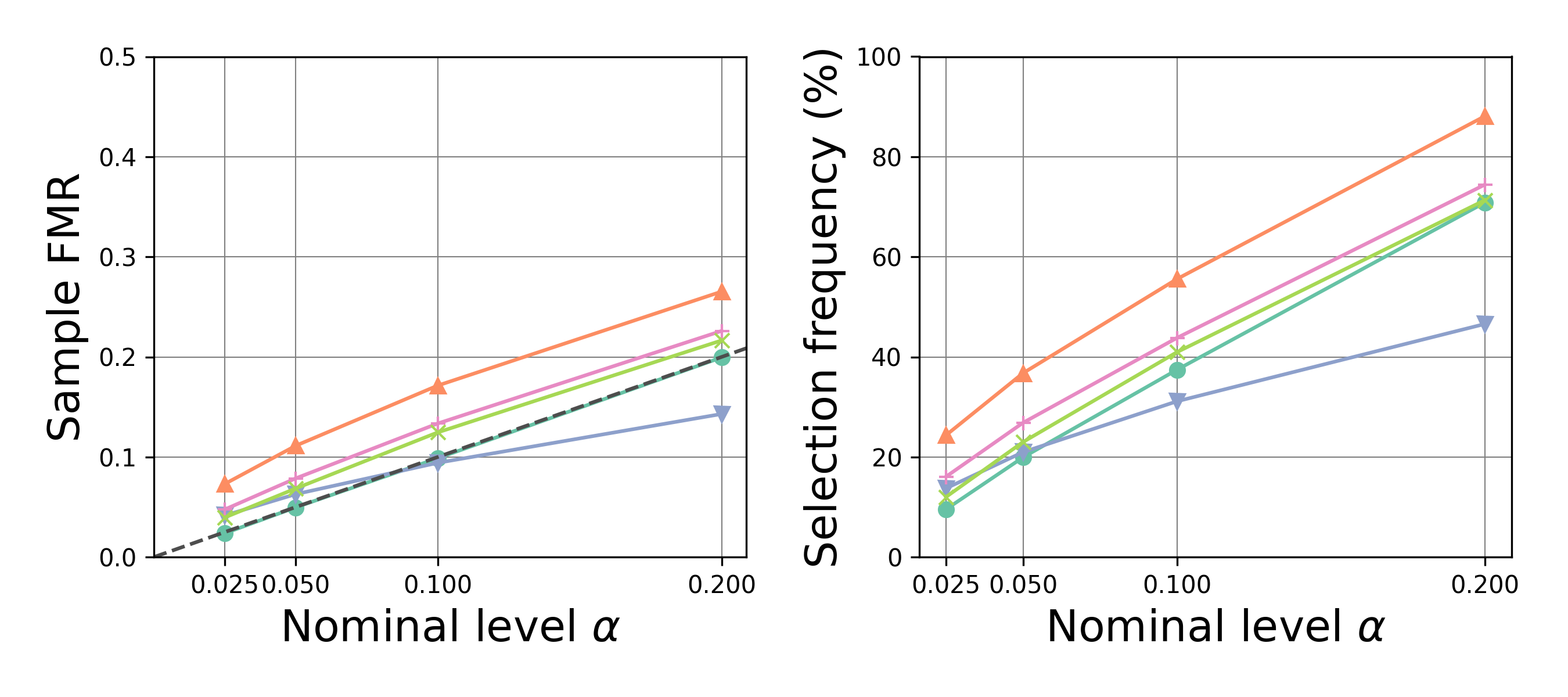}  
		\subcaption{$n=1000$}
	\end{minipage}
		\caption{FMR (left panel) and selection frequency (right panel) as a function of the nominal level $\alpha$. Diagonal covariances setting with $Q=3$, $d=2$, $\epsilon=3$.} \label{FMR:fig:Q3} 
\end{figure}

\paragraph{Larger dimension} 
	
	We now increase the dimension to $d=10$. In that case, parameter estimation is deteriorated. In particular, the maximum posterior probability for any point tends to be very over-estimated. To remedy this issue, we project the data onto a two-dimensional space using PCA. We then apply the EM algorithm to the projected data. This is similar in spirit to spectral clustering and it has the added benefit of combining the objectives of data reduction with clustering. Results are displayed in Figure \ref{FMR:fig:dim}. The conclusions are qualitatively the same as in the previous paragraph.

\begin{figure}
	\begin{minipage}{\linewidth}
		\centering
		\includegraphics[width=0.9\linewidth]{legend2.png}  
	\end{minipage}
	\begin{minipage}{0.5\linewidth}
		\centering
		\includegraphics[width=\linewidth]{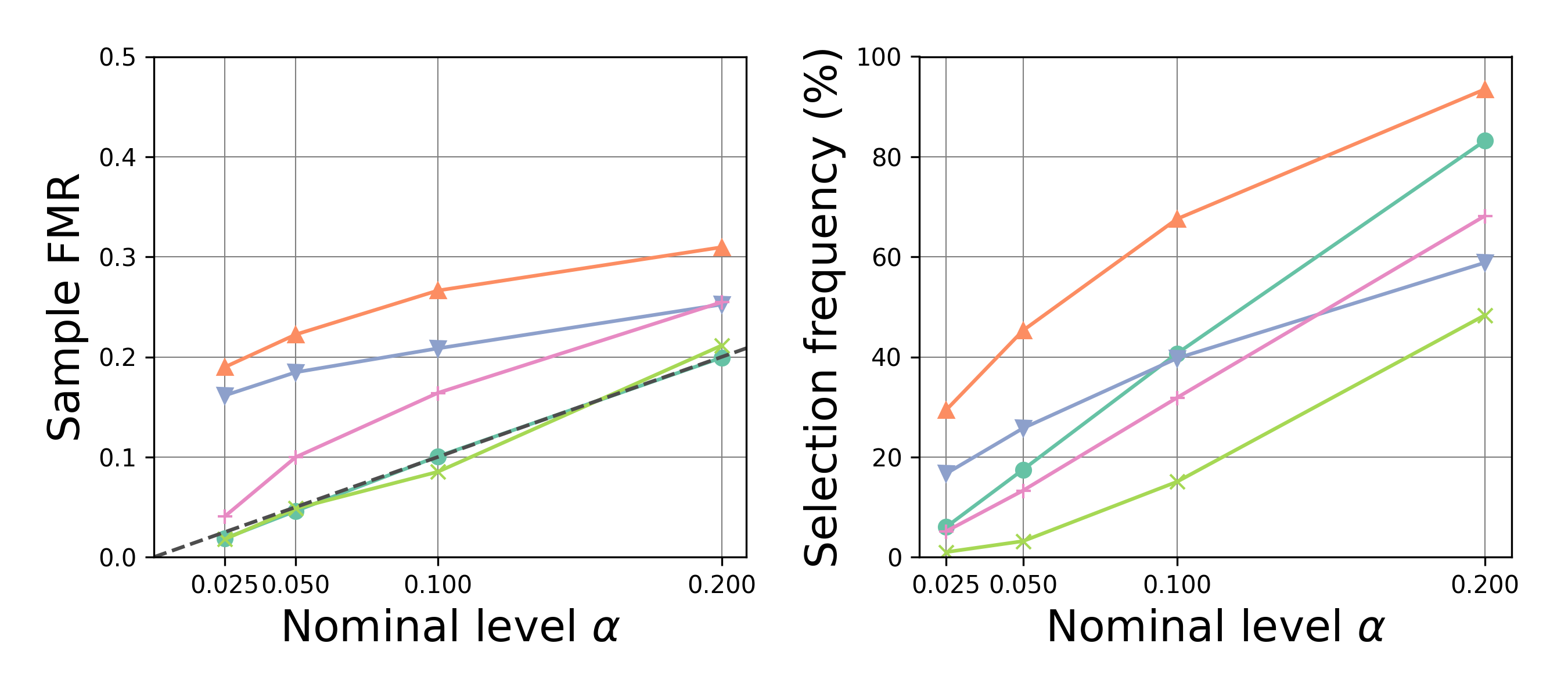}  
		\subcaption{$n=200$}
	\end{minipage}%
	\begin{minipage}{0.5\linewidth}
		\centering
		\includegraphics[width=\linewidth]{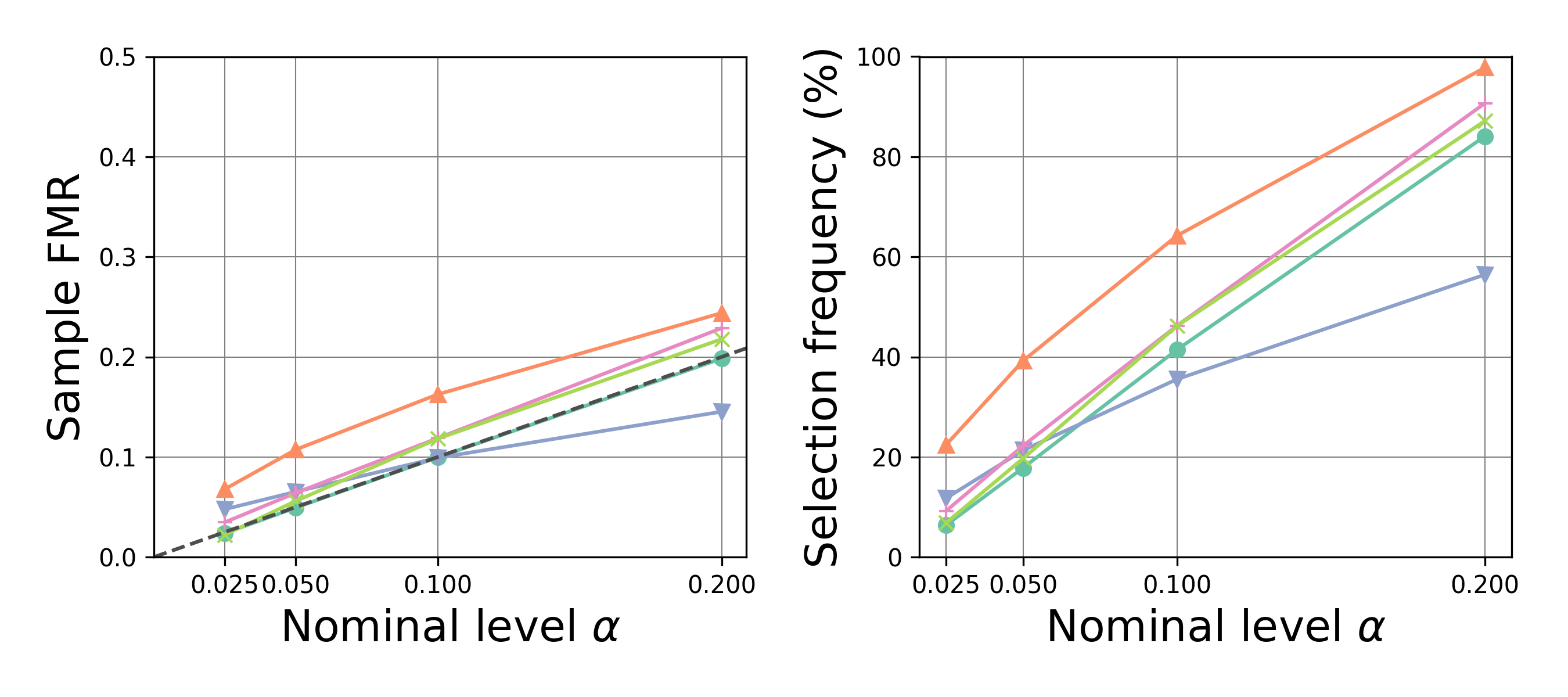}  
		\subcaption{$n=1000$}
	\end{minipage}
	\caption{FMR (left panel) and selection frequency (right panel) as a function of the nominal level $\alpha$. Diagonal covariances setting with $Q=2$, $d=10$, $\epsilon=\sqrt{2}$.} \label{FMR:fig:dim} 
\end{figure}

\subsection{Real data set} \label{FMR:sec:realdata}

We consider the Wisconsin Breast Cancer Diagnosis (WDBC) dataset from the UCI ML repository. The data consists of features computed from a digitalized image of a fine needle aspirate (FNA) of a breast mass, on a total of 569 patients (each corresponds to one FNA sample) of which 212 are diagnosed as Benign and 357 as Malignant. Ten real-valued measures were computed for each of the cell nucleus present in the images (e.g. radius, perimeter, texture, etc.). Then, the mean, standard error and mean of the three largest values of these measures were computed for each image, resulting in a total of 30 features. Here, we restrict the analysis to the variables that correspond to the means of these measures. 

We choose to model the data as a mixture of Student's t-distributions as proposed in \cite{peel00}. Student mixtures are appropriate for data  containing observations with longer than normal tails or atypical observations leading to  overlapping  clusters. 
Compared to Gaussian mixtures, Students are less concentrated and thus produce estimates of the posterior probabilities of class memberships that are  less extreme, which is favorable for our selection procedures.
In our study, the degree of freedom of each component is set to 4, and  no constraints are put on the rest of the parameters. The t-mixture is fit via the EM algorithm provided by the Python package \texttt{studenttmixture} \citep{peel00}.  


To start with, we restrict the analysis to the first two variables of the dataset, the mean radius and the mean texture of the images. For illustration, Figure \ref{FMR:fig:app} (panel (a)) displays the data. Different colors indicate the ground truth labels (this information is not used in the clustering).
One can see that  the Student approximation is fairly good for each of the groups, and   there is some overlap between them.   Figure \ref{FMR:fig:app} (panel (b)) displays the MAP clustering result for the t-mixture model without any selection. The FMR is computed with respect to the ground truth labels and amounts to 14 \%. Figure \ref{FMR:fig:app} (panel (c)) provides the result of our parametric bootstrap procedure with nominal level $\alpha = 5\%$. The procedure does not classify points that are at the intersection of the clusters, resulting in the classification of 70\% of the data, and the FMR equals 3\%, which is below the target level.

\begin{figure}
\begin{minipage}[t]{0.33\linewidth}
	\centering
	\includegraphics[width=\linewidth]{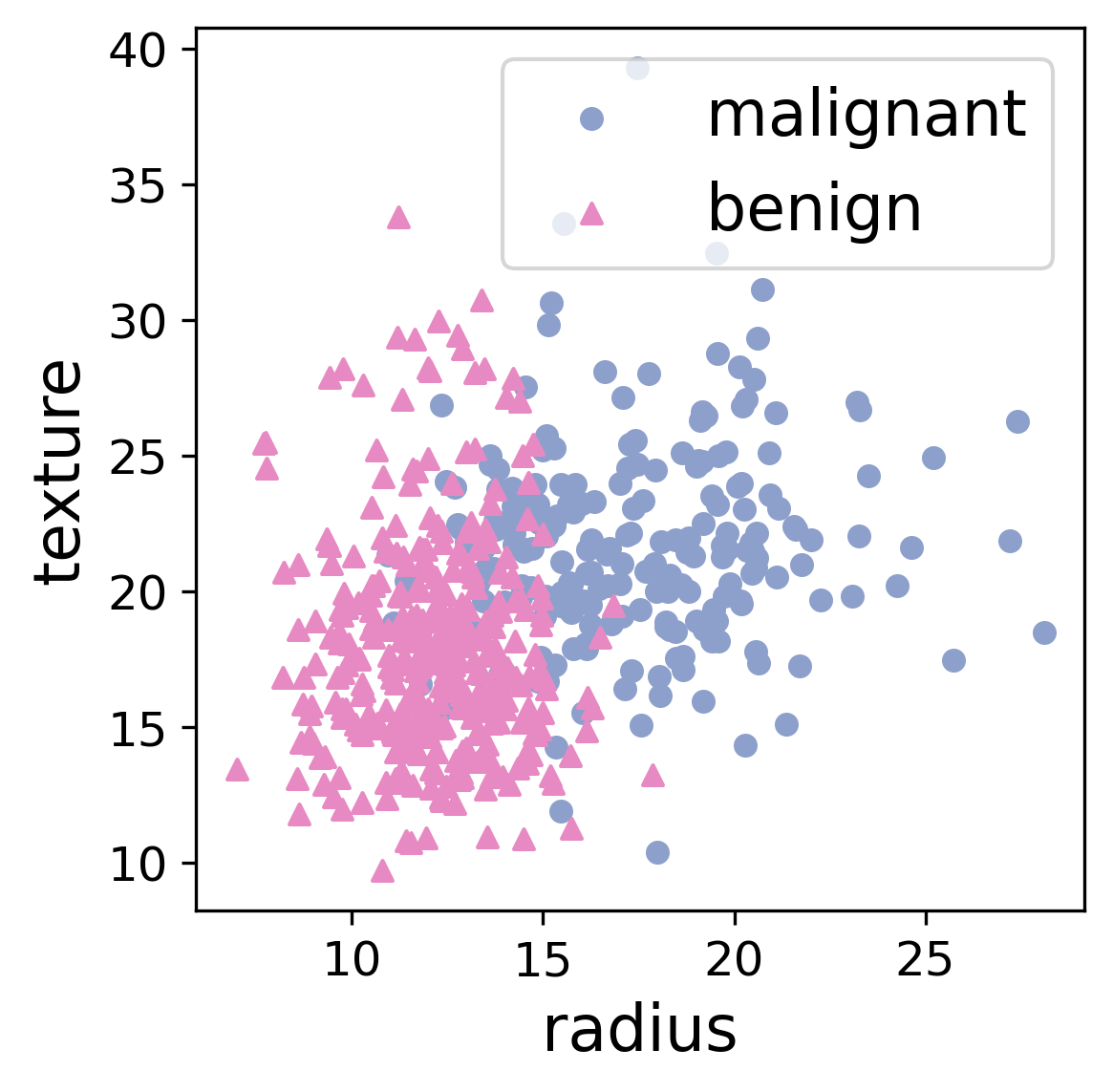}
	\subcaption{Ground truth labels} 
\end{minipage}
\begin{minipage}[t]{0.33\linewidth}
	\centering
	\includegraphics[width=\linewidth]{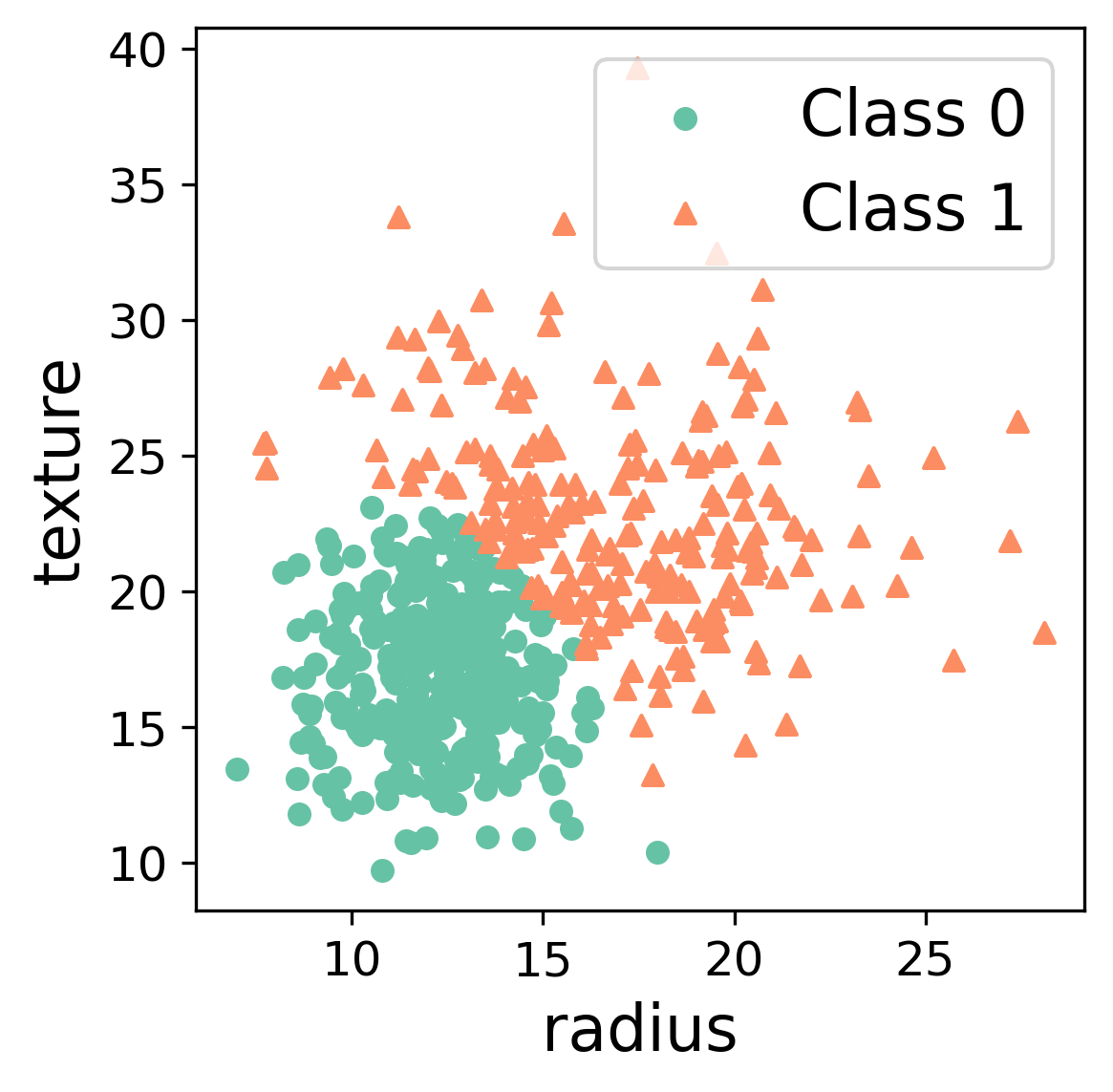}
	\subcaption{MAP clustering} 
\end{minipage}%
\begin{minipage}[t]{0.33\linewidth}
	\centering
	\includegraphics[width=\linewidth]{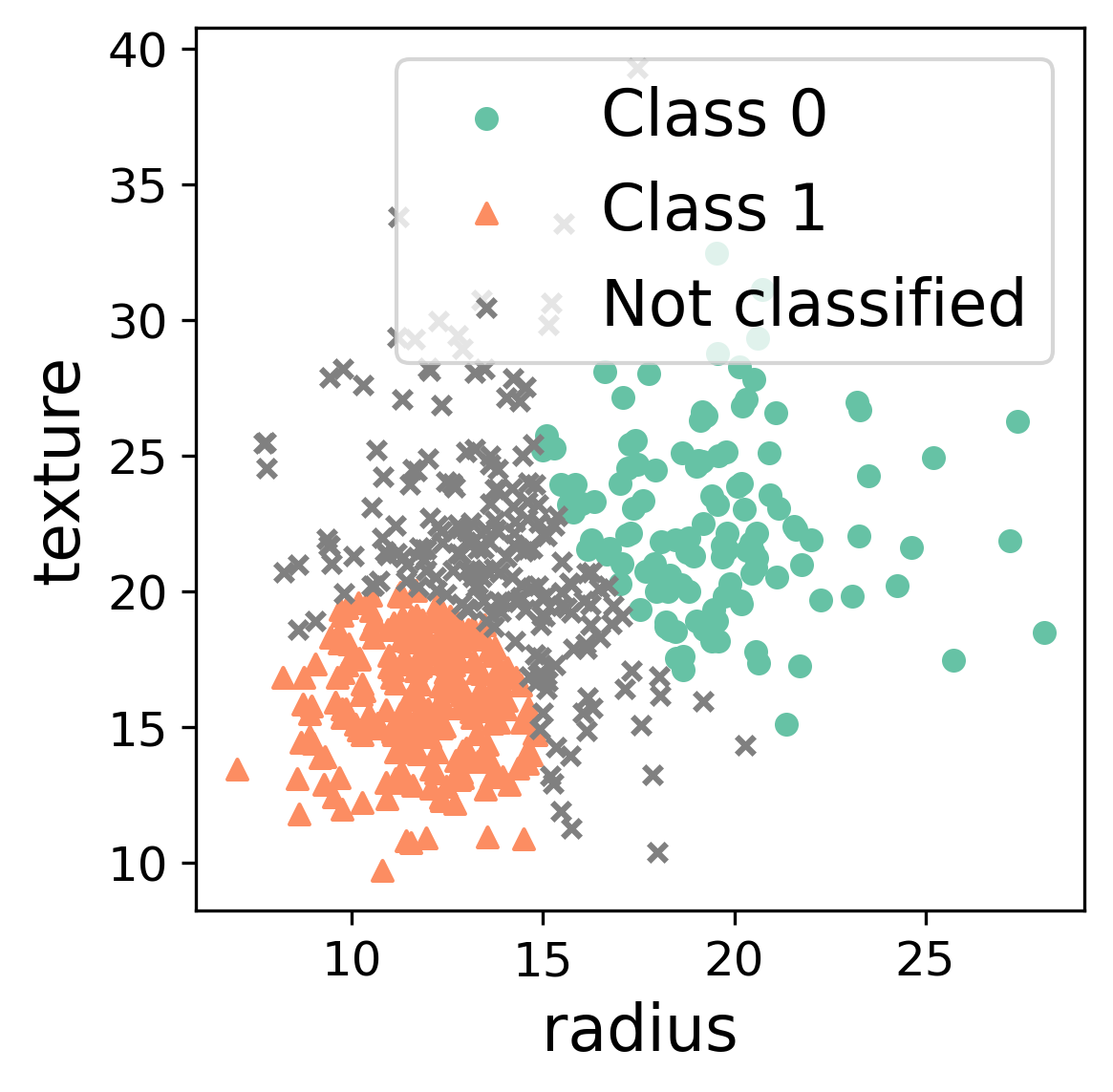}
	\subcaption{Bootstrap procedure} 
	\end{minipage}
\caption{Comparison of the clustering result using t-mixture modelling with ground truth labels on the WDBC dataset, restricted to the variables \textit{radius} and \textit{texture}, with and without selection. With the parametric bootstrap procedure applied at $\alpha = 5\%$, the FMR w.r.t. the ground truth labels is of 3\% versus 14\% without selection. }  \label{FMR:fig:app}
\end{figure}

\begin{figure}
	\begin{minipage}[t]{0.33\linewidth}
		\centering
		\includegraphics[width=\linewidth]{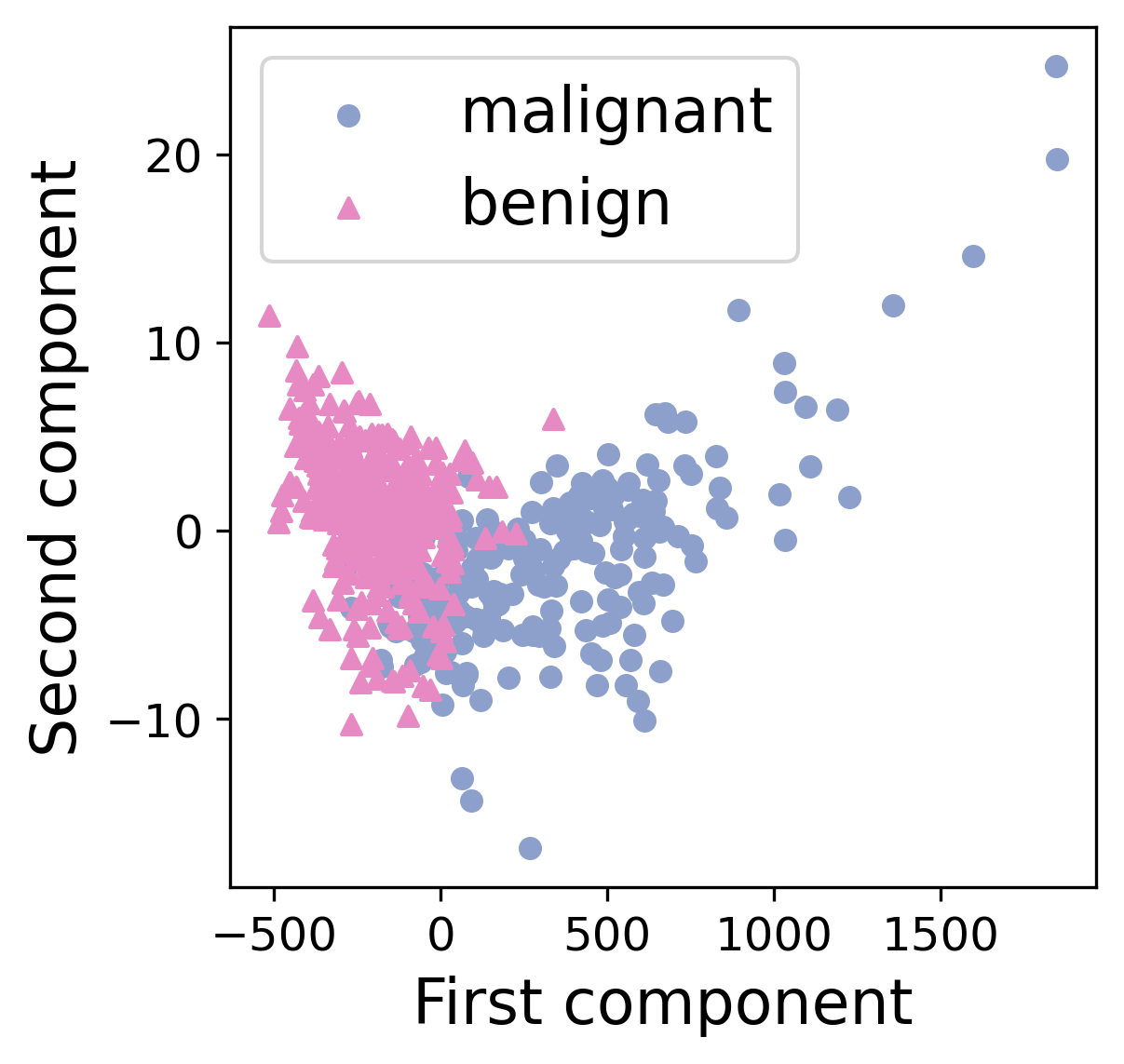}
		\subcaption{Ground truth labels} 
	\end{minipage}
	\begin{minipage}[t]{0.33\linewidth}
		\centering
		\includegraphics[width=\linewidth]{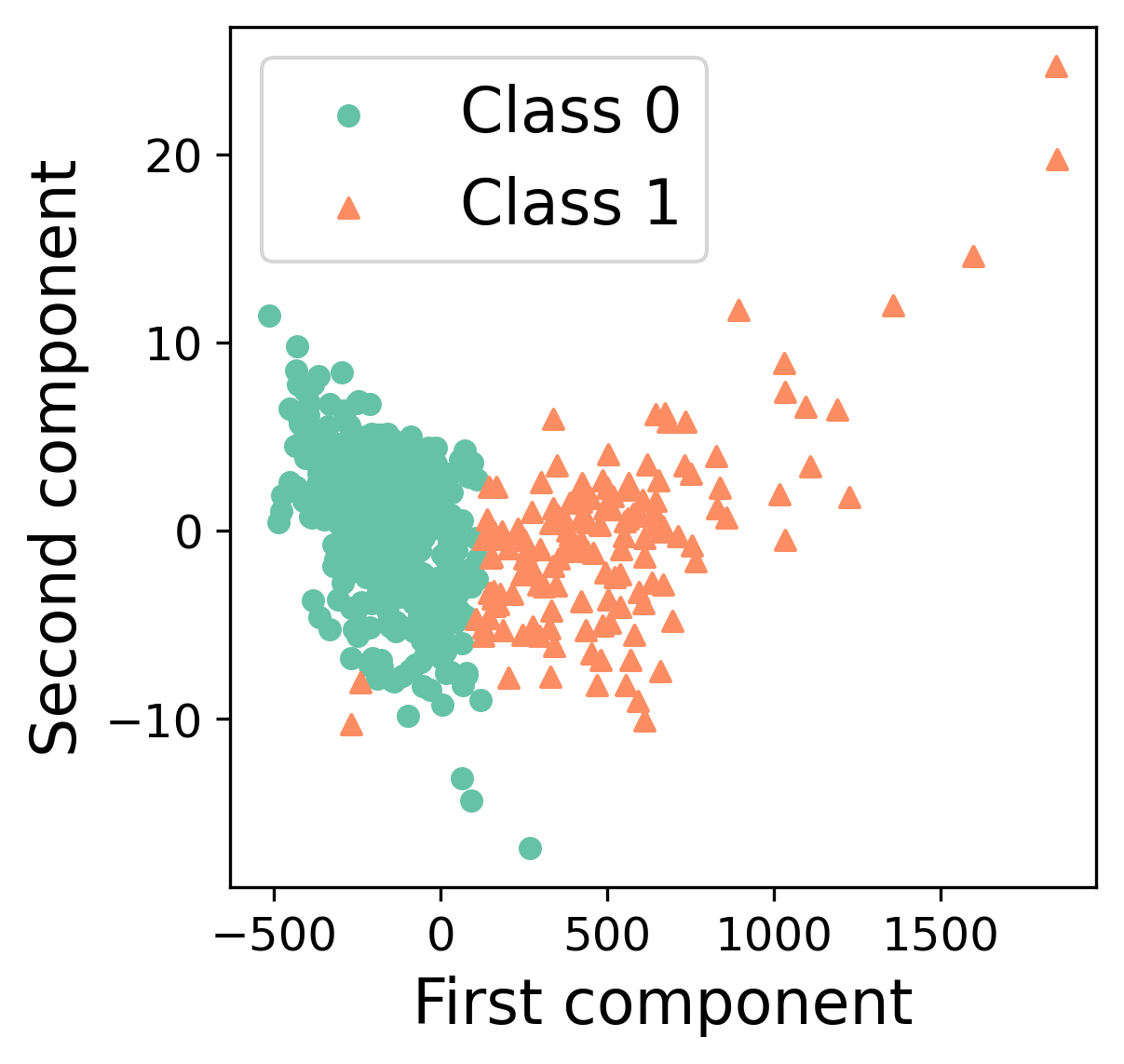}
		\subcaption{MAP clustering} 
	\end{minipage}%
	\begin{minipage}[t]{0.33\linewidth}
		\centering
		\includegraphics[width=\linewidth]{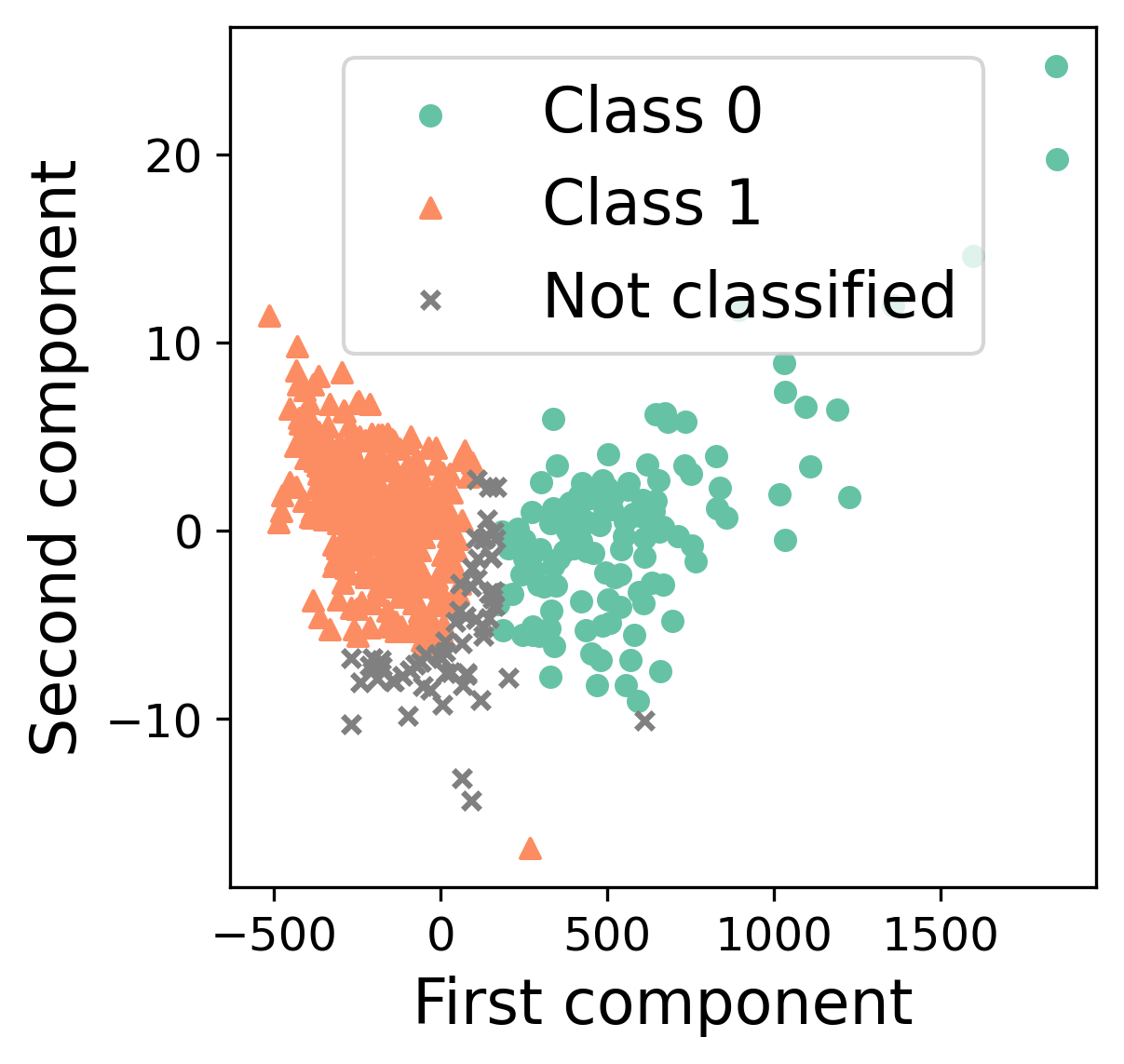}
		\subcaption{Bootstrap procedure} 
	\end{minipage}
	\caption{Comparison of the clustering result using PCA and t-mixture modelling with ground truth labels on the WDBC dataset, restricted to the first ten variables, with and without selection. With the parametric bootstrap procedure applied at $\alpha = 5\%$, the FMR w.r.t. the ground truth labels is of 10\% versus 14\% without selection. }  \label{FMR:fig:app10D}
\end{figure}

Finally, Figure \ref{FMR:fig:app10D} displays the results when restricting the analysis to the first ten variables of the dataset and applying PCA to reduce the dimension to $2$. In that case, the FMR computed with respect to the ground truth labels without selection is 14 \%, while using the bootstrap procedures, this reduces to 10 \%, with a selection frequency of 80\%.

\section{Conclusion and discussion}\label{FMR:sec:discussion}


We have presented new data-driven methods providing both clustering and selection that ensure an FMR control guarantee in a mixture model. The plug-in approach was shown to be theoretically valid both when the parameter estimation is accurate and the sample size is large enough. When this is not necessarily the case, we proposed two second-order bootstrap corrections that have been shown to increase the FMR control ability on numerical experiments.  Finally, applying our unsupervised methods to a supervised data set, our approach has been qualitatively validated by considering the attached labels as revealing the true clusters.

We underline that the cluster number $Q$ is assumed to be fixed and known throughout the study. 
In practice, it can be fitted from the data by using the standard AIC or BIC criteria, using the entire data before application of the selection rule. In addition, if several values of $Q$ make sense from a practical viewpoint, we recommend to provide to the practitioner the collection of the corresponding outputs.

Concerning the pure task of controlling the FMR in the mixture model, our methods provide a correct FMR control in some region of the parameter space, leaving other less favorable parameter configurations with a slight inflation in the FMR level. 
This phenomenon is well known for FDR control in the two-component mixture multiple testing model  \citep{SC2007,roquain2022false}, and facing a similar problem in our framework is well expected. 
On the one hand, in some cases, this problem  can certainly be solved by improving on parameter estimation: here the EM algorithm seems to over-estimate the extreme posterior probabilities, which makes the plug-in procedure too anti-conservative. On the other hand, it could be hopeless to expect a robust FMR control uniformly valid over all configurations, while being optimal in the favorable cases. To illustrate that point, we refer to the work \cite{roquain2022false} that shows that such a procedure does not exist in the FDR controlling case, when the null distribution is Gaussian with an unknown scaling parameter (which is a framework sharing similarities with the one considered here).  Investigating such a ``lower bound'' result in the current setting would provide better guidelines for the practitioner and is therefore an interesting direction for future research.  
In addition, in these unfavorable cases, adding labeled samples and considering a semi-supervised framework can be an appropriate alternative for practical use. This new sample is likely to considerably improve the inference. Studying the FMR control in that setting is another promising avenue.

\section*{Acknowledgments}
This work has been supported by ANR-16-CE40-0019 (SansSouci), ANR-17-CE40-0001 (BASICS) and by the GDR ISIS through the ``projets exploratoires'' program (project TASTY). A. Marandon has been supported by a grant from R\'egion \^Ile-de-France (``DIM Math Innov'').
We would like to thank Gilles Blanchard and St\'ephane Robin for interesting discussions. We are also grateful to Eddie Aamari and Nhat Ho for their help for proving Lemma~\ref{lem:wasserstein}.

\bibliographystyle{apalike}
\bibliography{bibliography}

\begin{thebibliography}{}

\bibitem[Abbe, 2018]{abbe18}
Abbe, E. (2018).
\newblock Community detection and stochastic block models: Recent developments.
\newblock {\em Journal of Machine Learning Research}, 18(177):1--86.

\bibitem[Abraham et~al., 2022]{abraham2022empirical}
Abraham, K., Castillo, I., and Roquain, {\'E}. (2022).
\newblock Empirical bayes cumulative $\ell$-value multiple testing procedure
  for sparse sequences.
\newblock {\em Electronic Journal of Statistics}, 16(1):2033--2081.

\bibitem[Angelopoulos et~al., 2021]{angelopoulos2021learn}
Angelopoulos, A.~N., Bates, S., Cand{\`{e}}s, E.~J., Jordan, M.~I., and Lei, L.
  (2021).
\newblock Learn then test: Calibrating predictive algorithms to achieve risk
  control.
\newblock {\em CoRR}, abs/2110.01052.

\bibitem[Arora and Kannan, 2005]{arora05}
Arora, S. and Kannan, R. (2005).
\newblock {Learning mixtures of separated nonspherical Gaussians}.
\newblock {\em The Annals of Applied Probability}, 15(1A):69 -- 92.

\bibitem[Arthur and Vassilvitskii, 2006]{arthur2006k}
Arthur, D. and Vassilvitskii, S. (2006).
\newblock k-means++: The advantages of careful seeding.
\newblock Technical report, Stanford.

\bibitem[Azizyan et~al., 2013]{azizyan13}
Azizyan, M., Singh, A., and Wasserman, L. (2013).
\newblock Minimax theory for high-dimensional gaussian mixtures with sparse
  mean separation.
\newblock In {\em Proceedings of the 26th International Conference on Neural
  Information Processing Systems - Volume 2}, NIPS'13, page 2139–2147, Red
  Hook, NY, USA. Curran Associates Inc.

\bibitem[Balakrishnan et~al., 2017]{balakrishnan17}
Balakrishnan, S., Wainwright, M.~J., and Yu, B. (2017).
\newblock {Statistical guarantees for the EM algorithm: From population to
  sample-based analysis}.
\newblock {\em The Annals of Statistics}, 45(1):77 -- 120.

\bibitem[Baraud, 2016]{baraud2016bounding}
Baraud, Y. (2016).
\newblock Bounding the expectation of the supremum of an empirical process over
  a (weak) vc-major class.
\newblock {\em Electronic journal of statistics}, 10(2):1709--1728.

\bibitem[Bartlett and Wegkamp, 2008]{bartlett08}
Bartlett, P.~L. and Wegkamp, M.~H. (2008).
\newblock Classification with a reject option using a hinge loss.
\newblock {\em Journal of Machine Learning Research}, 9(59):1823--1840.

\bibitem[Benjamini and Hochberg, 1995]{BH1995}
Benjamini, Y. and Hochberg, Y. (1995).
\newblock Controlling the false discovery rate: a practical and powerful
  approach to multiple testing.
\newblock {\em J. Roy. Statist. Soc. Ser. B}, 57(1):289--300.

\bibitem[Cai et~al., 2019]{CSWW2019}
Cai, T., Sun, W., and Wang, W. (2019).
\newblock Covariate-assisted ranking and screening for large-scale two-sample
  inference.
\newblock {\em Journal of the Royal Statistical Society: Series B (Statistical
  Methodology)}, 81(2):187--234.

\bibitem[Chow, 1970]{chow70}
Chow, C. (1970).
\newblock On optimum recognition error and reject tradeoff.
\newblock {\em IEEE Transactions on Information Theory}, 16(1):41--46.

\bibitem[Chretien et~al., 2019]{chretien19}
Chretien, S., Dombry, C., and Faivre, A. (2019).
\newblock {The Guedon-Vershynin Semi-Definite Programming approach to low
  dimensional embedding for unsupervised clustering}.
\newblock {\em {Frontiers in Applied Mathematics and Statistics}}.

\bibitem[Dempster et~al., 1977]{dempster77}
Dempster, A.~P., Laird, N.~M., and Rubin, D.~B. (1977).
\newblock Maximum likelihood from incomplete data via the em algorithm.
\newblock {\em Journal of the Royal Statistical Society. Series B
  (Methodological)}, 39(1):1--38.

\bibitem[Denis and Hebiri, 2020]{denis20}
Denis, C. and Hebiri, M. (2020).
\newblock Consistency of plug-in confidence sets for classification in
  semi-supervised learning.
\newblock {\em Journal of Nonparametric Statistics}, 32(1):42--72.

\bibitem[Efron et~al., 2001]{ETST2001}
Efron, B., Tibshirani, R., Storey, J.~D., and Tusher, V. (2001).
\newblock Empirical bayes analysis of a microarray experiment.
\newblock {\em Journal of the American Statistical Association},
  96(456):1151--1160.

\bibitem[Geifman and El-Yaniv, 2017]{geifman17}
Geifman, Y. and El-Yaniv, R. (2017).
\newblock Selective classification for deep neural networks.
\newblock In {\em Proceedings of the 31st International Conference on Neural
  Information Processing Systems}, NIPS'17, page 4885–4894, Red Hook, NY,
  USA. Curran Associates Inc.

\bibitem[Giraud and Verzelen, 2018]{giraud18}
Giraud, C. and Verzelen, N. (2018).
\newblock {Partial recovery bounds for clustering with the relaxed $K$-means}.
\newblock {\em {Mathematical Statistics and Learning}}, 1(3):317--374.

\bibitem[Herbei and Wegkamp, 2006]{herbei06}
Herbei, R. and Wegkamp, M.~H. (2006).
\newblock Classification with reject option.
\newblock {\em The Canadian Journal of Statistics / La Revue Canadienne de
  Statistique}, 34(4):709--721.

\bibitem[Ho and Nguyen, 2016]{ho16}
Ho, N. and Nguyen, X. (2016).
\newblock {On strong identifiability and convergence rates of parameter
  estimation in finite mixtures}.
\newblock {\em Electronic Journal of Statistics}, 10(1):271 -- 307.

\bibitem[Lei and Rinaldo, 2015]{LR2015}
Lei, J. and Rinaldo, A. (2015).
\newblock {Consistency of spectral clustering in stochastic block models}.
\newblock {\em The Annals of Statistics}, 43(1):215 -- 237.

\bibitem[Lu and Zhou, 2016]{lu2016statistical}
Lu, Y. and Zhou, H.~H. (2016).
\newblock Statistical and computational guarantees of lloyd's algorithm and its
  variants.
\newblock {\em arXiv preprint arXiv:1612.02099}.

\bibitem[Mary-Huard et~al., 2021]{maryhuard21}
Mary-Huard, T., Perduca, V., Blanchard, G., and Marie-Laure, M.-M. (2021).
\newblock Error rate control for classification rules in multiclass mixture
  models.

\bibitem[Massart, 2007]{massart07}
Massart, P. (2007).
\newblock {\em Concentration Inequalities and Model Selection Ecole d'Eté de
  Probabilités de Saint-Flour XXXIII - 2003}.
\newblock École d'Été de Probabilités de Saint-Flour, 1896. Springer Berlin
  Heidelberg, Berlin, Heidelberg, 1st ed. 2007. edition.

\bibitem[Melnykov, 2013]{melnykov2013distribution}
Melnykov, V. (2013).
\newblock On the distribution of posterior probabilities in finite mixture
  models with application in clustering.
\newblock {\em Journal of Multivariate Analysis}, 122:175--189.

\bibitem[Mohri et~al., 2012]{FoundationsMLbook12}
Mohri, M., Rostamizadeh, A., and Talwalkar, A. (2012).
\newblock {\em Foundations of Machine Learning}.
\newblock The MIT Press.

\bibitem[Najafi et~al., 2020]{najafi20}
Najafi, A., Motahari, S.~A., and Rabiee, H.~R. (2020).
\newblock {Reliable clustering of Bernoulli mixture models}.
\newblock {\em Bernoulli}, 26(2):1535 -- 1559.

\bibitem[Peel and McLachlan, 2000]{peel00}
Peel, D. and McLachlan, G.~J. (2000).
\newblock Robust mixture modelling using the t distribution.
\newblock {\em Statistics and Computing}, 10(4):339--348.

\bibitem[Rebafka et~al., 2022]{rebafka22}
Rebafka, T., Roquain, {\'E}., and Villers, F. (2022).
\newblock {Powerful multiple testing of paired null hypotheses using a latent
  graph model}.
\newblock {\em Electronic Journal of Statistics}, 16(1):2796 -- 2858.

\bibitem[Regev and Vijayaraghavan, 2017]{regev17}
Regev, O. and Vijayaraghavan, A. (2017).
\newblock On learning mixtures of well-separated gaussians.
\newblock In {\em 2017 IEEE 58th Annual Symposium on Foundations of Computer
  Science (FOCS)}, pages 85--96.

\bibitem[Roquain and Verzelen, 2022]{roquain2022false}
Roquain, E. and Verzelen, N. (2022).
\newblock False discovery rate control with unknown null distribution: Is it
  possible to mimic the oracle?
\newblock {\em The Annals of Statistics}, 50(2):1095--1123.

\bibitem[Shalev-Shwartz and Ben-David, 2014]{UnderstandingMLbook14}
Shalev-Shwartz, S. and Ben-David, S. (2014).
\newblock {\em Understanding Machine Learning: From Theory to Algorithms}.
\newblock Cambridge University Press, USA.

\bibitem[Storey, 2003]{storey2003positive}
Storey, J.~D. (2003).
\newblock The positive false discovery rate: a bayesian interpretation and the
  q-value.
\newblock {\em The Annals of Statistics}, 31(6):2013--2035.

\bibitem[Sun and Cai, 2007]{SC2007}
Sun, W. and Cai, T.~T. (2007).
\newblock Oracle and adaptive compound decision rules for false discovery rate
  control.
\newblock {\em Journal of the American Statistical Association},
  102(479):901--912.

\bibitem[Vempala and Wang, 2004]{vempala04}
Vempala, S. and Wang, G. (2004).
\newblock A spectral algorithm for learning mixture models.
\newblock {\em Journal of Computer and System Sciences}, 68(4):841--860.
\newblock Special Issue on FOCS 2002.

\bibitem[Wegkamp and Yuan, 2011]{wegkamp2011support}
Wegkamp, M. and Yuan, M. (2011).
\newblock Support vector machines with a reject option.
\newblock {\em Bernoulli}, 17(4):1368--1385.

\end{thebibliography}

\newpage


\appendix



\section{Proof of Theorems \ref{consistency} and \ref{FMR:th:convergence_rates}}\label{FMR:sec:proofmain}

\subsection{A general result}

In this section, we establish a general result, from which Theorems~\ref{consistency}~and~\ref{FMR:th:convergence_rates} can be deduced.
	It provides non-asymptotic bounds on the mFMR and the FMR of the plug-in procedure and on its average selection number, by relying only on Assumption~\ref{Tcontinuity}. To state the result, we introduce some additional quantities measuring the regularity of the model which will appear in our remainder terms. 
	Recall definitions \eqref{equlvalue}, \eqref{equTXtheta} and \eqref{tstaralpha} of $\ell_q(X, \theta)$, $T(X,\theta)$ and $t^*(\alpha)$ respectively, and let for $\epsilon,\delta,v>0$,
	\begin{align}
		\mathcal{W}_{\ell}(\epsilon) &= \sup \left\{ \underset{x \in \mathbb{R}^d}{\sup} \left[\max_{1\leq q\leq Q} |  \ell_{q}(x,\theta^*)- \ell_{q}(x,\theta)| \right]\:,\: \| \theta - \theta^*\| _2 \leq \epsilon, \: \theta \in \Theta \right\}; \label{Wellbound}\\
		\mathcal{W}_T(\delta) &= \sup\{ |\Pstar (T(X, \theta^*) < t') -\Pstar (T(X, \theta^*) < t)|\:,  
		\\ &\hspace{2cm} t,t'\in [0,1], |t'-t|\leq \delta \}; \label{WTbound}\\
		\Psi(\epsilon)&=\mathcal{W}_T(\mathcal{W}_{\ell}(\epsilon)^{1/2}) + \mathcal{W}_{\ell}(\epsilon)^{1/2} ; \label{Psibound} \\
		\mathcal{W}_{t^*,\alpha}(v)&= \sup \left\{  |t^*(\alpha+\beta)-t^*(\alpha)|\:,\: |\beta|  \leq v\right\}. \label{Wtstarbound}
	\end{align}

\begin{theorem} \label{result} 
	Let Assumption~\ref{Tcontinuity} be true. For any $\alpha \in (\alpha_c, \bar{\alpha})$ and constants $s^*=s^*(\alpha,\theta^*)\in (0,1)$ and $e^*=e(\alpha,\theta^*)>0$ depending only on $\alpha$ and $\theta^*$, the following holds.
	Consider the plug-in procedure $\widehat{\proc}^{\mbox{\tiny PI}}_\alpha=(\widehat{\mathbf{Z}}^{\mbox{\tiny PI}}, \widehat{S}^{\mbox{\tiny PI}}_\alpha)$ introduced in Algorithm~\ref{algo:pi} and based on an estimator $\hat{\theta}$ satisfying Assumption~\ref{ass_estim}, with $\eta(\epsilon,\theta^*)$ defined by \eqref{equeta}.
	Then for $\epsilon\leq e^*$ and $n\geq (2e)^3$, letting
	\begin{align*}
		\Delta_n(\epsilon)=2 \left(\mathcal{W}_T( \mathcal{W}_{t^*,\alpha}(2\delta_n + 8\Psi(\epsilon)/s^*))+ 4\Psi(\epsilon)+ 2\delta_n\right),
	\end{align*}
	for $\delta_n=C\sqrt{(\log n)/n}/s^*$ where $C=2+56Q\sqrt{\VCdim} + 28Q^2\sqrt{\VCdim_{-}}$ and with the quantities $\mathcal{W}_T$, $\mathcal{W}_{\ell}$, $\Psi$, $\mathcal{W}_{t^*,\alpha}$ defined by \eqref{WTbound}, \eqref{Wellbound}, \eqref{Psibound}, \eqref{Wtstarbound}, respectively, it holds:
	\begin{itemize}
		\item The procedure $\widehat{\proc}^{\mbox{\tiny PI}}_\alpha$  controls both the FMR and the mFMR at level close to $\alpha$  in the following sense:
		\begin{align*}
			&\FMR(\widehat{\proc}^{\mbox{\tiny PI}}_{\alpha})\leq \alpha  + \Delta_n(\epsilon)/s^* + 5/n^2 + \eta(\epsilon,\theta^*); \\
			&\mFMR(\widehat{\proc}^{\mbox{\tiny PI}}_{\alpha})\leq \alpha  + \Delta_n(\epsilon)/s^* + s^{*-1}\left[50/n^2 + 10 \eta(\epsilon,\theta^*)\right].
		\end{align*}
		\item The procedure $\widehat{\proc}^{\mbox{\tiny PI}}_\alpha$ is nearly optimal in the following sense:
		for any other procedure $\proc = (\widehat{\mathbf{Z}}, S) $ that controls the mFMR at level $\alpha$, 
		\begin{align*}
			&n^{-1}\esp_{\theta^*}(|\widehat{S}^{\mbox{\tiny PI}}_{\alpha}|)\geq n^{-1} \esp_{\theta^*}(|S|) - \Delta_n(\epsilon).
		\end{align*}
	\end{itemize}
\end{theorem}

Before proving this result (which will be done in the next subsections), let us first show that Theorem~\ref{result} implies Theorems~\ref{consistency}~and~\ref{FMR:th:convergence_rates}.

\paragraph{Proof of Theorem~\ref{consistency}}
By Lemma~\ref{lemconsistW} below,  $\Delta_n(\epsilon)$ tends to $0$ when $n$ tends to infinity and $\epsilon$ tends to $0$. Moreover, by consistency of $\hat{\theta}$,  $\eta(\epsilon,\theta^*)$ tends to $0$ for all $\epsilon>0$. This implies the result.

\begin{lemma}\label{lemconsistW}
	Under Assumption~\ref{Tcontinuity}, we have $\lim_{\delta\to 0} \mathcal{W}_T(\delta) = 0$, $\lim_{v\to 0} \mathcal{W}_{t^*,\alpha}(v)=0$. Under Assumption~\ref{densitycontinuity}, we have $\lim_{\epsilon\to 0} \mathcal{W}_{\ell}(\epsilon) = 0$. Under both assumptions, we have $\lim_{\epsilon\to 0} \Psi(\epsilon) = 0$.
\end{lemma}

\begin{proof}
	The only non-trivial fact is for $\mathcal{W}_{t^*,\alpha}(v)$. Assumption~\ref{Tcontinuity} and Lemma~\ref{mFMR_properties} provide that 
	$t\mapsto \mFMR^*_t$ is a one-to-one continuous increasing map from $(t^*(\alpha_c), t^*(\bar{\alpha}))$ to $(\alpha_c, \bar{\alpha})$.  Hence, for $\alpha\in  (\alpha_c, \bar{\alpha})$, $\beta\mapsto t^*(\alpha+\beta)$ is continuous in $0$ and $\lim_{v\to 0} \mathcal{W}_{t^*,\alpha}(v)=0$.
\end{proof}

\paragraph{Proof of Theorem~\ref{FMR:th:convergence_rates}}
By using Assumption~\ref{lip} (with the notation therein) and Lemma~\ref{lemrateW} below, we have
\begin{align*}
	\Delta_n(\epsilon) 
	&=2 \left(\mathcal{W}_T( \mathcal{W}_{t^*,\alpha}(2\delta_n + 8\Psi(\epsilon)/s^*))+ 4\Psi(\epsilon)+ 2\delta_n\right)\\
	&\leq 2C_2 C_3 \left(2 \delta_n + (8/s^*) \sqrt{C_1}(C_2 + 1 )\sqrt{\epsilon}\right)
	+ 8 \sqrt{C_1}(C_2 + 1 )\sqrt{\epsilon} + 4 \delta_n \\
	&= 8\sqrt{C_1}(C_2+1) (1+2C_2 C_3) \sqrt{\epsilon}/s^*  + 4(C_2 C_3+1)C\sqrt{\log n / n}/s^* ,
\end{align*}
because $s^*\leq 1$ and by definition of $\delta_n$. 
This gives \eqref{FMR_bound} and \eqref{sr_bound} with $A = 8\sqrt{C_1}(C_2+1) (1+2C_2 C_3)/s^{*2}$ and $B= 4(C_2 C_3+1)C/s^{*2}$. 

\begin{lemma}\label{lemrateW}
Under Assumption~\ref{lip}, we have $ \mathcal{W}_{\ell}(\epsilon) \leq C_1 \epsilon$, $\mathcal{W}_T(\delta) \leq C_2 \delta$, $\mathcal{W}_{t^*,\alpha}(v)\leq C_3 v$ and $\Psi(\epsilon)\leq \sqrt{C_1} (C_2+1)  \sqrt{\epsilon}$  for $\epsilon,\delta,v$ small enough.
\end{lemma}

\subsection{An optimal procedure}\label{FMR:sec:optimalproc}

We consider in this section the procedure that serves as an optimal procedure in our theory. 
For $t\in [0,1]$, let $\proc^*_t = (\widehat{\mathbf{Z}}^*, S^*_t)$ be the procedure using the Bayes clustering $\widehat{\mathbf{Z}}^*$ \eqref{equ:Bayesclustering} and the selection rule $S^*_t = \{i\in \{1,\dots,n\}\::\: T^*_i < t \}$. 
Let us consider the map $t\in [0,1]\mapsto \mFMR(\proc^*_t)$ and note that $ \mFMR(\proc^*_t)=\mFMR^*_t$ as defined by \eqref{mFMRstart}.
Lemma~\ref{mFMR_properties} below provides the key properties for this function. 

\begin{definition}
The optimal procedure at level $\alpha$ is defined by $\proc^*_{t^*(\alpha)}$ where $t^*(\alpha)$ is defined by \eqref{tstaralpha}.
\end{definition}
Note that the optimal procedure is not the same as the oracle procedure defined in Section~\ref{FMR:sec:oracle}, although these two procedures are expected to behave roughly in the same way (at least for a large $n$).

Under  Assumption \ref{Tcontinuity}, Lemma~\ref{mFMR_properties} entails that, for $\alpha>\alpha_c$,  $\mFMR(\proc^*_{t^*(\alpha)}) \leq \alpha$. Hence, $\proc^*_{t^*(\alpha)}$ controls the mFMR at level $\alpha$. In addition, it is optimal in the following sense:  any other mFMR controlling procedure should select less items than $\proc^*_{t^*(\alpha)}$.

\begin{lemma}[Optimality of $\proc^*_{t^*(\alpha)}$] \label{oracle_optimality}
Let Assumption \ref{Tcontinuity} be true and choose $\alpha \in (\alpha_c, \bar{\alpha}]$. Then the oracle procedure $\proc^*_{t^*(\alpha)}= (\widehat{\mathbf{Z}}^*, S^*_{t^*(\alpha)})$ satisfies the following:
\begin{itemize}
\item[(i)] $\mFMR(\proc^*_{t^*(\alpha)}) = \alpha$;
\item[(ii)]  for any procedure $\proc = (\widehat{\mathbf{Z}}, S) $ such that $\mFMR(\proc) \leq \alpha$, we have
$ \esp_{\theta^*}(|S|) \leq \esp_{\theta^*}(|S^*_{t^*(\alpha)}|).$
\end{itemize}
\end{lemma}

\subsection{Preliminary steps for proving Theorem~\ref{result}}

To keep the main proof concise, we need to define several additional notation. Let for $t\in [0,1]$ and $\theta\in \Theta$ (recall \eqref{equTXtheta})
\begin{align}
\mathbf{\widehat{L}}_0(\theta, t) &= \frac{1}{n} \sum_{i=1}^n T(X_i, \theta) \ind_{T(X_i, \theta) < t}; \label{L0chap}\\
\mathbf{\widehat{L}}_1(\theta, t) &= \frac{1}{n} \sum_{i=1}^n \ind_{T(X_i, \theta) < t}\label{L1chap}.
\end{align} 
Denote $\mathbf{\widehat{L}}= \mathbf{\widehat{L}}_0 / \mathbf{\widehat{L}}_1$,  $\mathbf{L}_0=\esp_{\theta^*}\mathbf{\widehat{L}}_0$, $\mathbf{L}_1=\esp_{\theta^*}\mathbf{\widehat{L}}_1$, $\mathbf{L}=\mathbf{L}_0 / \mathbf{L}_1$ (with the convention $0/0 = 0$). 
Note that for any $\alpha>\alpha_c$, the mFMR of the optimal procedure $\proc^*_{t^*(\alpha)}$ defined in Section~\ref{FMR:sec:optimalproc} is given by
$
\mFMR(\proc^*_{t^*(\alpha)})=\mathbf{L}(\theta^*, t^*(\alpha))=\alpha.
$

Also, we denote from now on $\ell_{i,q}^* = \Pstar(Z_i = q | X_i)$ for short and introduce for any parameter $\theta\in \Theta$ (recall \eqref{equlvalue} and \eqref{equTXtheta})
\begin{align}
\bar{q}(X_i, \theta) &\in \underset{q \in \{1, \dots, Q \}}{\argmax}  \ell_q(X_i, \theta),\quad 1\leq i\leq n;   \label{equqbar} \\
U(X_i, \theta) &= 1 - \ell^*_{i, \bar{q}(X_i, \theta)}, \quad 1\leq i\leq n;\label{equU}\\
\mathbf{\widehat{M}}_0(\theta, t) &= \frac{1}{n} \sum_{i=1}^n U(X_i, \theta) \ind_{T(X_i, \theta) < t},\quad t\in [0,1],\label{equM0}
\end{align}
Note that $\mathbf{\widehat{M}}_0(\theta^*, t)=\mathbf{\widehat{L}}_0(\theta^*, t)$ but in general $\mathbf{\widehat{M}}_0(\theta, t)$ is different from $\mathbf{\widehat{L}}_0(\theta, t)$. We denote $\mathbf{\widehat{M}}=\mathbf{\widehat{M}}_0 / \mathbf{\widehat{L}}_1 $, $\mathbf{M}_0= \esp_{\theta^*}\mathbf{\widehat{M}}_0$ and $\mathbf{M}=\mathbf{M}_0 / \mathbf{L}_1$ (with the convention $0/0 = 0$).

When $\alpha  \in (\alpha_c, \bar{\alpha}]$ (recall \eqref{alphac} and \eqref{alphabar}), we also let 
\begin{equation}\label{equsstar}
s^*=s^*(\alpha,\theta^*) = n^{-1} \esp_{\theta^*} \left( |S^*_{t^*(\frac{\alpha + \alpha_c}{2})}| \right) = \mathbf{L}_1(\theta^*, t^*((\alpha + \alpha_c)/2))>0.
\end{equation}
We easily see that the latter is positive: if it was zero then $S^*_{t^*((\alpha + \alpha_c)/2))}$ would be empty  which would entails that $\mFMR(\proc^*_{t^*((\alpha + \alpha_c)/2)})$ is zero. This is excluded by definition \eqref{alphac} of $\alpha_c$ because $(\alpha + \alpha_c)/2>\alpha_c$.

Also, we are going to extensively use the event
\begin{align*}
\Omega_\epsilon &= \left\{ \underset{\sigma \in [Q]}{\min} \|\hat{\theta}^\sigma - \theta^*\|_2 < \epsilon \right\}.  
\end{align*}
On this event, we fix any permutation $\sigma \in [Q]$  (possibly depending on $X$) such that $\|\hat{\theta}^\sigma - \theta^*\|_2 < \epsilon$. 
Now using Lemma~\ref{lempluginthres}, the plug-in selection rule can be rewritten as 
$\widehat{S}^{\mbox{\tiny PI}}_\alpha =\{ i\in \{1,\dots,n\}\::\: \hat{T}_i < \hat{t}(\alpha) \}$ (denoted by $\widehat{S}$ in the sequel for short),
where
\begin{align}\label{thatalpha}
\hat{t}(\alpha) = \sup \{ t \in [0,1]\::\: \mathbf{\widehat{L}}(\hat{\theta}, t) \leq \alpha \}.
\end{align}


With the above notation, we can upper bound what is inside the brackets of $\FMR(\widehat{\proc}^{\mbox{\tiny PI}})$ and $\mFMR(\widehat{\proc}^{\mbox{\tiny PI}})$ as follows.

\begin{lemma}\label{lem:sigmaprecise}
For the permutation $\sigma$ in $\Omega_\epsilon$ realizing $\|\hat{\theta}^\sigma - \theta^*\|_2 < \epsilon$, we have on the event  $\Omega_\epsilon$ the following relations:
\begin{align*}
|\widehat{S}|&=\mathbf{\widehat{L}}_1(\hat{\theta}^\sigma, \hat{t}(\alpha)) ; \\
\underset{\sigma' \in [Q]}{\min}  \esp_{\theta^*} \left(\varepsilon_{\widehat{S}}(\sigma'(\widehat{\mathbf{Z}}) ,\mathbf{Z})\: \bigg|\: \mathbf X\right) 
&\leq \mathbf{\widehat{M}}_0(\hat{\theta}^\sigma, \hat{t}(\alpha));\\
\underset{\sigma' \in [Q]}{\min} \esp_{\theta^*} \left( \frac{\varepsilon_{\widehat{S}}(\sigma'(\widehat{\mathbf{Z}}) ,\mathbf{Z})}{\max(|\widehat{S}|,1)} \: \bigg|\: \mathbf X\right) &\leq  \mathbf{\widehat{M}}(\hat{\theta}^\sigma, \hat{t}(\alpha)) .
\end{align*}
\end{lemma}

Finally, we make use of the concentration of the empirical processes $\mathbf{\widehat{L}}_0(\theta, t)$, $\mathbf{\widehat{L}}_1(\theta, t)$, and $\mathbf{\widehat{M}}_0(\theta, t)$, uniformly with respect to $\theta \in \subsetEsti$ (where $\subsetEsti$ is defined in Assumption~\ref{ass_estim}). Thus, we define the following events, for $\delta>0$ (recall $s^*$ defined by \eqref{equsstar}):
\begin{align*}
\Gamma_{0, \delta, t} &= \left \{ \underset{\theta \in  \subsetEsti}{\sup} \left| \mathbf{\widehat{L}}_0(\theta, t) - \mathbf{L}_0(\theta, t) \right|  \leq \delta \right \}; \\
\Gamma_{1, \delta, t} &= \left \{ \underset{\theta \in \subsetEsti}{\sup} \left| \mathbf{\widehat{L}}_1(\theta, t) - \mathbf{L}_1(\theta, t) \right|  \leq \delta \right \}; \\
\Gamma_{\delta, t} &= \left \{ \underset{\substack{\theta \in \subsetEsti, \\ \mathbf{L}_1(\theta, t) \geq s^* }}{\sup} \left| \mathbf{\widehat{L}}(\theta, t) - \mathbf{L}(\theta, t) \right|  \leq \delta \right \};\\
\Upsilon_{0, \delta, t} &= \left\{  \underset{\theta \in \subsetEsti}{\sup} \left| \mathbf{\widehat{M}}_0(\theta, t) - \mathbf{M}_0(\theta, t) \right|  \leq \delta \right \}.
\end{align*}
Note that the following holds:
\begin{equation}\label{inclusiongamma}
 \Gamma_{0, \delta s^*/2, t} \cap \Gamma_{1, \delta s^*/2, t}\subset \Gamma_{\delta, t}.
 \end{equation}
Indeed, on the event $\Gamma_{0, \delta s^*/2, t} \cap \Gamma_{1, \delta s^*/2, t}$, provided that $\mathbf{L}_1(\theta, t) \geq s^*$, we have
\begin{align*}
&\left| \frac{\mathbf{\widehat{L}}_0(\theta, t)}{\mathbf{\widehat{L}}_1(\theta, t)} - \frac{\mathbf{L}_0(\theta, t)}{\mathbf{L}_1(\theta, t)} \right| \\
&\leq \left| \frac{\mathbf{L}_0(\theta, t) - \mathbf{\widehat{L}}_0(\theta, t)}{\mathbf{L}_1(\theta, t)} \right| + \mathbf{\widehat{L}}_0(\theta, t) \left| \frac{1}{\mathbf{\widehat{L}}_1(\theta, t)} - \frac{1}{\mathbf{L}_1(\theta, t)} \right| \\
&\leq (\delta s^*/2)/ s^* +(\delta s^*/2)/ s^*  = \delta,
\end{align*}
because $\mathbf{\widehat{L}}_0(\theta, t)\leq \mathbf{\widehat{L}}_1(\theta, t)$. This proves the desired inclusion.
%

\subsection{Proof of Theorem~\ref{result} }


Let us now provide a proof for Theorem~\ref{result}.

\paragraph{Step 1: bounding $\hat{t}(\alpha)$  w.r.t. $t^*(\alpha)$}
Recall \eqref{tstaralpha}, \eqref{thatalpha} and \eqref{equsstar}.
In this part, we only consider realizations on the event $\Omega_\epsilon$. Let $\beta \in [\frac{2\alpha + \alpha_c}{3}, \frac{\alpha + \bar\alpha}{2}]$. 
By Lemma~\ref{lemcdftheta}, we have 
$$
\mathbf{L}_1(\hat{\theta}^\sigma, t^*(\beta)) \geq \mathbf{L}_1(\theta^*, t^*(\beta)) - \Psi(\| \hat{\theta}^\sigma - \theta^*\| _2)\geq  \mathbf{L}_1(\theta^*, t^*((2\alpha + \alpha_c)/3))- \Psi(\epsilon),$$
because $t^*(\beta)\geq t^*(\frac{2\alpha + \alpha_c}{3})$ since $t^*(\cdot)$ is non decreasing by Lemma~\ref{mFMR_properties}.
Hence $\mathbf{L}_1(\hat{\theta}^\sigma, t^*(\beta)) \geq s^*$  for $\epsilon$ smaller than a threshold only depending on $\theta^*$ and $\alpha$. Hence,  we have on $ \Gamma_{\delta, t^*(\beta)} $ that 
\begin{align*}
\mathbf{L}(\hat{\theta}^\sigma, t^*(\beta)) - \delta \leq \mathbf{\widehat{L}}(\hat{\theta}^\sigma, t^*(\beta)) &\leq \delta + \mathbf{L}(\hat{\theta}^\sigma,  t^*(\beta)). 
\end{align*}
By using again Lemma~\ref{lemcdftheta}, we have  
\begin{align*}
\mathbf{L}(\theta^*, t^*(\beta)) -  3\Psi(\epsilon)/s^*
\leq \mathbf{L}(\hat{\theta}^\sigma, t^*(\beta)) \leq \mathbf{L}(\theta^*, t^*(\beta)) +  3\Psi(\epsilon)/s^*.
\end{align*}
Given that $\mathbf{L}(\theta^*, t^*(\beta)) = \mFMR(\proc^*_{t^*(\beta)}) = \beta$ (see Lemma~\ref{oracle_optimality} (i)), it follows that for $\gamma=\gamma (\epsilon, \delta) = \delta + 4\Psi(\epsilon)/s^*$,  on the event $ \Gamma_{\delta, t^*(\alpha - \gamma)} \cap  \Gamma_{\delta, t^*(\alpha + \gamma)} $,
\begin{align*}
\mathbf{\widehat{L}}(\hat{\theta}^\sigma, t^*(\alpha - \gamma)) &\leq \alpha, \quad \mathbf{\widehat{L}}(\hat{\theta}^\sigma, t^*(\alpha + \gamma)) > \alpha, 
\end{align*}
where we indeed check that $\alpha - \gamma \geq \frac{2\alpha + \alpha_c}{3}$ and  $\alpha +\gamma \leq \frac{\alpha + \bar\alpha}{2}$ for $\delta$ and $\epsilon$ smaller than some threshold only depending on $\theta^*$ and $\alpha$. 
In a nutshell, we have established
\begin{align}\label{inclusionevent}
\Gamma_{\delta, t^*(\alpha - \gamma)} \cap  \Gamma_{\delta, t^*(\alpha + \gamma)} \cap \Omega_\epsilon\subset \left\{t^*(\alpha - \gamma)\leq \hat{t}(\alpha) \leq t^*(\alpha + \gamma) \right\}.
\end{align}


\paragraph{Step 2: upper-bounding the FMR}


Let us consider the event
\begin{align*}
\Lambda_{\alpha,\delta,\epsilon}:=\Gamma_{0, \delta s^*/2, t^*(\alpha - \gamma)} &\cap \Gamma_{1, \delta s^*/2, t^*(\alpha - \gamma)}
\cap\Gamma_{0, \delta s^*/2, t^*(\alpha + \gamma)} \\
&\cap \Gamma_{1, \delta s^*/2, t^*(\alpha + \gamma)}
\cap\Upsilon_{0, \delta, t^*(\alpha + \gamma)}\cap \Omega_\epsilon,
\end{align*}
where the different events have been defined in the previous section.

Let us prove \eqref{FMR_bound}. By using Lemma~\ref{lem:sigmaprecise} and \eqref{inclusionevent},
\begin{align*}
\FMR(\hat{\proc}) 
&\leq \esp_{\theta^*} [ \mathbf{\widehat{M}}(\hat{\theta}^\sigma, \hat{t}(\alpha)) \ind_{\Lambda_{\alpha,\delta,\epsilon}}  ] + \P((\Lambda_{\alpha,\delta,\epsilon})^c) \\
&\leq \esp_{\theta^*} \left[\frac{\mathbf{\widehat{M}}_0(\hat{\theta}^\sigma, t^*(\alpha + \gamma))}{\mathbf{\widehat{L}}_1(\hat{\theta}^\sigma, t^*(\alpha - \gamma))}\ind_{\Lambda_{\alpha,\delta,\epsilon}} \right] +  \P((\Lambda_{\alpha,\delta,\epsilon})^c) .
\end{align*}
Now using a concentration argument on the event $\Lambda_{\alpha,\delta,\epsilon}\subset \Gamma_{1, \delta , t^*(\alpha - \gamma)} \cap \Upsilon_{0, \delta, t^*(\alpha + \gamma)}$, we have
\begin{align}
\FMR(\hat{\proc}) &\leq \esp_{\theta^*} \left[\frac{\mathbf{{M}}_0(\hat{\theta}^\sigma, t^*(\alpha + \gamma)) + \delta}{\mathbf{{L}}_1(\hat{\theta}^\sigma, t^*(\alpha - \gamma))-\delta}\ind_{\Lambda_{\alpha,\delta,\epsilon}} \right] +  \P((\Lambda_{\alpha,\delta,\epsilon})^c) \nonumber\\
&\leq \frac{\mathbf{{M}}_0({\theta^*}, t^*(\alpha + \gamma)) + 3\Psi(\epsilon)+ \delta}{\mathbf{{L}}_1({\theta^*}, t^*(\alpha - \gamma))- \Psi(\epsilon)-\delta}+  \P((\Lambda_{\alpha,\delta,\epsilon})^c)\nonumber\\
&= \frac{\mathbf{{L}}_0({\theta^*}, t^*(\alpha + \gamma)) + 3\Psi(\epsilon)+ \delta}{\mathbf{{L}}_1({\theta^*}, t^*(\alpha - \gamma))- \Psi(\epsilon)-\delta}+  \P((\Lambda_{\alpha,\delta,\epsilon})^c),\label{FMRbound}
\end{align}
by using Lemma~\ref{lemcdftheta} and that $\mathbf{{M}}_0({\theta^*}, t)=\mathbf{{L}}_0({\theta^*}, t)$ for all $t$ by definition.
Now, using again Lemma~\ref{lemcdftheta}, we have
\begin{align*}
\mathbf{{L}}_0({\theta^*}, t^*(\alpha + \gamma))&\leq \mathbf{{L}}_0({\theta^*}, t^*(\alpha - \gamma)) + \mathcal{W}_T(t^*(\alpha + \gamma)-t^*(\alpha - \gamma))\\
&\leq  \mathbf{{L}}_0({\theta^*}, t^*(\alpha - \gamma)) + \mathcal{W}_T( \mathcal{W}_{t^*,\alpha}(2 \gamma))
 \end{align*}
 This entails
 \begin{align*}
\FMR(\hat{\proc}) &\leq \frac{\mathbf{{L}}_0({\theta^*}, t^*(\alpha - \gamma)) +\mathcal{W}_T( \mathcal{W}_{t^*,\alpha}(2 \gamma))+ 3\Psi(\epsilon)+ \delta}{\mathbf{{L}}_1({\theta^*}, t^*(\alpha - \gamma))- \Psi(\epsilon)-\delta} +  \P((\Lambda_{\alpha,\delta,\epsilon})^c)\\
&\leq \frac{\mathbf{{L}}_0({\theta^*}, t^*(\alpha - \gamma)) }{\mathbf{{L}}_1({\theta^*}, t^*(\alpha - \gamma))- \Psi(\epsilon)-\delta} \\
&+(s^*/2)^{-1} \left(\mathcal{W}_T( \mathcal{W}_{t^*,\alpha}(2 \gamma))+ 3\Psi(\epsilon)+ \delta\right) +  \P((\Lambda_{\alpha,\delta,\epsilon})^c),
 \end{align*}
by choosing $\epsilon,\delta$ smaller than a threshold (only depending on $\theta^*$ and $\alpha$) so that $\mathbf{{L}}_1({\theta^*}, t^*(\alpha - \gamma))- \Psi(\epsilon)-\delta\geq s^*/2$. Now using $\mathbf{{L}}_0({\theta^*}, t^*(\alpha - \gamma))=(\alpha-\gamma) \mathbf{{L}}_1({\theta^*}, t^*(\alpha - \gamma))$, we have
  \begin{align*}
 \frac{\mathbf{{L}}_0({\theta^*}, t^*(\alpha - \gamma)) }{\mathbf{{L}}_1({\theta^*}, t^*(\alpha - \gamma))- \Psi(\epsilon)-\delta}
 &= (\alpha - \gamma) \left(1+ \frac{\Psi(\epsilon)+\delta}{\mathbf{{L}}_1({\theta^*}, t^*(\alpha - \gamma))- \Psi(\epsilon)-\delta}\right)\\
 &\leq \alpha  \left(1+ (s^*/2)^{-1}(\Psi(\epsilon)+\delta)\right).
  \end{align*}
This leads to
 \begin{align*}
\FMR(\hat{\proc}) &\leq \alpha  +(2/s^*) \left[\mathcal{W}_T( \mathcal{W}_{t^*,\alpha}(2\delta + 8\Psi(\epsilon)/s^*))+ 4\Psi(\epsilon)+ 2\delta\right] +  \P((\Lambda_{\alpha,\delta,\epsilon})^c),
  \end{align*}
which holds true for $\delta,\epsilon$ smaller than a threshold only depending on $\theta^*$ and $\alpha$.

\paragraph{Step 3: upper-bounding the mFMR}

We apply a similar technique as for step 2. By using Lemma~\ref{lem:sigmaprecise} and \eqref{inclusionevent},
\begin{align*}
\mFMR(\hat{\proc}) 
&\leq \frac{\esp_{\theta^*} [ \mathbf{\widehat{M}}_0(\hat{\theta}^\sigma, \hat{t}(\alpha))\ind_{\Lambda_{\alpha,\delta,\epsilon}} ]+ \P((\Lambda_{\alpha,\delta,\epsilon})^c)}{\esp_{\theta^*} [ \mathbf{\widehat{L}}_1(\hat{\theta}^\sigma, \hat{t}(\alpha))\ind_{\Lambda_{\alpha,\delta,\epsilon}} ]}    \\
&\leq \frac{\esp_{\theta^*} [ \mathbf{\widehat{M}}_0(\hat{\theta}^\sigma, t^*(\alpha+\gamma))\ind_{\Lambda_{\alpha,\delta,\epsilon}}]+ \P((\Lambda_{\alpha,\delta,\epsilon})^c)}{\esp_{\theta^*} [ \mathbf{\widehat{L}}_1(\hat{\theta}^\sigma, t^*(\alpha-\gamma))\ind_{\Lambda_{\alpha,\delta,\epsilon}}]}.
\end{align*}
Now using a concentration argument on $\Lambda_{\alpha,\delta,\epsilon}\subset \Gamma_{1, \delta , t^*(\alpha - \gamma)} \cap \Upsilon_{0, \delta, t^*(\alpha + \gamma)}$, we have
\begin{align*}
\mFMR(\hat{\proc}) &\leq \frac{\esp_{\theta^*} [ (\mathbf{{M}}_0(\hat{\theta}^\sigma, t^*(\alpha+\gamma))+\delta)\ind_{\Lambda_{\alpha,\delta,\epsilon}}]+ \P((\Lambda_{\alpha,\delta,\epsilon})^c)}{\esp_{\theta^*} [ (\mathbf{{L}}_1(\hat{\theta}^\sigma, t^*(\alpha-\gamma))-\delta)\ind_{\Lambda_{\alpha,\delta,\epsilon}}]} \\
&\leq \frac{\mathbf{{M}}_0({\theta^*}, t^*(\alpha + \gamma)) + 3\Psi(\epsilon)+ \delta+  \P((\Lambda_{\alpha,\delta,\epsilon})^c)}{\mathbf{{L}}_1({\theta^*}, t^*(\alpha - \gamma))- \Psi(\epsilon)-\delta-  \P((\Lambda_{\alpha,\delta,\epsilon})^c)}\\
&= \frac{\mathbf{{L}}_0({\theta^*}, t^*(\alpha + \gamma)) + 3\Psi(\epsilon)+ \delta+  \P((\Lambda_{\alpha,\delta,\epsilon})^c)}{\mathbf{{L}}_1({\theta^*}, t^*(\alpha - \gamma))- \Psi(\epsilon)-\delta-  \P((\Lambda_{\alpha,\delta,\epsilon})^c)},
\end{align*}
by using Lemma~\ref{lemcdftheta} and that $\mathbf{{M}}_0({\theta^*}, t)=\mathbf{{L}}_0({\theta^*}, t)$ by definition.
Letting $x=\mathbf{{L}}_0({\theta^*}, t^*(\alpha + \gamma)) + 3\Psi(\epsilon)+ \delta$, $y=\mathbf{{L}}_1({\theta^*}, t^*(\alpha - \gamma))- \Psi(\epsilon)-\delta$ and $u=\P((\Lambda_{\alpha,\delta,\epsilon})^c)$, we have obtained the bound $(x+u)/(y-u)$, which has to be compared with the FMR bound \eqref{FMRbound}, which  reads $x/y + u$. Now, when $y\in [0,1]$, $x\geq 0$,  $x/y\leq 2$, $u/y\leq 1/2$, $y-u\geq s^*/2$, we have 
$$(x+u)/(y-u)\leq \frac{x/y}{1-u/y} + (2/s^*) u\leq x/y (1+ 2u/y) + (2/s^*) u \leq x/y + (10/s^*)u.$$
As a result, for $\epsilon,\delta$ small enough, and $\P((\Lambda_{\alpha,\delta,\epsilon})^c)\leq s^*/4$, we obtain the same bound as for the FMR, with $\P((\Lambda_{\alpha,\delta,\epsilon})^c)$ replaced by $(10/s^*)\P((\Lambda_{\alpha,\delta,\epsilon})^c)$.

\paragraph{Step 4:  lower-bounding the selection rate} 

In Step 3, when bounding the mFMR, we derived a lower bound for the denominator of the mFMR, that is, $\esp_{\theta^*}(|\hat{S}|)$. It reads
\begin{align*}
n^{-1}\esp_{\theta^*}(|\hat{S}|) &\geq  \mathbf{{L}}_1({\theta^*}, t^*(\alpha - \gamma))- \Psi(\epsilon)-\delta-  \P((\Lambda_{\alpha,\delta,\epsilon})^c)\\
&\geq \mathbf{{L}}_1({\theta^*}, t^*(\alpha)) - \mathcal{W}_T (t^*(\alpha)- t^*(\alpha - \gamma))- \Psi(\epsilon)-\delta-  \P((\Lambda_{\alpha,\delta,\epsilon})^c)\\
&\geq n^{-1}\esp_{\theta^*}(|S^*_{t^*(\alpha)}|) - \mathcal{W}_T (\mathcal{W}_{t^*,\alpha}( \gamma))- \Psi(\epsilon)-\delta-  \P((\Lambda_{\alpha,\delta,\epsilon})^c),
\end{align*}
by using \eqref{WTbound} and \eqref{Wtstarbound}.
Now consider another procedure
$\proc = (\widehat{\mathbf{Z}}, S) $ that controls the mFMR at level $\alpha$, that is, $\mFMR(\proc) \leq \alpha$.
By Lemma~\ref{oracle_optimality}, we then have $\esp_{\theta^*}(|S^*_{t^*(\alpha)}|) \geq \esp_{\theta^*}(|S|) $.

\paragraph{Step 5:  concentration}

Finally, we bound  $\P((\Lambda_{\alpha,\delta,\epsilon})^c)$ by using Lemma~\ref{lem:conc} with $x=(1 + 2c) \sqrt{\frac{\log n }{n}}$ (with $c$ defined in Lemma \ref{lem:conc}). This gives for $\delta=2x/s^*$, and $n\geq (2e)^3$ 
$$
\P((\Lambda_{\alpha,\delta,\epsilon})^c)\leq 5/n^2 + \P(\Omega_\epsilon^c).
$$

\section{Proofs of lemmas}

\paragraph{Proof of Lemma~\ref{lem:clusteringBayes}}

The clustering risk of  $\widehat{\mathbf{Z}}$ is given by
\begin{align*}
R(\widehat{\mathbf{Z}}) &= \esp_{\theta^*} \left( \underset{\sigma \in [Q]}{\min} \esp_{\theta^*} \left(\:n^{-1}\:  \sum_{i=1}^n \ind \{ Z_i \neq \sigma(\hat{Z}_i) \} \: \bigg|\: \mathbf X\right)\right) \\
&=  \esp_{\theta^*} \left( \underset{\sigma \in [Q]}{\min} \:n^{-1}\:  \sum_{i=1}^n \Pstar ( Z_i \neq \sigma(\hat{Z}_i) \:|\: \mathbf X) \right) \\
&\geq  \esp_{\theta^*} \left( \underset{\widehat{\mathbf Z }}{\min} \:n^{-1}\:  \sum_{i=1}^n \Pstar ( Z_i \neq \hat{Z}_i \:|\: \mathbf X) \right), 
\end{align*}
where, by independence, the minimum in the lower bound is achieved for the Bayes clustering.
Thus, $
R(\widehat{\mathbf Z } ) \geq n^{-1}\: \sum_{i=1}^n  \esp_{\theta^*} (T_i^* )
$.
Moreover, $n^{-1}\: \sum_{i=1}^n  \esp_{\theta^*} (T_i^* ) \ge R(\widehat{\mathbf{Z}}^*)$, since
\begin{align*} R(\widehat{\mathbf{Z}}^*) &= \esp_{\theta^*} \left( \underset{\sigma \in [Q]}{\min} \:n^{-1}\:  \sum_{i=1}^n \Pstar ( Z_i \neq \sigma(\hat{Z}^*_i) \:|\: \mathbf X) \right) \\
 & \leq \esp_{\theta^*} \left( \:n^{-1}\:  \sum_{i=1}^n \Pstar ( Z_i \neq \hat{Z}^*_i \:|\: \mathbf X) \right). 
\end{align*}
Thus, $\min_{\widehat{\mathbf{Z}}} R(\widehat{\mathbf{Z}}) =  R(\widehat{\mathbf{Z}}^*)$ and the proof is completed.

\paragraph{Proof of Lemma~\ref{lem:FMRTistar}}

Following  the reasoning of the proof of Lemma~\ref{lem:clusteringBayes}, we have 
\begin{align*}
\FMR_{\theta^*}(\proc) &= \esp_{\theta^*} \left( \min_{\sigma \in [Q]} \:\esp_{\theta^*} \left( \frac{\sum_{i\in S} \ind \{Z_i \neq \sigma(\hat{Z}^*_i) \}}{\max(|S|,1)}  \: \bigg|\: \mathbf X \right)\right) \\
&= \esp_{\theta^*} \left( \frac{\sum_{i\in S} T^*_i }{\max(|S|,1)}  \right) .
\end{align*}

 \paragraph{Proof of Lemma~\ref{lem:FMRPI}}

By definition, we have
\begin{align*}
\FMR (\widehat{\proc}^{\mbox{\tiny PI}}_\alpha) 
&= \esp_{\theta^*} \left( \min_{\sigma \in [Q]} \esp_{\theta^*} \left( \frac{\sum_{i=1}^n \ind_{Z_i \neq \sigma(\hat{Z}^{\mbox{\tiny PI}}_i(\mathbf{X}))} \ind \{ i\in \hat{S}^{\mbox{\tiny PI}}(\mathbf{X}) \} } {\max(|\hat{S}^{\mbox{\tiny PI}}(\mathbf{X})|,1) } \: \bigg|\: \mathbf X\right)\right) , 
\end{align*}
so that \eqref{eqFMRplugin} follows by a direct integration w.r.t. the latent variable $Z$.
 
\paragraph{Proof of Lemma~\ref{oracle_optimality}}

By Lemma~\ref{mFMR_properties}, we have that $\mFMR(\proc^*_t)$ is monotonous in $t$ and continuous w.r.t. $t$ on $(t^*(\alpha_c), 1]$, thus for $\alpha \in (\alpha_c, \bar{\alpha}]$,  $\mFMR(\proc^*_{t^*(\alpha)}) = \alpha$ which gives (i). 
For (ii), let $\proc = (\widehat{\mathbf{Z}}, S)$ be a procedure such that $\mFMR(\proc) \leq \alpha$. 
Let us consider the procedure $\proc'$ with the Bayes clustering $\widehat{\mathbf{Z}}^*$ and the same selection rule $S$. Since $\proc'$ is based on a Bayes clustering, by the same reasoning leading to $R(\widehat{ \mathbf Z }^*) \leq R(\widehat{ \mathbf Z })$ in Section \ref{FMR:sec:oracle}, we have that 
$
\mFMR(\proc')\leq \mFMR(\proc)\leq \alpha$
with 
$$
\mFMR(\proc') =  \frac{ \esp_{\theta^*} \left( \sum_{i\in S} T^*_i \right)  }{\esp_{\theta^*} ( |S|)}.
$$
Hence, 
\begin{equation}\label{interm2}
\esp_{\theta^*} \left( \sum_{i\in S} T^*_i \right) \leq \alpha \esp_{\theta^*} ( |S|).
\end{equation}
 Now we use an  argument similar to the proof of Theorem 1 in \cite{CSWW2019}. By definition of $S^*_{t^*(\alpha)}$, we have that
\begin{align*}
\sum_{i=1}^n \left(\ind_{i \in S^*_{t^*(\alpha)}(\mathbf{X})} - \ind_{i \in S(\mathbf{X})} \right) (T^*_i - t^*(\alpha))  \leq 0
\end{align*}
which we can rewrite as
\begin{align*}
\sum_{i=1}^n \left(\ind_{i \in S^*_{t^*(\alpha)}(\mathbf{X})} - \ind_{i \in S(\mathbf{X})} \right) (T^*_i - t^*(\alpha) + \alpha - \alpha) \leq 0
\end{align*} 
and so 
\begin{align*}
&\esp_{\theta^*} \left( \sum_{i=1}^n \left(\ind_{i \in S^*_{t^*(\alpha)}(\mathbf{X})} - \ind_{i \in S(\mathbf{X})} \right) (T^*_i - \alpha) \right)  \\
&\leq (t^*(\alpha) - \alpha)  \esp_{\theta^*} \left( \sum_{i=1}^n \left(\ind_{i \in S^*_{t^*(\alpha)}(\mathbf{X})} - \ind_{i \in S(\mathbf{X})} \right)  \right) \\
&= (t^*(\alpha) - \alpha) (\esp_{\theta^*}(|S^*_{t^*(\alpha)}|) - \esp_{\theta^*}(|S|)).
\end{align*} 
On the other hand, $\mFMR(\proc^*_{t^*(\alpha)}) = \alpha$ together with \eqref{interm2} implies that 
\begin{align*}
&\esp_{\theta^*} \left( \sum_{i=1}^n \left(\ind_{i \in S^*_{t^*(\alpha)}(\mathbf{X})} - \ind_{i \in S(\mathbf{X})} \right) (T^*_i - \alpha)  \right) \\
&= 
 \esp_{\theta^*} \left( \sum_{i \in S^*_{t^*(\alpha)}} T^*_i - \alpha |S^*_{t^*(\alpha)}| - \sum_{i \in S} T^*_i + \alpha |S| \right) \geq 0 . 
\end{align*}
Combining, the relations above provides 
\begin{align*}
(t^*(\alpha) - \alpha) (\esp_{\theta^*}(|S^*_{t^*(\alpha)}|) - \esp_{\theta^*}(|S|)) \geq 0 .
\end{align*}
Finally, noting that $t^*(\alpha) - \alpha > 0$ since $\alpha=\mFMR(\proc^*_{t^*(\alpha)}) < t^*(\alpha)$ by (ii) Lemma~\ref{mFMR_properties}, this gives $\esp_{\theta^*}(|S^*_{t^*(\alpha)}|) - \esp_{\theta^*}(|S|) \geq 0$ and  concludes the proof.

\paragraph{Proof of Lemma~\ref{lem:sigmaprecise}}
First, we have  by definition $ \ell_q(X_i, \theta^\sigma) =  \ell_{\sigma(q)}(X_i, \theta) $ and thus $T(X_i,\hat{\theta})=T(X_i,\hat{\theta}^\sigma)$ by taking the maximum over $q$. This gives $\widehat{S}^\sigma=\widehat{S}$ and yields the first equality.
Next, we have on $\Omega_\epsilon$,
\begin{align*}
\underset{\sigma' \in [Q]}{\min} \esp_{\theta^*} \left(\varepsilon_{\widehat{S}}(\sigma'(\widehat{\mathbf{Z}}) ,\mathbf{Z})\: \bigg|\: \mathbf X\right) &\leq \esp_{\theta^*} \left( \varepsilon_{\widehat{S}}(\sigma(\widehat{\mathbf{Z}}) ,\mathbf{Z}) \: \bigg|\: \mathbf X\right)  \\ 
&\leq \esp_{\theta^*} \left(\varepsilon_{\widehat{S}^\sigma}(\sigma(\widehat{\mathbf{Z}}) ,\mathbf{Z})  \: \bigg|\: \mathbf X\right) ,
\end{align*}
still because $\widehat{S}^\sigma=\widehat{S}$. Now observe that, 
\begin{align*}
 \esp_{\theta^*} \left(\varepsilon_{\widehat{S}^\sigma}(\sigma(\widehat{\mathbf{Z}}) ,\mathbf{Z})\: \bigg|\: \mathbf X\right) 
 &=\frac{1}{n} \sum_{i=1}^n  \Pstar(Z_i \neq \sigma(\bar{q}(X_i, \hat{\theta})) \: \big|\: \mathbf X) 
 \ind_{T(X_i, \hat{\theta}^\sigma) < \hat{t}(\alpha)}\\
 &= \frac{1}{n} \sum_{i=1}^n (1 - \ell^*_{i, \sigma(\bar{q}(X_i, \hat{\theta}))}) \ind_{T(X_i, \hat{\theta}^\sigma) < \hat{t}(\alpha)}\\
 &=  \mathbf{\widehat{M}}_0(\hat{\theta}^\sigma, \hat{t}(\alpha))  ,
\end{align*}
because $\sigma(\bar{q}(X_i, \hat{\theta}))= \bar{q}(X_i, \hat{\theta}^\sigma)$. This proves the result.

\section{Auxiliary results}

\begin{lemma} \label{mFMR_properties}
Let us consider the procedure $\proc^*_t$ defined in Section~\ref{FMR:sec:optimalproc} and the functional $\mFMR^*_t$ defined by \eqref{mFMRstart}. Then we have 
\begin{equation}\label{equmFMRproct}
\mFMR(\proc^*_t) = \frac{ \esp_{\theta^*} \left( \sum_{i=1}^n T^*_i \ind_{T^*_i < t}\right)  }{\esp_{\theta^*} \left( \sum_{i=1}^n \ind_{T^*_i < t}\right)} = \mFMR^*_t,\quad t\in [0,1].
\end{equation}
Moreover, the following properties for the function $t\in [0,1]\mapsto \mFMR(\proc^*_t)$:
\begin{enumerate}
\item [(i)] $\mFMR(\proc^*_t)$ is non-decreasing in $t\in [0,1]$ and, under Assumption~\ref{Tcontinuity}, it is increasing in $t\in (t^*(\alpha_c), t^*(\bar\alpha))$;
\item [(ii)] $\mFMR(\proc^*_t) < t$ for $t \in (0,1]$;
\item [(iii)] Under Assumption~\ref{Tcontinuity},  $ \mFMR(\proc^*_t)$ is continuous w.r.t. $t$ on $(t^*(\alpha_c), 1]$, where $t^*(\alpha_c)$ is given by \eqref{alphac}.
\end{enumerate}
\end{lemma}

\begin{proof}
First, \eqref{equmFMRproct} is obtained similarly than \eqref{FMRoracle}.
For proving (i), let $t_1, t_2 \in [0,1]$ such that $t_1 < t_2$. We show that $\mFMR(\proc^*_{t_1}) \leq \mFMR(\proc^*_{t_2})$. Remember here the convention $0/0=0$ and that $\mFMR(\proc^*_{t})=\esp_{\theta^*} \left( T(X, \theta^*) \:|\: T(X, \theta^*) < t \right)$. First, if $ \Pstar \left( T(X, \theta^*) < t_1 \right)=0$ then the result is immediate. Otherwise, we have that \begin{align*}
&\mFMR(\proc^*_{t_1}) - \mFMR(\proc^*_{t_2}) \\
&= ( \Pstar \left( T(X, \theta^*) < t_1 \right) )^{-1} \\
&\cdot  \esp_{\theta^*} \left(T(X,\theta^*) \left\{ \ind_{T(X,\theta^*) < t_1} - \frac{ \Pstar \left( T(X, \theta^*) < t_1 \right) } { \Pstar \left( T(X, \theta^*) < t_2 \right)} \ind_{T(X,\theta^*)< t_2} \right\} \right) ,
\end{align*}
where, given that $t_1 < t_2$, the quantity in the brackets is positive when $T(X,\theta^*) < t_1$ and is negative or zero otherwise. Hence,
\begin{align*}
&T(X,\theta^*) \left\{ \ind_{T(X,\theta^*) < t_1} - \frac{ \Pstar \left( T(X, \theta^*) < t_1 \right) } { \Pstar \left( T(X, \theta^*) < t_2 \right)} \ind_{T(X,\theta^*)< t_2} \right\} \\
&\leq 
t_1 \left\{ \ind_{T(X,\theta^*) < t_1} - \frac{ \Pstar \left( T(X, \theta^*) < t_1 \right) } { \Pstar \left( T(X, \theta^*) < t_2 \right)} \ind_{T(X,\theta^*)< t_2} \right\} .
\end{align*}
Taking the expectation makes the right-hand-side equal to zero, from which the result follows. 
Now, to show the increasingness, if $\mFMR(\proc^*_{t_1}) = \mFMR(\proc^*_{t_2})$ for $t^*(\alpha_c)<t_1<t_2<t^*(\bar\alpha)$, then the above reasoning shows that
\begin{align*}
 (T(X,\theta^*)-t_1) \left\{ \ind_{T(X,\theta^*) < t_1} - \frac{ \Pstar \left( T(X, \theta^*) < t_1 \right) } { \Pstar \left( T(X, \theta^*) < t_2 \right)} \ind_{T(X,\theta^*)< t_2} \right\}  \leq 0
\end{align*}
and has an expectation equal to $0$. Hence, given that $T(X,\theta^*)$ is continuous, we derive that almost surely 
$$
\Pstar \left( T(X, \theta^*) < t_2 \right) \ind_{T(X,\theta^*)< t_1} = \Pstar \left( T(X, \theta^*) < t_1 \right)  \ind_{T(X,\theta^*)< t_2},
$$
that is, $\Pstar(t_1\leq T^*_i<t_2)=0$, which is excluded by Assumption~\ref{Tcontinuity}. This entails $\mFMR(\proc^*_{t_1}) < \mFMR(\proc^*_{t_2})$.

For proving (ii), let $t>0$. If $\Pstar \left( T(X, \theta^*) < t \right)=0$ then the result is immediate. Otherwise, we have that $\mFMR(\proc^*_t) - t = (\Pstar \left( T(X, \theta^*) < t \right))^{-1} \esp_{\theta^*} ( (T(X,\theta^*) - t) \ind \{ T(X,\theta^*)  < t \} )$. The latter is clearly not positive, and is moreover negative because $(T(X,\theta^*) - t) \ind \{ T(X,\theta^*)  < t \}\leq 0$ and $\Pstar(T(X,\theta^*) = t)=0$ by Assumption~\ref{Tcontinuity}.

For proving (iii), let $\psi_0(t) =  \esp_{\theta^*} ( T(X,\theta^*) \ind \{ T(X,\theta^*) < t \} )$ and $\psi_1(t) = \Pstar  (T(X,\theta^*) < t  )$, the numerator and denominator of $\mFMR(\proc^*_t)=\mFMR^*_t$, respectively. $\psi_1(t)$ is non-decreasing in $t$, with $\psi_1(0)=0$ and $\psi_1(1) > 0$. Moreover, $\psi_0$ and $\psi_1$ are both continuous under Assumption~\ref{Tcontinuity}. Then denote by $t_c$ the largest $t$ s.t. $\psi_1(t)=0$. $\psi_1$ is zero on $[0, t_c]$ then strictly positive and non-decreasing on $(t_c, 1]$, and we have that $t_c = t^*(\alpha_c)$. Hence, $\mFMR(\proc^*_t)$ is zero on $[0, t_c]$ then strictly positive and continuous on $(t_c, 1]$. 
\end{proof}

\begin{remark}
With the notation of the above proof, $t\mapsto \mFMR(\proc^*_t)$ may have a discontinuity point at $t_c$ since for $t_n \underset{t_n > t_c}{\rightarrow} t_c$, as $\psi_1(t_n) \rightarrow 0$, one does not necessarily have that $\mFMR(\proc^*_t) \rightarrow 0$. 
\end{remark}

\begin{lemma}[Expression of plug-in procedure as a thresholding rule]\label{lempluginthres}
For any $\alpha\in (0,1)$, let us consider the plug-in procedure $\widehat{\proc}^{\mbox{\tiny PI}}_\alpha=(\widehat{\mathbf{Z}}^{\mbox{\tiny PI}}, \widehat{S}^{\mbox{\tiny PI}}_\alpha)$ defined by Algorithm~\ref{algo:pi} and denote $K=|\widehat{S}^{\mbox{\tiny PI}}_\alpha|$ the maximum of the $k\in \{0,\dots,n\}$ such that $\max(k,1)^{-1} \sum_{j=1}^k \hat{T}_{(j)} \leq \alpha$ for 
$\hat{T}_i =1 - \max_q \ell_q(X_i, \hat{\theta}),$ $1\leq i\leq n$. 
Consider also 
$
\hat{t}(\alpha) 
$
defined by \eqref{thatalpha}. Let Assumption \ref{Tcontinuity} be true and consider an estimator $\hat{\theta}$ satisfying Assumption~\ref{ass_estim}. 
Then it holds that $\hat{t}(\alpha)=\hat{T}_{(K+1)}$ and 
\begin{align*}
\widehat{S}^{\mbox{\tiny PI}}_\alpha =\{ i\in \{1,\dots,n\}\::\: \hat{T}_i < \hat{t}(\alpha) \}.
\end{align*}

\end{lemma}

\begin{proof}
If $\hat{T}_{(K)}<\hat{T}_{(K+1)}$ then the result is immediate. Thus it suffices to show that $\hat{T}_{(K)} = \hat{T}_{(K+1)}$ occurs with probability $0$. 
From Assumption~\ref{ass_estim} (with the countable set $\mathcal{D}$ defined therein), we have 
\begin{align*}
\Pstar(\hat{T}_{(K)} = \hat{T}_{(K+1)}) \leq   \Pstar \left( \bigcup_{i \neq j} \{\hat{T}_i = \hat{T}_j\} \right) \leq  \Pstar \left( \bigcup_{\theta \in  \mathcal{D}} \bigcup_{i \neq j} \{T(X_i, \theta) = T(X_j, \theta) \}\right).
\end{align*}
Now, the right term is a countable union of events which are all of null probability under Assumption \ref{Tcontinuity}. The result follows. 
\end{proof}

\begin{lemma}\label{lemcdftheta}
	We have for all $\theta\in \Theta$,
	\begin{align}
		\sup_{t\in [0,1]} |\mathbf{L}_1(\theta, t) - \mathbf{L}_1(\theta^*, t)|  &\leq \Psi(\|\theta^*-\theta\|); \label{diffL1}\\
		\sup_{t\in [0,1]} |\mathbf{L}_0(\theta, t) - \mathbf{L}_0(\theta^*, t)|  &\leq 2\Psi(\|\theta^*-\theta\|); \label{diffL0}\\
		\sup_{t\in [t^*((\alpha + \alpha_c)/2), 1]} |\mathbf{L}(\theta, t) - \mathbf{L}(\theta^*, t)|  &\leq 3\Psi(\|\theta^*-\theta\|)/s^*; \label{diffL}\\
		\sup_{t\in [0,1]} |\mathbf{M}_0(\theta, t) - \mathbf{M}_0(\theta^*, t)|  &\leq 3\Psi(\|\theta^*-\theta\|); \label{diffM0}\\
		\sup_{t\in [t^*((\alpha + \alpha_c)/2), 1]} |\mathbf{M}(\theta, t) - \mathbf{M}(\theta^*, t)|  &\leq 4\Psi(\|\theta^*-\theta\|)/s^*; \label{diffM}
	\end{align}
	where $\alpha\in (\alpha_c,\bar\alpha]$ and $s^*>0$ is given by \eqref{equsstar}.  In addition, for all $\theta\in \Theta$ and $t,t'\in [0,1]$,
	\begin{align}
		|\mathbf{L}_0(\theta, t) - \mathbf{L}_0(\theta, t')|  &\leq 4\Psi(\|\theta^*-\theta\|)+ \mathcal{W}_T(|t-t'|). \label{diffL0t}
	\end{align}
	
\end{lemma}

\begin{proof}
	Fix $\theta\in \Theta$ and $t\in [0,1]$. We have for any $\delta>0$,
	\begin{align*}
		&\left| \Pstar (T(X, \theta) < t) -\Pstar (T(X, \theta^*) < t)\right|  \\
		&\leq \left( \Pstar (T(X, \theta^*) < t + \delta) -\Pstar (T(X, \theta^*) < t)\right) \vee  \left( \Pstar (T(X, \theta^*) < t) -\Pstar (T(X, \theta^*) < t-\delta)\right) \\
		&\:\:\:+\Pstar(|T(X, \theta^*)-T(X, \theta)|> \delta)\\
		&\leq \mathcal{W}_T(\delta) + \esp_{\theta^*}(|T(X, \theta^*)-T(X, \theta)|)/ \delta.
	\end{align*}
	In addition, by definition \eqref{equTXtheta}, 
	\begin{align*}
		|T(X, \theta^*)-T(X, \theta)| &\leq |\max_{1\leq q\leq Q} \ell_{q}(X,\theta^*)-\max_{1\leq q\leq Q} \ell_{q}(X,\theta)|\\
		&\leq \max_{1\leq q\leq Q} | \ell_{q}(X,\theta^*)- \ell_{q}(X,\theta)|.
	\end{align*}
	Hence, 
	$$
	\left| \Pstar (T(X, \theta) < t) -\Pstar (T(X, \theta^*) < t)\right|  \leq \inf_{\delta\in (0,1)}\left\{\mathcal{W}_T(\delta) + \mathcal{W}_{\ell}(\|\theta^*-\theta\|)/ \delta\right\} \leq \Psi(\|\theta^*-\theta\|),
	$$
	which establishes \eqref{diffL1}. 
	
	Next, we have 
	\begin{align*}
		& \mathbf{L}_0(\theta, t) - \mathbf{L}_0(\theta^*, t) \\
		&=
		\esp_{\theta^*} [ T(X, \theta) (\mathbbm{1}_{T(X,\theta) < t} -  \mathbbm{1}_{T(X,\theta^*) < t}) + \mathbbm{1}_{T(X,\theta^*) < t}( T(X, \theta) - T(X, \theta^*))  ]  \\
		&\leq t |\Pstar(T(X, \theta) < t) - \Pstar(T(X, \theta^*) < t)| \\
		&\quad + |\esp_{\theta^*} [ \mathbbm{1}_{T(X,\theta^*) < t}( T(X, \theta) - T(X, \theta^*))]| \\ 
		&\leq |\Pstar(T(X, \theta) < t) - \Pstar(T(X, \theta^*) < t)| + \esp_{\theta^*}  |T(X, \theta) - T(X, \theta^*)| \\
		&\leq 2 \Psi(\|\theta^*-\theta\|)
	\end{align*}
	By exchanging the role of $\theta$ and $\theta^*$ in the above reasoning, the same bound holds for $\mathbf{L}_0(\theta^*, t) - \mathbf{L}_0(\theta, t)$, which gives \eqref{diffL0}.
	To prove \eqref{diffL}, we use for any $t\in [t^*(\frac{\alpha + \alpha_c}{2}), 1]$, 
	\begin{align*}
		&\left| \frac{\mathbf{L}_0(\theta, t)}{\mathbf{L}_1(\theta, t)} - \frac{\mathbf{L}_0(\theta^*, t)}{\mathbf{L}_1(\theta^*, t)} \right| \\
		&\leq \left| \frac{\mathbf{L}_0(\theta, t) - \mathbf{L}_0(\theta^*, t)}{\mathbf{L}_1(\theta^*, t)} \right| + \mathbf{L}_0(\theta, t) \left| \frac{1}{\mathbf{L}_1(\theta^*, t)} - \frac{1}{\mathbf{L}_1(\theta, t)} \right| \\
		&\leq 2\Psi(\|\theta^*-\theta\|)/s^* + \frac{1}{\mathbf{L}_1(\theta^*, t)} \frac{\mathbf{L}_0(\theta, t)}{\mathbf{L}_1(\theta, t)} \left| \Pstar(T(X, \theta^*) < t) - \Pstar(T(X, \theta) < t) \right| \\
		&\leq 3\Psi(\|\theta^*-\theta\|)/s^*, 
	\end{align*}
	because $\mathbf{L}_0(\theta, t)\leq \mathbf{L}_1(\theta, t)$  and $\mathbf{L}_1(\theta^*, t)\geq s^*$ by monotonicity.
	Similarly to the bound on $\mathbf{L}_0$, we derive
	\begin{align*}
		& |\mathbf{M}_0(\theta, t) - \mathbf{M}_0(\theta^*, t)| \\
		&\leq |\Pstar(T(X, \theta) < t) - \Pstar(T(X, \theta^*) < t)| + \esp_{\theta^*} |U(X, \theta) - U(X, \theta^*)|. 
	\end{align*}
	Define $\bar{q}(X, \theta) \in \underset{q \in \{1, \dots, Q \}}{\argmax}  \ell_q(X, \theta)$. Now, since $U(X, \theta^*)\leq U(X, \theta)$ by definition \eqref{equU}, we have
	\begin{align*}
		\esp_{\theta^*} |U(X, \theta) - U(X, \theta^*)|&=\esp_{\theta^*}  [ U(X, \theta) - U(X, \theta^*) ] \\
		&= \esp_{\theta^*}  [ \ell_{\bar{q}(X, \theta)}(X, \theta^*) -\ell_{\bar{q}(X, \theta^*)}(X, \theta^*) ] \\
		&= \esp_{\theta^*}  [ \ell_{\bar{q}(X, \theta)}(X, \theta^*) - \ell_{\bar{q}(X, \theta)}(X, \theta) \\
		&\quad + \ell_{\bar{q}(X, \theta)}(X, \theta) -\ell_{\bar{q}(X, \theta^*)}(X, \theta^*) ] \\
		&\leq \esp_{\theta^*} [ \max_{1\leq q\leq Q} | \ell_q(X, \theta^*) - \ell_q(X, \theta) | ] \\
		&\quad + \esp_{\theta^*} [ \max_{1\leq q\leq Q} \ell_q(X, \theta) - \max_{1\leq q\leq Q} \ell_q(X, \theta^*) ] \\
		&\leq 2 \esp_{\theta^*} [ \max_{1\leq q\leq Q} | \ell_q(X, \theta^*) - \ell_q(X, \theta) | ] \leq 2 \Psi(\|\theta^*-\theta\|).\end{align*}
	This proves \eqref{diffM0} and leads to \eqref{diffM} by following the reasoning that provided \eqref{diffL}.

Next, we have for $0\leq t'\leq t\leq 1$, by \eqref{diffL0}, 
\begin{align*}
	|\mathbf{L}_0(\theta, t) - \mathbf{L}_0(\theta, t')| \leq |\mathbf{L}_0(\theta^*, t) - \mathbf{L}_0(\theta^*, t')| + 4\Psi(\|\theta^*-\theta\|).
\end{align*}
Moreover, 
\begin{align*}
	|\mathbf{L}_0(\theta^*, t) - \mathbf{L}_0(\theta^*, t')| 
	&= \mathbf{L}_0(\theta^*, t) - \mathbf{L}_0(\theta^*, t')
	=\esp_{\theta^*} [ T(X, \theta^*) \mathbbm{1}_{t'\leq T(X,\theta^*) < t}   ]  \\
	&\leq \esp_{\theta^*} [  \mathbbm{1}_{t'\leq T(X,\theta^*) < t}   ] \\
	&= \Pstar(T(X,\theta^*) < t)-\Pstar(T(X,\theta^*) < t'),
\end{align*}
which is below $\mathcal{W}_T(t-t')$ by \eqref{WTbound}. This leads to \eqref{diffL0t}.
\end{proof}

\begin{lemma}[Concentration of $\mathbf{\widehat{L}_0}$ \eqref{L0chap}, $\mathbf{\widehat{L}_1}$ \eqref{L1chap}, and $\mathbf{\widehat{M}_0}$ \eqref{equM0}]
\label{lem:conc}
Let Assumption~\ref{Tcontinuity} be true. 
Recall  $\VCdim, \VCdim_{-}$ defined by \eqref{VCdim}, \eqref{VCdimminus} respectively, set 
$c:= 14Q\sqrt{\VCdim} + 7Q^2\sqrt{\VCdim_{-}}$ 
and consider any countable set $\subsetEsti\subset \Theta$. 
 For all $t \in (0,1]$ and for $n\geq (2e)^3$, we have
\begin{align}
\Pstar \left( \underset{\theta \in \subsetEsti}{\sup} \left| \mathbf{\widehat{L}}_0(\theta, t) - \mathbf{L}_0(\theta, t) \right|  > x \right) &\leq n^{-2};  \label{proba1}\\
\Pstar \left( \underset{\theta \in \subsetEsti}{\sup} \left| \mathbf{\widehat{L}}_1(\theta, t) - \mathbf{L}_1(\theta, t) \right|  > x \right) &\leq n^{-2}; \label{proba2}\\
\Pstar \left( \underset{\theta \in \subsetEsti}{\sup} \left| \mathbf{\widehat{M}}_0(\theta, t) - \mathbf{M}_0(\theta, t) \right|  > x \right) &\leq n^{-2}, \label{proba3}
\end{align}
for any $x  \geq (1 +2c) \sqrt{\frac{\log n}{n}}$ and provided that $(1 +2c) \sqrt{\frac{\log n}{n}}\leq 1$. 
\end{lemma}


\begin{proof} 
For a fixed $t \in (0,1]$, let $\mathscr{F}_{L_0} = \{ T(., \theta)\ind \{T(., \theta) \leq t \}, \theta \in \subsetEsti \}$, 
$\mathscr{F}_{L_1} = \{ \ind \{T(., \theta) \leq t \}, \theta \in \subsetEsti \}$, and $\mathscr{F}_{M_0} = \{ U(., \theta)\ind \{T(., \theta) \leq t \}, \theta \in \subsetEsti \}$. We apply Lemma~\ref{lem:massart} and  Lemma~\ref{lem:rademacher} for $\xi_i=X_i$, $1\leq i\leq n$, $b=1$, $a=0$ and for each $\mathscr{F}\in \{\mathscr{F}_{L_0},\mathscr{F}_{L_1},\mathscr{F}_{M_0}\}$ to get that the corresponding probability in \eqref{proba1}-\eqref{proba2}-\eqref{proba3} is at most $n^{-2}$ by taking 
$$
x\geq \sqrt{\frac{\log n}{n}} + 2 \esp \Radm_n (\mathscr{F}),
$$
where $\Radm_n (\mathscr{F})$ denotes the Rademacher complexity of $\mathscr{F}$, see \eqref{radcomplex}.
We now bound each $\Radm_n (\mathscr{F})$ by using Lemma \ref{lem:Radproduct}:
\begin{align}
	\esp \Radm_n (\mathscr{F}_{L_0}) &\leq \esp \Radm_n (\mathscr{F}_{L_1}) + \esp \Radm_n (\{T(., \theta), \theta \in \Theta\})\nonumber \\
	&\leq \esp \Radm_n (\mathscr{F}_{L_1}) + \sum_{q=1}^Q \esp \Radm_n (\{ \ell_q(., \theta), \theta \in \Theta \}) \label{eq:RadT} ;\\
	\esp \Radm_n (\mathscr{F}_{L_1}) &\leq \sum_{q=1}^Q  \esp \Radm_n (\{ \ind \{\ell_q(., \theta) < 1-t \}, \theta \in \Theta \}) \label{eq:RadL1}; \\
	\esp \Radm_n (\mathscr{F}_{M_0}) &\leq \esp \Radm_n (\mathscr{F}_{L_1}) + \esp \Radm_n (\{U(., \theta), \theta \in \Theta\}),\nonumber
\end{align}
where for \eqref{eq:RadT} and \eqref{eq:RadL1}, we used that $T(., \theta) = 1-\max_q \ell_q(., \theta)$ and $\ind \{T(., \theta) \leq t \}  = 1-\prod_{q=1}^Q \ind \{\ell_q(., \theta) < 1-t \}$ and the fact that the variables $\ell_q(X_i, \theta) $ are continuous by Assumption~\ref{Tcontinuity}. Similarly, we have $U(., \theta) = \sum_{q=1}^{Q} \ell_q(., \theta^*) \prod_{k \neq q}\ind \{ \ell_q(., \theta) \geq \ell_k(., \theta) \}$.
Hence, Lemma \ref{lem:Radproduct} once again entails that
\begin{align}
	\esp \Radm_n (\{U(., \theta), \theta \in \Theta\}) &\leq 
	\sum_{q=1}^Q \sum_{k=1, k\neq q}^Q \esp \Radm_n (\{ \ind\{\ell_q(., \theta) - \ell_k(., \theta) \geq 0  \}, \theta \in \Theta\}) \label{eq:RadU} \nonumber \\
	&+  \sum_{q=1}^Q \esp \Radm_n (\{ \ell_q(., \theta), \theta \in \Theta \}). 
\end{align}

%

To bound both $\esp \Radm_n (\{\ell_q(., \theta), \theta \in \Theta\})$ and $\esp \Radm_n (\{ \ind \{\ell_q(., \theta) < 1-t \}, \theta \in \Theta \})$, we use the results of \cite{baraud2016bounding} (more specifically the proof of Theorem 1 therein), to obtain that they are bounded by 
\begin{align*}
  \sqrt{\VCdim\log \frac{2en}{\VCdim}} \frac{\sqrt{2}}{\sqrt{n}} + 4\VCdim\log \frac{2en}{\VCdim} \frac{1}{n}\leq 7  \sqrt{\VCdim \frac{\log n }{n} } , 
 \end{align*}
 provided that 
 $\VCdim (\log n) /n\leq 1$ and for $n\geq (2e)^3$.
 Similarly, $\esp \Radm_n (\{ \ind\{\ell_q(., \theta) - \ell_k(., \theta) \geq 0  \}, \theta \in \Theta\})$ is bounded by 
 \begin{align*}
  \sqrt{\VCdim_{-}\log \frac{2en}{\VCdim_{-}}} \frac{\sqrt{2}}{\sqrt{n}} + 4\VCdim_{-}\log \frac{2en}{\VCdim_{-}} \frac{1}{n}\leq 7  \sqrt{\VCdim_- \frac{\log n }{n} }, \end{align*}
 $\VCdim_- (\log n) /n\leq 1$ and for $n\geq (2e)^3$.
Combining this with what is above entails
\begin{align}
	\esp \Radm_n (\mathscr{F}_{L_1})
	&\leq 7 Q  \sqrt{\VCdim \frac{\log n }{n} } \nonumber\\ 
	\esp \Radm_n (\mathscr{F}_{L_0}) 
	&\leq 14 Q  \sqrt{\VCdim \frac{\log n }{n} } \nonumber\\ 
	\esp \Radm_n (\mathscr{F}_{M_0}) 
	&\leq 	14 Q  \sqrt{\VCdim \frac{\log n }{n} } +7 Q^2  \sqrt{\VCdim_- \frac{\log n }{n} } \nonumber.
\end{align}
In particular, all expectations are upper-bounded by $c \sqrt{\frac{\log n}{n}}$, which leads to the result.
\end{proof}

\begin{lemma} \label{lem:Radproduct} 
 If $\mathcal{F}$ is a class of indicator functions and $\mathcal{G}$ is a class of functions from $\mathbb{R}^d$ to $[0,1]$, we have
 \begin{align*}
 \esp \Radm_n (\mathcal{F}\cdot \mathcal{G})&\leq \esp \Radm_n (\mathcal{F})+\esp \Radm_n (\mathcal{G})\\
\esp \Radm_n (\max(\mathcal{F}, \mathcal{G}))&\leq \esp \Radm_n (\mathcal{F})+\esp \Radm_n (\mathcal{G}),
 \end{align*}
 where we denoted $\mathcal{F}\cdot \mathcal{G}=\{fg, f\in \mathcal{F},g\in \mathcal{G}\}$ and $\max(\mathcal{F}, \mathcal{G})=\{f\vee g,f\in \mathcal{F},g\in \mathcal{G} \}$.
 \end{lemma}
 
 \begin{proof}
 We have
 \begin{align*}
  \esp \Radm_n (\mathcal{F}\cdot \mathcal{G})&= \esp\left(\sup_{f\in \mathcal{F},g\in \mathcal{G}} \left|\sum_{i=1}^n \varepsilon_i f.g(X_i)\right|\right) \\
 &\leq \esp\left(\sup_{f\in \mathcal{F},g\in \mathcal{G}} \left|\sum_{i=1}^n \varepsilon_i (f(X_i) + g(X_i)) \right|\right) \\
&\leq  \Radm_n (\mathcal{F}+\mathcal{G}) ,
 \end{align*}
 because $fg=(f+g-1)_+=0.5 (f+g-1+|f+g-1|)$ and by applying the contraction lemma of Talagrand (see e.g. Lemma 5.7 in \cite{FoundationsMLbook12}) with $x \mapsto 0.5(x-1 + \vert x-1 \vert )$ which is 1-Lipchitz.
Then we conclude by using the triangular inequality.
For the max we use $\max(f,g)=0.5(f+g+|f-g|)$. 
\end{proof}

\begin{lemma} \label{lem:VCdim_exp} 
	Consider the case where $Q=2$ and $\{F_u, u \in \mathcal{U}\}$ is an exponential family, i.e. there exists some functions $A, B, C, D$ such that $f(x, u) = \exp \left( A(u)^t B(x) - C(u) + D(x) \right)$. Let $k$ be the dimension of the sufficient statistic vector $B(x)$. 
	If $k \geq 3$, then $\VCdim, \VCdim_{-}$ defined by \eqref{VCdim}, \eqref{VCdimminus} satisfy $\VCdim, \VCdim_{-} \leq Q k(k+1) \left[3 \log (k(k+1)) +2 (Q-1)\right]$. 
	In addition, this bound still holds for $\VCdim_{-}$ in the case $Q \geq 3$. 
\end{lemma}
\begin{proof} Let us first bound $\VCdim$. Given that, for $Q=2$, $\theta=(\pi_1,\pi_2,\phi_1,\phi_2)$, $\ell_1(x, \theta) \geq t$ is equivalent to $ \pi_1f(x, \phi_1) / \pi_2 f(x, \phi_2) \geq g(t)$ for some function $g$, we get that $\ell_1(x, \theta) \geq t$ if and only if $ a(\theta)^t B(x) - b(\theta) \geq h(t)$ for some functions $a,b,h$. 
The set family is a subset of $\{ \{x \in \mathbb{R}^d, a^t B(x) + b \geq 0\}, a \in \mathbb{R}^k, b \in \mathbb{R}  \}$, whose VC dimension is bounded by $k(k+1) \left[3 \log (k(k+1)) +2 \right]$ for $k\geq 3$, see Lemma 10.3 in \cite{UnderstandingMLbook14}. 
By symmetry, this bound also holds for the VC dimension of $\{ \ell_2(\cdot, \theta), \theta \in \Theta \}$. It follows that $\VCdim \leq Q k(k+1) \left[3 \log (k(k+1)) +2 \right]+2(Q-1)$ (see, e.g., Exercice 3.24 in \cite{FoundationsMLbook12} on the VC dimension of the union of two classes with bounded VC dimension). 

For $\VCdim_{-}$, we have that for any $q\neq q' \in \{1, \dots, Q\}$,  $\ell_q(x, \theta) - \ell_{q'}(x, \theta) \geq 0$ is equivalent to $\pi_qf(x, \phi_q) / \pi_{q'} f(x, \phi_{q'})  \geq 1$. The rest of the proof follows similarly as for $\VCdim$. 
\end{proof}

\begin{lemma}[Talagrand's inequality, Theorem 5.3. in \cite{massart07}]\label{lem:massart}
Let $\xi_1, \dots, \xi_n$ independent r.v., $\mathscr{F}$ a countable class of measurable functions s.t. $a \leq f \leq b$ for every $f \in \mathscr{F}$ for some real numbers $a\leq b$, and $W = \sup_{f \in \mathscr{F}} | \sum_{i=1}^n f(\xi_i) - \esp(f(\xi_i)) |$. Then, for any $x>0$, 
\begin{align*}
\mathbb{P}( W - \esp(W) \geq x ) \leq e^{-\frac{2x^2}{n(b-a)^2}}. 
\end{align*}
\end{lemma}

\begin{lemma}[Rademacher complexity bound, see, e.g., Lemma~1 in \cite{baraud2016bounding}]\label{lem:rademacher}
In the setting of Lemma~\ref{lem:massart} (and with the notation therein), we have 
\begin{align*}
\esp(W) \leq  2\Radm_n (\mathscr{F}) ,
\end{align*}
where 
\begin{equation}\label{radcomplex}
 \Radm_n (\mathscr{F}) =\sup_{f\in \mathcal{F}} \left|\sum_{i=1}^n \varepsilon_i f(\xi_i)\right|
 \end{equation}
  is the Rademacher complexity of the class $\mathscr{F}$ (with $\varepsilon_1,\dots, \varepsilon_n$ being i.i.d. random signs).
\end{lemma}

\section{Auxiliary results for the Gaussian case}\label{FMR:sec:gauss}

 \subsection{Convergence rate for parameter estimation}
 

The following result presents two situations where the parameter of a Gaussian mixture model can be consistently estimated, with an explicit rate.

\begin{proposition}\label{prop:etagauss} 
Consider the mixture model (Section~\ref{FMR:sec:model}) in the $d$-multivariate Gaussian case with true parameter $\theta^* = (\pi^*, \phi^*)$, where  $\phi^*_q=(\mu^*_q,\Sigma^*_q) $, $1\leq q\leq Q$. Then $\eta(\epsilon,\theta^*)$ defined by \eqref{equeta} is such that $\eta( \epsilon_n, \theta^*) \leq 1/n$ for $\epsilon_n \geq C\sqrt{ \log n / n }$, where $C>0$ is a sufficiently large constant, in two following situations: 
\begin{itemize}
	\item[(i)]  $\hat{\theta}$ is the constrained MLE, that is, computed for $\phi_q=(\mu_q,\Sigma_q) \in \mathcal{U}$ with constrained parameter space $\mathcal{U} = [-a_n, a_n]^d \times \{\Sigma \in S^{++}_d, \underline{\lambda} \leq \lambda_1(\Sigma) \leq \lambda_d(\Sigma) \leq \bar{\lambda} \}$\footnote{Here, $\lambda_1(\Sigma)$ (resp. $\lambda_d(\Sigma)$) denotes the smallest (resp. largest) eigenvalue of $\Sigma$.} where $a_n \leq L (\log n)^\gamma$ for some $L, \gamma >0$ and $S^{++}_d$ denotes the space of positive definite matrices, with $\underline{\lambda}, \bar{\lambda}>0$. In that case, 
	$C$ only depends on $\theta^*$ and $L, \gamma,\underline{\lambda}, \bar{\lambda}$. 
		\item[(ii)] $\hat{\theta}$ is the estimator coming from EM algorithm (when the iteration number is infinite) for an initialization $\mu_1^{(0)}, \mu_2^{(0)}$ such that $\|(\mu_1^{(0)} - \mu_2^{(0)}) -(\mu_1-\mu_2)\|\leq  \Delta/4$, where $\Delta =\| \mu_1 -\mu_2 \|_2 $ is the separation between the true means. Here, we consider an homoscedastic model with $\Sigma_1 = \Sigma_2 = \Sigma = \nu I_d$ with known $\nu$. The conclusion applies if the signal-to-noise ratio $\Delta / \nu$ is large enough, and for a constant $C$ of the form $c(\nu,\Delta)\sqrt{d}$. 
\end{itemize}

\end{proposition} 

\begin{proof}
Since case (ii) is a direct application of \cite{balakrishnan17}, we focus in what follows on proving case (i), by revisiting the result of \cite{ho16}. 
First, in the considered model, any mixture can be defined in terms of $\{ f_{u}, u\in \mathcal{U} \} $ and a discrete mixing measure $G = \sum_{q=1}^Q \pi_q \delta_{\phi_q}$ with $Q$ support points, as $ \sum_{q=1}^Q \pi_q f_{\phi_q} = \int f_u(x) dG(u)$.
As shown by \cite{ho16}, the convergence of mixture model parameters can be measured in terms of a Wasserstein distance on the space of mixing measures. Let $G_1=\sum_{q=1}^Q \pi^1_q \delta_{\phi^1_q}$ and $G_2=\sum_{q=1}^Q \pi^2_q \delta_{\phi^2_q}$ be two discrete probability measures on some parameter space, which is equipped with metric $\| . \|$.
The Wasserstein distance of order $1$ between $G_1$ and $G_2$ is given by
\begin{align*}
W_1(G_1,G_2) = \inf_p \sum_{q,l} p_{q,l} \| \phi^1_q - \phi^2_l \|  
\end{align*}
where the infimum is over all couplings $(p_{q,l})_{1 \leq q,l \leq Q} \in [0,1]^{Q \times Q}$ such that $\sum_l p_{q,l} = \pi^1_q$ and $\sum_q p_{q,l} = \pi^2_l$.
Let $G^*, \hat{G}_n$ denote the true mixing measure and the mixing measure that corresponds to the restricted MLE considered here, respectively. 
Theorem 4.2. in \cite{ho16} implies that, with the notation of \cite{ho16}, for any $\epsilon_n \geq (\sqrt{C_1}/c) \delta_n,$ and $\delta_n \leq  C\sqrt{ \log n / n }$, we have $\Pstar( W_1(\hat{G}_n, G^*) \geq (c/C_1) \epsilon_n) \leq c e^{-n\epsilon_n^2}$. 
We apply this relation for $\epsilon_n = \max((\sqrt{C_1}/c) \delta_n, \sqrt{ \log (cn)/ n })$. In that case, we have still 
$\epsilon_n$ of order $\sqrt{ \log n / n }$ and the upper-bound is at most $1/n$.  
On the other hand, if we have a convergence rate in terms of $W_1$, then we have convergence of the mixture model parameters in terms of $\| . \|$ at the same rate, see Lemma \ref{lem:wasserstein}. This concludes the proof. 
\end{proof}

\begin{lemma} \label{lem:wasserstein} 
Let $G_n= \sum_{q=1}^Q \pi^n_q \delta_{\phi^n_q}$ be a sequence of discrete probability measures on $\mathcal{U}$, and let $G^*, W_1$ be defined as in the proof of Proposition \ref{prop:etagauss}. There exists a constant C only depending on $G^*$ such that if $W_1(G_n, G^*) \rightarrow 0$, then for sufficiently large $n$, 
\begin{align*}
W_1(G_n, G^*) \geq C \min_{\sigma \in [Q]} \| \theta^\sigma_n - \theta^* \|. 
\end{align*}
\end{lemma}
\begin{proof}

In what follows, we let $\{p^n_{q,l} \}$ denote the corresponding probabilities of the optimal coupling for the pair $(G_n, G^*)$. We start by showing that $(\phi^n_q)_q \rightarrow (\phi^*_q)_q$ in $\| . \|$ up to a permutation of the labels. Let $\sigma^n$ the permutation of the labels such that $\| \phi^n_q - \phi^*_l \|  \geq \| \phi^n_{\sigma^n(l)} - \phi^*_l \|$ for all $q, l \in \{1, ..., Q\}$. Then, by definition,
\begin{align*}
W_1(G_n, G^*) &\geq \sum_{1 \leq q,l \leq Q} p_{q,l}^n \| \phi^n_{\sigma(l)} - \phi^*_l \|  \\
&= \sum_l \pi^*_l \| \phi^n_{\sigma^n(l)} - \phi^*_l \|.
\end{align*}
It follows that each $\| \phi^n_{\sigma^n(l)} - \phi^*_l \| $ must converge to zero. 
Since $(\phi^n_q)_q \rightarrow (\phi^*_q)_q$ up to a permutation of the labels, without loss of generality we can assume that $\phi^n_q \rightarrow \phi^*_q$ for all $q$. Let $\Delta \phi^n_q := \phi^n_q - \phi^*_q$ and $\Delta \pi^n_q := \pi^n_q - \pi^*_q$. Write $W_1(G_n, G^*)$ as 
\begin{align*}
W_1(G_n, G^*) = \sum_q p^n_{qq}  \| \Delta \phi^n_q \|  + \sum_{q \neq l} p^n_{ql} \| \phi^n_q - \phi^*_l \| 
\end{align*}

Define $C_{ql}=  \| \phi^*_q - \phi^*_l \|$ and $C= \min_{q \neq l} C_{ql} > 0$. It follows from the convergence of $\phi^n$ that 
for $q \neq l$, $\| \phi^n_q - \phi^*_l \| \geq C/2$ for sufficiently large $n$. Thus, 
\begin{align*}
W_1(G_n, G^*) \geq \frac{C}{2} \sum_{q \neq l} p^n_{ql}
\end{align*}
We deduce that $\sum_{q \neq l} p^n_{ql} \rightarrow 0$. As a result, $p^n_{qq} = \pi^*_q- \sum_{l \neq q} p^n_{lq} \rightarrow \pi^*_q$, and so, $p^n_{qq} \geq (1/2) \pi^*_{\min} := \min_l \pi^*_l$ for sufficiently large $n$. 
On the other hand, $\sum_{q \neq l} p^n_{ql} = \sum_q \pi^n_q- p^n_{qq} = \sum_q \pi^*_q - p^n_{qq}$ where $p^n_{qq} \leq \min(\pi^n_q, \pi^*_q)$. Thus, $\sum_{q \neq l} p^n_{ql} \geq \sum_q \pi^n_q - \min(\pi^n_q, \pi^*_q) = \sum_{q, \pi^n_q \geq \pi^*_q} \pi^n_q -\pi^*_q = \sum_{q, \pi^n_q \geq \pi^*_q} |\pi^n_q -\pi^*_q|$ and similarly we have that $\sum_{q \neq l} p^n_{ql} \geq \sum_{q, \pi^*_q \geq \pi^n_q} |\pi^n_q -\pi^*_q|$. It follows that $2\sum_{q \neq l} p^n_{ql} \geq \sum_q |\pi^n_q - \pi^*_q|$. 
Therefore, for sufficiently large $n$,
\begin{align*}
W_1(G_n, G^*) \geq \frac{1}{2} \pi^*_{\min} \sum_q  \| \Delta \phi^n_q \|  + \frac{C}{4} \sum_q |\Delta \pi^n_q|. 
\end{align*}
This gives the result.
\end{proof}

  \subsection{Gaussian computations}

The following lemma holds.
 
\begin{lemma}\label{lem:exass}
 Let us consider the multivariate Gaussian case where $\phi_q=(\mu_q,\Sigma_q)$, $1\leq q \leq Q$, with $Q=2$, $\Sigma_1=\Sigma_2$ is an invertible covariance matrix and $\mu_1$ and $\mu_2$ are two different vectors of $\mathbb{R}^d$. Then  Assumptions~\ref{Tcontinuity},~\ref{densitycontinuity} and~\ref{lip} hold true for $\alpha_c=0$ and for a level $\alpha\in (0,\bar\alpha)\backslash\mathcal{E}$ for $\mathcal{E}$ a set of Lebesgue measure $0$.

\end{lemma}

\begin{proof}
Let us first prove that $\ell_q(X, \theta)$ is a continuous random variable under $\Pstar$ (this is established below without assuming $\Sigma_1=\Sigma_2$ for the sake of generality). We have
\begin{align*}
  &\Pstar \left( \ell_1(X, \theta) = t \right)\\
  &= \Pstar \left(  f_{\phi_1}(X) /f_{\phi_{2}}(X)= t  \pi_{2} / \pi_1   \right)\\
  &=\Pstar \left(  (X-\mu_1)^t \Sigma_1^{-1} (X-\mu_1) -  (X-\mu_2)^t \Sigma_2^{-1} (X-\mu_2) = -2\log\left(t  \pi_{2} / \pi_1 \right) - \log(|\Sigma_1|/|\Sigma_2|) \right).
  \end{align*}
 Now, 
 \begin{align*}
  &(X-\mu_1)^t \Sigma_1^{-1} (X-\mu_1) -  (X-\mu_2)^t \Sigma_2^{-1} (X-\mu_2) 
 \\ 
 &=(X-\mu_1)^t \Sigma_1^{-1} (X-\mu_1) -  (X-\mu_1)^t \Sigma_2^{-1} (X-\mu_2) - (\mu_1-\mu_2)^t \Sigma_2^{-1} (X-\mu_2) \\
 &=(X-\mu_1)^t (\Sigma_1^{-1}- \Sigma_2^{-1})(X-\mu_1) -  (X-\mu_1)^t \Sigma_2^{-1} (\mu_1-\mu_2) - (\mu_1-\mu_2)^t \Sigma_2^{-1} (X-\mu_2)\\
 &=(X-\mu_1)^t (\Sigma_1^{-1}- \Sigma_2^{-1})(X-\mu_1)  - (\mu_1-\mu_2)^t \Sigma_2^{-1} (2X-\mu_2-\mu_1).
   \end{align*}
Since the real matrix  $\Sigma_1^{-1}- \Sigma_2^{-1}$ is symmetric, we can diagonalize it and we end up with a  subset of $\mathbb{R}^d$ of the form
$$
\left\{y\in \mathbb{R}^d\::\: \sum_{j=1}^d \left(\alpha_j y_j^2 + \beta_j y_j\right) + \gamma =0\right\},
$$
for some real parameters $\alpha_j,\beta_j,\gamma$. The result follows because this set has a Lebesgue measure equal to $0$ in any case. 

Now, since $\Sigma_1=\Sigma_2=\Sigma$, we have for all $t\in (0,1)$,
\begin{align*}
\left\{ T(X, \theta) > t \right\} &=  \left\{\forall q\in \{1,\dots,Q\},  \pi_q f_{\phi_q}(X) < (1-t) \sum_{\ell=1}^Q \pi_{\ell} f_{\phi_{\ell}}(X) \right\}\\
&=\left\{ \pi_1 f_{\phi_1}(X) < (1/t-1) \pi_{2} f_{\phi_{2}}(X) \right\}\cap \left\{ \pi_2 f_{\phi_2}(X) < (1/t-1) \pi_{1} f_{\phi_{1}}(X) \right\}\\
&=\left\{ (1/t-1)^{-1}<\frac{\pi_1 f_{\phi_1}(X) }{\pi_2 f_{\phi_{2}}(X)} <  (1/t-1)  \right\}.
\end{align*}
Applying $2\log(\cdot)$ on each part of the relation, we obtain
\begin{align*}
\left\{ T(X, \theta) > t \right\} &=\left\{- 2 \log(1/t-1)<  a^t X+b  <   2\log (1/t-1)  \right\},
\end{align*}
for 
\begin{align*}
a=a(\theta)&= 2  \Sigma^{-1} (\mu_1-\mu_2)\in \mathbb{R}^d\backslash\{0\}\\
b=b(\theta)&=- (\mu_1-\mu_2)^t \Sigma^{-1} (\mu_1+\mu_2)+ 2 \log (\pi_1/\pi_2)\in \mathbb{R}^d.
\end{align*}
Since under $P_{\theta^*}$ we have $X\sim \pi^*_1 \mathcal{N}(\mu_1^*,\Sigma^*) + \pi^*_2 \mathcal{N}(\mu_2^*,\Sigma^*)$, we have $a^t X+b \sim \pi^*_1 \mathcal{N}(a^t\mu_1^*+b,a^t\Sigma^*a) + \pi^*_2 \mathcal{N}(a^t\mu_2^*+b,a^t\Sigma^*a)$. This yields for all $t\in (0,1)$,
\begin{align}
\P_{\theta^*}( T(X, \theta) > t ) =&\pi_1 \left[\Phi\left(\frac{2\log (1/t-1)-a^t\mu_1^*-b}{(a^t\Sigma^*a)^{1/2}}\right) \right.\nonumber \\
&\left. -\Phi\left(\frac{-2\log (1/t-1)-a^t\mu_1^*-b}{(a^t\Sigma^*a)^{1/2}}\right) \right]\nonumber\\
&+\pi_2 \left[\Phi\left(\frac{2\log (1/t-1)-a^t\mu_2^*-b}{(a^t\Sigma^*a)^{1/2}}\right) \right. \nonumber\\
&\left. -\Phi\left(\frac{-2\log (1/t-1)-a^t\mu_2^*-b}{(a^t\Sigma^*a)^{1/2}}\right) \right].\label{relationkeyassumpgauss}
\end{align}
A direct consequence is that for all $t\in (0,1)$, we have $\P_{\theta^*}( T(X, \theta) > t )<1$, that is, $\P_{\theta^*}( T(X, \theta) \leq t )=\P_{\theta^*}( T(X, \theta) < t )>0$. Hence, $\alpha_c$ defined in \eqref{alphac} is equal to zero.  
Moreover, from \eqref{relationkeyassumpgauss}, we clearly have that $t\in (0,1) \mapsto \P_{\theta^*}( T(X, \theta) > t )$ is decreasing, so that $t\in (0,1) \mapsto \P_{\theta^*}( T(X, \theta) \leq t )$ is increasing. 
This proves that Assumption~\ref{Tcontinuity} holds in that case. \\

Let us now check Assumptions~\ref{densitycontinuity} and~\ref{lip}.  Assumptions~\ref{densitycontinuity} and~\ref{lip} (i) follow from  Result 2.1  in \cite{melnykov2013distribution}. \\


As for Assumption~\ref{lip} (ii), from \eqref{relationkeyassumpgauss}, we only have to show that the function
$t\in (0,1)\mapsto \frac{\partial }{\partial t}\Phi\left(\frac{\log (1/t-1)-\alpha^*}{\beta^*}\right)$ is uniformly bounded by some constant $C=C(\alpha^*,\beta^*)$, for any $\alpha^*\in \mathbb{R}$ and $\beta^*>0$. A straightforward calculation leads to the following: for all $t\in(0,1)$,
\begin{align}
\left|\frac{\partial }{\partial t}\Phi\left(\frac{\log (1/t-1)-\alpha^*}{\beta^*}\right) \right|=  \frac{e^{-(\frac{\log (1/t-1)-\alpha^*}{\beta^*})^2/2}}{\beta^*\sqrt{2\pi}} \frac{1}{t(1-t)}.
\label{relationkeyassumpgauss2}
\end{align}
Consider now $t_0=t_0(\alpha^*,\beta^*)\in (0, 1/2)$ such that $(\frac{\log (1/t-1)-\alpha^*}{\beta^*})^2\geq 2\log(1/t)$ for all $t\in (0,t_0)$. It is clear that the right-hand-side of \eqref{relationkeyassumpgauss2} is upper-bounded by $\frac{1}{\beta^*\sqrt{2\pi} (1-t_0)}$ on $t\in (0,t_0)$.
Similarly, let $t_1=t_1(\alpha^*,\beta^*)\in (1/2,1)$ such that $(\frac{\log (1/t-1)-\alpha^*}{\beta^*})^2\geq 2\log(1/(1-t))$ for all $t\in (t_1,1)$. It is clear that the right-hand-side of \eqref{relationkeyassumpgauss2} is upper-bounded by $\frac{1}{\beta^*\sqrt{2\pi} t_1}$ on $t\in (t_1,1)$.  Finally, for $t\in [t_0,t_1]$, the upper-bound $\frac{1}{\beta^*\sqrt{2\pi} t_0(1-t_1)}$ is valid. This proves that Assumption~\ref{lip} (ii) holds.\\

Let us now finally turn to Assumption~\ref{lip} (iii). Lemma~\ref{mFMR_properties} ensures that $t\in (0,t^*(\bar\alpha))\mapsto \mFMR^*_t$ is continuous increasing. Hence, $t^*:\beta\in (0,\bar \alpha)\mapsto t^*(\beta)$ defined in \eqref{tstaralpha} is the inverse of this function and is also continuous increasing. It is therefore differentiable almost everywhere in $(0,\bar \alpha)$, so everywhere in $(0,\bar\alpha)\backslash\mathcal{E}$ where $\mathcal{E}$ is a set of Lebesgue measure $0$. By taking $\alpha$ in $(0,\bar\alpha)\backslash\mathcal{E}$, this ensures  that $t^*$ is differentiable in $\alpha$ and thus that Assumption~\ref{lip} (iii) holds.
\end{proof}

\begin{lemma} \label{lem:VCdim_gaussian} 
	In the multivariate gaussian case with $Q=2$ and $\Sigma_1=\Sigma_2$, we have that $\VCdim \leq 2d+4$ and $\VCdim_{-} \leq 2d+4$.
\end{lemma}

\begin{proof} In that case, we have that (see the proof of Lemma \ref{lem:exass})
	\begin{align*}
		\{\ell_q(x, \theta) \leq u, x \in \mathbb{R}^d \} 
		&= \{a_\theta^t x + b_\theta \geq g(u), x \in \mathbb{R}^d \}.
	\end{align*}
Since the VC dimension of the vector space of real-valued affine functions  is bounded by $d+1$ (see, e.g., Exercice 3.19 in \cite{FoundationsMLbook12}). We obtain the result by applying the usual bound on the VC dimension of the union of two classes with bounded VC dimension (see, e.g., Exercice 3.24 in \cite{FoundationsMLbook12}).  
\end{proof}

\end{document}